\documentclass{amsart}
\usepackage{amsmath}
\usepackage{amsfonts}
\usepackage{amssymb}
\usepackage{epsfig}
\newtheorem{theorem}{Theorem}[section]

\newtheorem{lm}[theorem]{Lemma}

\newtheorem{cor}[theorem]{Corollary}

\newtheorem{rem}[theorem]{Remark}
\newtheorem{pr}[theorem]{Proposition}
\newtheorem{example}[theorem]{Example}
\newtheorem{hyp}[theorem]{Hypothesis}

\usepackage{color}

\begin{document}
	
	\title{Quasi-reductive supergroups with small even parts}
	\author{}
	\address{}
	\email{}
	\author{A. N. Zubkov}
	\address{Sobolev Institute of Mathematics, Omsk Branch, Pevtzova 13, 644043 Omsk, Russian Federation}
	\email{a.zubkov@yahoo.com}
	\begin{abstract}
		We describe all supergroups with the largest even supersubgroups being isomorphic to $\mathrm{GL}_2, \mathrm{SL}_2$ or $\mathrm{PSL}_2$. These results are applied to
		the description of centralizers of certain tori in the quasi-reductive supergroups.
	\end{abstract}
	\maketitle
	
	\section*{Introduction}
	
	The purpose of this article is to describe (affine and algebraic) supergroups $\mathbb{G}$ whose the largest even (super)subgroups $\mathbb{G}_{ev}$ are isomorphic to one of the following groups :
	$\mathrm{GL}_2, \mathrm{SL}_2$ or $\mathrm{PSL}_2$. This is motivated by the following observation.
	
	Let $\mathbb{G}$ be a quasi-reductive supergroup (over an algebraically closed field $\Bbbk$ of zero or odd characteristic) and let $S$ be a (not necessary maximal) torus in the reductive group $G\simeq\mathbb{G}_{ev}$. Combining \cite[Theorem 6.6]{mas-shib} and \cite[Corollary 17.59]{milne}, we obtain that $\mathrm{Cent}_{\mathbb{G}}(S)$ is a quasi-reductive supergroup as well. 
	
	Furthermore, if $T$ is a maximal torus in $G$ and $\Delta$ is the root system of $\mathbb{G}$, which is associated with the adjoint action of $T$ on the Lie superalgebra $\mathfrak{G}=\mathrm{Lie}(\mathbb{G})$, then for any $\alpha\in\Delta$ we set $T_{\alpha}=\ker(\alpha)_t$. Note that $T_{\alpha}$ is a subtorus of $T$ of codimension at most one, 
	such that $G_{\alpha}=\mathrm{Cent}_G(T_{\alpha})$ is either $T$ or a (split) reductive group of semisimple rank one  (cf. \cite[Theorem 21.11]{milne}). In the latter case
	$G_{\alpha}$ is isomorphic to one of the following groups : 
	\[\mathrm{GL}_2\times T', \ \mathrm{SL}_2\times T', \mathrm{PSL}_2\times T',\]
	where $T'$ is a certain subtorus of $T$. Further, the subtorus $T'$ is central in $\mathbb{G}_{\alpha}=\mathrm{Cent}_{\mathbb{G}}(T_{\alpha})$ and 
	$\mathbb{G}_{\alpha}/T'$ has the even part being isomorphic to $\mathrm{GL}_2, \mathrm{SL}_2$ or $\mathrm{PSL}_2$. 
	Therefore, the problem of describing supergroups with the even part isomorphic to $\mathrm{GL}_2, \mathrm{SL}_2$ or $\mathrm{PSL}_2$ naturally arises.
	
	Our approach to solving this problem is based on the Harish-Chandra pair method (see \cite{mas-shib, maszub}). More precisely, any supergroup $\mathbb{G}$ is uniquely (up to an isomorphism) determined by the associated Harish-Chandra pair $(G, \mathfrak{G}_{\bar 1})$, where $\mathfrak{G}_{\bar 1}$ is regarded as a $G$-module with respect to the adjoint action. Thus all we need is to describe (up to an isomorphism as well) all Harish-Chandra pairs $(G, V)$, such that
	$G$ is isomorphic to $\mathrm{GL}_2, \mathrm{SL}_2$ or $\mathrm{PSL}_2$, and then translate it to the category of supergroups.
	
	As might be expected, the case $G\simeq\mathrm{GL}_2$ turned out to be the most non-trivial. However, its complete solution relies heavely on the case $G\simeq\mathrm{SL}_2$.
	
	Note that if $\mathbb{G}$ is quasi-reductive, then its \emph{unipotent radical} $\mathrm{R}_u(\mathbb{G})$ is \emph{purely odd}, that is $(\mathrm{R}_u(\mathbb{G}))_{ev}=e$.
	The latter implies that $\mathrm{R}_u(\mathbb{G})$ is a direct product of one-dimensional purely odd unipotent supergroups, hence abelian (see the next section for more detail). 
	Since $(\mathbb{G}/\mathrm{R}_u(\mathbb{G}))_{ev}\simeq G$, without loss of generality, one can assume that $\mathrm{R}_u(\mathbb{G})=e$, i.e. $\mathbb{G}$ is a \emph{reductive} supergroup. 
	
	So, let $\mathbb{G}$ be a quasi-reductive and reductive supergroup with $G\simeq\mathrm{SL}_2$. By Theorem \ref{or ... } one of the following statements holds :
	
	1. $\mathbb{G}\simeq\mathrm{SpO}(2|1)$.
	
	2. $\mathrm{char}\Bbbk=3$ and the $\mathrm{SL}_2$-module $\mathfrak{G}_{\bar 1}$ is isomorphic to the \emph{Weyl module} $V(3)$. Moreover, the corresponding Harish-Chandra pair structure, or equivalently, the Lie bracket on $\mathfrak{G}_{\bar 1}$, is unique (up to a scalar multiple). 
	The supergroup, associated with this Harish-Chandra pair, is denoted by $\mathbb{H}(3/1)$, and it was first introduced in \cite{zub-bar}. In the terminology of \cite{zub-bar}, $\mathbb{H}(3/1)$ is an almost-simple supergroup (in the strong sense), as well as $\mathrm{SpO}(2|1)$ is. Moreover, $\mathbb{H}(3/1)\simeq\mathrm{Aut}(\mathfrak{spo}(2|1))^0$.
	
	3. $\mathfrak{G}_{\bar 1}$ is isomorphic to $L(0)\oplus L(2)$, where $L(2)$ can be identified with $\mathfrak{sl}_2$ (equipped with the adjoint action of $\mathrm{SL}_2$), so that
	\[ [L(0), L(0)]=[L(2), L(2)]=0 \ \mbox{and} \ [b, x]=bx, b\in\Bbbk\simeq L(0), x\in\mathrm{sl}_2.\]
	The supergroup, associated with this Harish-Chandra pair, is denoted by $\mathbb{H}(0\oplus 2)$. One can show that its \emph{derived supersubgroup} $\mathcal{D}(\mathbb{H}(0\oplus 2))$ is isomorphic to a semi-direct product of $\mathrm{SL}_2$ and $\mathrm{R}_u(\mathcal{D}(\mathbb{H}(0\oplus 2)))$. 
	Besides,  $\mathbb{H}(0\oplus 2)/\mathcal{D}(\mathbb{H}(0\oplus 2))$ is the one-dimensional purely odd unipotent supergroup.
	
	The case $G\simeq\mathrm{PSL}_2$ can be easily reduced to the previous one, so that $\mathbb{G}\simeq\mathbb{H}(0\oplus 2)/\mu_2$. 
	
	Finally, if $G\simeq\mathrm{GL}_2$, then all the Harish-Chandra pairs are divided into several families in accordance with what $[\mathfrak{G}_{\bar 1}, \mathfrak{G}_{\bar 1}]$ is equal to. We will present here the most non-trivial cases, referring to Section 5.1 in the rest. 
	
	Let $Z$ denote the center of $\mathrm{GL}_2$. Then $\mathfrak{G}_{\bar 1}=\oplus_{k\in\mathbb{Z}}(\mathfrak{G}_{\bar 1})_k$, where \[(\mathfrak{G}_{\bar 1})_k=\{x\in \mathfrak{G}_{\bar 1}\mid gx=\det(g)^{\frac{k}{2}}x, \ \mbox{for any} \ g\in Z \}.\]
	Further, $[(\mathfrak{G}_{\bar 1})_k, (\mathfrak{G}_{\bar 1})_t]\neq 0$ implies $k+t=0$ and for any integer $k$, $[(\mathfrak{G}_{\bar 1})_k, (\mathfrak{G}_{\bar 1})_{-k}]$ is equal to $\mathfrak{z}=\mathrm{Lie}(Z)=\Bbbk I_2, \mathfrak{sl}_2$ or $\mathfrak{gl}_2$. 
	
	We discuss in detail the cases $\mathfrak{G}_{\bar 1}=(\mathfrak{G}_{\bar 1})_k\oplus (\mathfrak{G}_{\bar 1})_{-k}$ and $[(\mathfrak{G}_{\bar 1})_k, (\mathfrak{G}_{\bar 1})_{-k}]]=\mathfrak{sl}_2$ or $[(\mathfrak{G}_{\bar 1})_k, (\mathfrak{G}_{\bar 1})_{-k}]=\mathfrak{gl}_2$ only. All other cases are easily reduced to these ones. 
	
	If $[(\mathfrak{G}_{\bar 1})_k, (\mathfrak{G}_{\bar 1})_{-k}]]=\mathfrak{sl}_2$, then $(\mathfrak{G}_{\bar 1})_{-k}\simeq L((t, t))\simeq \Bbbk\otimes\det^{-t}$ and
	$(\mathfrak{G}_{\bar 1})_k\simeq L((t+1, t-1))\simeq L(1, -1)\otimes\det^t\simeq\mathrm{sl}_2\otimes\det^t, t\in\mathbb{Z}$. We also have
	\[[b\otimes\mathrm{det}^{-t}, x\otimes\mathrm{det}^t]=bx, b\in\Bbbk, x\in\mathfrak{sl}_2.\]
	Let $[(\mathfrak{G}_{\bar 1})_k, (\mathfrak{G}_{\bar 1})_{-k}]]=\mathfrak{gl}_2$. If $k=0$, then we have two different cases.
	
	a) $\mathfrak{G}_{\bar 1}=R\oplus V$, where $[R, V]=0$, $V$ is self-dual as a $\mathrm{GL}_2$-module, $R\simeq\mathfrak{gl}_2$, regarded as the adjoint $\mathrm{GL}_2$-module, and the Harish-Chandra pair structure on $R$ and $V$ are given by
	\[[\alpha I_2+x, \beta I_2+y]=(\alpha\beta a+c\mathrm{tr}(xy)) I_2+ (\beta x+\alpha y), \ x, y\in\mathfrak{sl}_2, \ \alpha, \beta\in\Bbbk,\]
	where $a, c$ are scalars, which are not equal to zero simultaneously, and \[[v, v']=<v, v'>I_2, v, v'\in V,\]
	where the pairing $< \ , \ >$ is induced by an isomorphism $V\simeq V^*$. Besides, $V$ is trivial as the $\mathfrak{gl}_2$-module. 
	
	It is worth noting that the subpair $(Z, V)$ determines the \emph{solvable radical} $\mathrm{R}(\mathbb{G})$ of $\mathbb{G}$. 
	The supersubgroup, associated with this Harish-Chandra subpair $(\mathrm{GL}_2, R)$, is denoted by ${\bf Q}(2; a, c)$. One can show that ${\bf Q}(2; a, c)\simeq {\bf Q}(2; a', c')$ if and only if there is
	$\alpha\in\Bbbk\setminus 0$, such that $a=\alpha a', c=\alpha^{-1}c'$. Besides,  ${\bf Q}(2; 1, \frac{1}{2})$ is isomorphic to the queer supergroup ${\bf Q}(2)$. So, the supergroups
	${\bf Q}(2; a, c)$ can be regarded as \emph{deformations} of ${\bf Q}(2)$. The whole supergroup $\mathbb{G}$ is a product of this supersubgroup ${\bf Q}(2; a, c)$ and its solvable radical.
	
	b) $\mathfrak{G}_{\bar 1}=R\oplus L$, where $R\simeq\mathfrak{sl}_2$, regarded as the adjoint $\mathrm{GL}_2$-module, and $L$ are self-dual and the Lie bracket on them is defined as above. Further, $L$ contains a submodule $V$, such that $[R, V]=0, L/V\simeq L((0, 0))\simeq\Bbbk$, and for any $r\in R, l\in L$, there is $[r, l]=[r, \bar l]=\bar l r$, where $\bar l=l+V$ is identified with a scalar in $\Bbbk$. Besides, $V$ is trivial as the $\mathfrak{gl}_2$-module.

	Finally, the case $k\neq 0$ splits into four :
	
	$1'$. $p| k=2t$. $(\mathfrak{G}_{\bar 1})_k\simeq L((t, t))\oplus L((t+1, t-1)), (\mathfrak{G}_{\bar 1})_{-k}\simeq L((-t, -t))\simeq L((t, t))^*$ and in the above notations, we have
	\[[a\otimes\mathrm{det}^t, b\otimes\mathrm{det}^{-t}]=ab I_2, \ [x\otimes\mathrm{det}^t, b\otimes\mathrm{det}^{-t}]=bx, \ a, b\in\Bbbk, x, y\in\mathfrak{sl}_2.\]
	Symmetrically,  $(\mathfrak{G}_{\bar 1})_k\simeq L((t, t))\oplus L((t+1, t-1)), (\mathfrak{G}_{\bar 1})_{-k}\simeq L((-t+1, -t-1))\simeq L((t+1, t-1))^*$ with the obvious modification of the Lie bracket.
	
	To formulate $2'$ and $3'$ we need some additional notations. Let $N_{\pm k}$ denote the $\mathrm{GL}_2$-submodules $\{x\in M_{\pm k}\mid [x, M_{\mp k}]\subseteq\mathfrak{z} \}$.
	Then $M_{k}/N_{k}$ and $M_{-k}/N_{-k}$ are mutually non-isomorphic modules from the list $L((t+1, t-1)), L((-t, -t))$. For the sake of certainty, we assume that
	$M_k/N_k\simeq L((t+1, t-1)), M_{-k}/N_{-k}\simeq L((-t, -t))$. Besides, in $(2')$ and $(3')$ we still assume that $p|k$.
	
	$2'$. $\mathfrak{G}_{\bar 1}\simeq (L((t+1, t-1))\oplus N_k)\oplus M_{-k}$, where $[L((t+1, t-1)), N_{-k}]=0$ and the Lie bracket on
	$L((t+1, t-1))\times M_{-k}/N_{-k}\to\mathfrak{sl}_2$ is determined as in the point $1'$. Besides, the Lie bracket on $N_k\times M_{-k}\to\mathfrak{z}$ is a non-degenerate pairing and both $N_k, M_{-k}$ are trivial $\mathfrak{gl}_2$-modules.  
	
	$3'$. $\mathfrak{G}_{\bar 1}\simeq (L((t+1, t-1))\oplus N_k)\oplus (L((-t, -t))\oplus N_{-k})$, where 
	\[[L((t+1, t-1)), N_{-k}]=[L((-t, -t)), N_k]=0, \]
	the Lie bracket on $L((t+1, t-1))\times L((-t, -t))$ is determined as in the point $1'$, and it induces a non-degenerate pairing $N_k\times N_{-k}\to\mathfrak{z}$, where both $N_k, N_{-k}$ are trivial $\mathfrak{gl}_2$-modules.
	
	$4'$. $p\nmid k=2t-1$, $(\mathfrak{G}_{\bar 1})_k\simeq L((t, t-1))\simeq V\otimes\det^{t-1}, (\mathfrak{G}_{\bar 1})_{-k}\simeq L((t, t-1))^*\simeq L((1-t, -t))\simeq V\otimes\det^{-t}$,
	where $V\simeq L((1, 0))$ is the standard two-dimensional $\mathrm{GL}_2$-module. Let $v_1, v_{-1}$ be a basis of $V$, such that for any diagonal matrix 
	\[g=\left(\begin{array}{cc}
		g_1 & 0 \\
		0 & g_{-1}
	\end{array}\right)\in\mathrm{GL}_2\]
	we have $gv_{\pm 1}=g_{\pm 1}v_{\pm 1}$.  Then
	\[ [v_1\otimes\mathrm{det}^{t-1}, v_1\otimes\mathrm{det}^{-t}]=E_{1, -1}, \ [v_1\otimes\mathrm{det}^{t-1}, v_{-1}\otimes\mathrm{det}^{-t}]=-\frac{1}{2}(H-\frac{1}{k}I_2), \]
	\[ [v_{-1}\otimes\mathrm{det}^{t-1}, v_1\otimes\mathrm{det}^{-t}]=-\frac{1}{2}(H+\frac{1}{k}I_2), \ [v_{-1}\otimes\mathrm{det}^{t-1}, v_{-1}\otimes\mathrm{det}^{-t}]=-E_{-1, 1},\]
	where $H=E_{11}-E_{-1, -1}$ and $E_{ij}v_k=\delta_{jk}v_i, i, j, k\in\{\pm 1\}$.
	
	Let $\mathbb{H}(t)$ denote the supergroup, which corresponds to the Harish-Chandra pair from the point $4'$. Respectively, let $\mathfrak{H}(t)$ denote its Lie superalgebra. Note that $\mathbb{H}(1)\simeq\mathrm{SL}(2|1)$ and by \cite[Lemma 2.2]{zub-bar} each supergroup $\mathbb{H}(t)$ is almost-simple (in the weak sense).
	
	Next, it is easy to see that
	$\mathbb{H}(t)\simeq\mathbb{H}(1-t)$. One can show even more, $\mathbb{H}(t)$ is isomorphic to $\mathbb{H}(t')$ if and only if $t=t'$ or $t+t'=1$. What is more curious is that all Lie superalgebras $\mathfrak{H}(t)$ are isomorphic to $\mathfrak{H}(1)\simeq\mathfrak{sl}(2|1)$!
	
	Similarly, let $\mathbb{S}(t)$ and $\mathbb{L}(t)$ denote the supergroups corresponding to the non-isomorphic Harish-Chandra pairs from the point $1'$. Compairing (super)dimensions, one can easily see that for any integers $t, t'$ we have $\mathbb{S}(t)\not\simeq\mathbb{L}(t')$, $\mathbb{S}(t)\not\simeq\mathbb{H}(t')$ and $\mathbb{L}(t)\not\simeq\mathbb{H}(t')$. 
	We also prove that $\mathbb{S}(t)\simeq\mathbb{S}(t')$ or $\mathbb{L}(t)\simeq\mathbb{L}(t')$ if and only if $t=\pm t'$.
	
	The problem to describe the isomorphism classes of supergroups from point $2'$ and $3'$ is much more complicated and we leave it for the future work. 
	
	The article is organized as follows. In the first section we collect all auxiliary facts, results and notations concerning the representation theory of reductive groups, quasi-reductive supergroups and Harish-Chandra pairs, associated with them. In the second section we discuss the structure of centralizers of certain tori in quasi-reductive supergroups and provide some examples to illustrate it.  
	
	The third section is devoted to the description of supergroups with the even part being isomorphic to $\mathrm{SL}_2$. Using these results we describe supergroups with even parts being isomorphic to $\mathrm{PSL}_2\simeq\mathrm{PGL}_2$ or $\mathrm{GL}_2$, in the fourth and fifth sections respectively. Unfortunately, the results of the fifth section are difficult, if not impossible, to formulate in the form of a single theorem. This is why we selected the most interesting cases and discussed them above.
	
	In the sixth section we apply these results to describe the structure of the centralizers $\mathrm{Cent}_{\mathbb{G}}(T_{\alpha})$ and $\mathrm{Cent}_{\mathbb{G}}(\mathbb{T}_{\alpha})$ in a more detailed way.  Here $\mathbb{T}=\mathrm{Cent}_{\mathbb{G}}(T)$ is so-called \emph{Cartan supersubgroup}. We prove that if $\mathbb{G}$ is quasi-reductive and reductive, then the Cartan supersubgroup $\mathbb{T}$ is reductive also.
	
	\section{Auxiliary results}
	
	Throughout this article $\Bbbk$ is an algebraically closed field of arbitrary characteristic, unless stated otherwise.
	
	Recall some standard facts from representation theory of reductive groups (see also \cite{jan} for more details). 
	
	Let $G$ be a (smooth and connected) reductive group over $\Bbbk$ and $T$ its maximal torus. The elements of the character group $X(T)$ are called \emph{weights}, whenever they appear in a \emph{weight decomposition} $\oplus_{\lambda\in X(T)} W^{\lambda}$ of a $T$-module $W$, where $W^{\lambda}=\{w\in W\mid tw=\lambda(t)w, t\in T\}$.

	Let $\mathfrak{g}$ denote the Lie algebra of $G$. We have a \emph{root decomposition} 
	\[\mathfrak{g}=\mathfrak{t}\oplus (\oplus_{\alpha\in \Delta}\mathfrak{g}^{\alpha} )\]
	with respect to the adjoint action of $T$, where $\mathfrak{t}=\mathrm{Lie}(T)=\mathfrak{g}^0$ and $\Delta$ is the root system of $G$, determined by the choice of $T$. The choice of subset of positive roots $\Delta^+$ corresponds to the choice of \emph{positive} Borel subgroup $B^+$. 	The \emph{opposite} Borel subgroup $B^-$ corresponds to the subset of negative roots $\Delta^-=-\Delta^+$. The choice of $\Delta^+$ determines on $X(T)$ a partial order $\leq$ also, such that $\mu\leq\lambda$ if $\lambda-\mu$ is a sum of positive roots.
	
	Let $B$ be a Borel subgroup of $G$, such that $T\leq B$. Then $B=T\ltimes U$, where $U$ is the \emph{unipotent radical} of $B$. Recall that each $\lambda\in X(T)$ determines the one-dimensional representation $\Bbbk_{\lambda}$ of $B$, such that $U$ acts trivially on $\Bbbk_{\lambda}$.
	
	Set $H^0(\lambda)=\mathrm{ind}^G_{B^-}\Bbbk_{\lambda}$. A weight $\lambda$ is called \emph{dominant}, if $H^0(\lambda)\neq 0$. If $\lambda$ is dominant, then
	the socle of $H^0(\lambda)$ is an irreducible $G$-module, which is denoted by $L(\lambda)$. Conversely, any irreducible $G$-module is isomorphic to some $L(\lambda)$. In addition,
	$H^0(\lambda)^{\lambda}$ is a $B^+$-stable line and for any $\mu\neq\lambda$, $H^0(\lambda)^{\mu}\neq 0$ implies $\mu <\lambda$. 
	
	Let $W$ be a $G$-module and $V$ be a one-dimensional $B^+$-submodule of $W$, or equivalently, $V\subseteq W^{U^+}$. Then each $v\in V\setminus 0$ is called \emph{primitive}. For example, any nonzero element from 
	$H^0(\lambda)^{\lambda}$ is primitive. Set $V(\lambda)=H(-w_0(\lambda))^*$, where $w_0$ is the longest element of the \emph{Weyl group} $W(G, T)=\mathrm{N}_G(T)/T$ of $G$ (also determined by the choice of $T$). The module $V(\lambda)$ is said to be  the \emph{Weyl module}, or the \emph{universal highest weight module} of weight $\lambda$. Note that the top of $V(\lambda)$ is isomorphic to $L(\lambda)$. Moreover, for any $\mu\neq\lambda$, $V(\lambda)^{\mu}\neq 0$ implies $\mu<\lambda$. In particular, $V(\lambda)$ is generated by a primitive element of the weight $\lambda$. 
	The modules $H^0(\lambda)$ and $V(\lambda)$ have the same formal characters, but not necessary isomorphic. 
	
	The following lemma is a folklore.
	\begin{lm}\label{Weyl module is mapped onto}
		Let  $V$ be a finite dimensional $G$-module, such that its top is isomorphic to an irreducible $G$-module $L(\lambda)$ and any other composition factor $L(\mu)$ of $V$ satisfies $\mu<\lambda$.
		Then $V$ is an epimorphic image of the Weyl module $V(\lambda)$.  Symmetrically, if the socle of $V$ is isomorphic to $L(\lambda)$ and any other composition factor $L(\mu)$ of $V$  satisfies $\mu<\lambda$, then $V$ is embedded into $H^0(\lambda)$.		
	\end{lm}
	\begin{proof}
		The category of finite dimensional $G$-modules has the anti-equivalence $M\mapsto \llap{}^{\tau}M$, such that $\llap{}^{\tau} L(\lambda)\simeq L(\lambda)$ and
		$\llap{}^{\tau}H^0(\lambda)\simeq V(\lambda)$ (cf. \cite[II.2.12(2) and II.2.13(2)]{jan}).  In particular, the second statement is equivalent to the first one.
		Choose a nonzero vector $v\in V$ of weight $\lambda$. Since all weights of $V$ are less or equal to $\lambda$, the vector $v$ is \emph{primitive} (or a \emph{highest weight}) vector. Moreover, $v$ generates $V$ modulo its radical , hence $v$ generates the whole $V$ as well. In other words, $V$ is a \emph{highest weight module} of highest weight $\lambda$ and \cite[Lemma II.2.13(b)]{jan} concludes the proof.	
	\end{proof}
	For any integer $r\geq 1$ let $G_r$ denote $r$-th \emph{Frobenius kernel} (cf. \cite[I.9.4]{jan}), which is a finite normal infinitesimal subgroup of $G$.  
	\begin{lm}\label{g-invariants}
		Let $V$ be a $G$-module. If $\mathrm{char}\Bbbk=p>0$ and $V$ is a trivial $\mathfrak{g}$-module, then each composition factor of $V$ has a form $L(\lambda)^{[1]}\simeq L(p\lambda)$.
	\end{lm}
	\begin{proof}
		The statement follows by \cite[Proposition II.3.15]{jan}, \cite[Corollary II.3.17]{jan} (Steinberg's tensor product theorem) and the natural equivalence between the category of $G_1$-modules and the category of $U^{[p]}(\mathfrak{g})$-modules (cf. \cite[I.9.6]{jan}).   	
	\end{proof}
	\begin{rem}\label{converse of Lemma 1.2}
		The converse takes place if $\mathfrak{g}=\mathfrak{t}+\mathfrak{h}$, where $\mathfrak{t}=\mathrm{Lie}(T)$ and $\mathfrak{h}=[\mathfrak{h}, \mathfrak{h}]$.	Indeed, if each composition factor of $V$ has a form $L(\lambda)^{[1]}$, then the image of the induced Lie algebra morphism $\mathfrak{g}\to\mathfrak{gl}(V)$ is a nilpotent subalgebra. In particular, the image of $\mathfrak{h}$ is zero and it remains to note that $V$ is a semi-simple $\mathfrak{t}$-module, hence trivial as well.
	\end{rem}
	Let $\mathsf{SAlg}_{\Bbbk}$ denote the category of (associative and unital) super-commutative superalgebras. 
	An \emph{affine supergroup} $\mathbb{G}$ is a representable functor from $\mathsf{SAlg}_{\Bbbk}$ to the category of groups. In other words,  there is a Hopf superalgebra, denoted by
	$\Bbbk[\mathbb{G}]$, such that
	\[\mathbb{G}(A)=\mathrm{Hom}_{\mathsf{SAlg}_{\Bbbk}}(\Bbbk[\mathbb{G}], A), \ A\in\mathsf{SAlg}_{\Bbbk}.\]
	Closed supersubgroups of $\mathbb{G}$ are in one-to-one correspondence with the Hopf superideals of $\Bbbk[\mathbb{G}]$. More precisely, a group subfunctor $\mathbb{H}\leq\mathbb{G}$ is a closed supersubgroup if and only if there is a Hopf superideal $I$ of $\Bbbk[\mathbb{G}]$, such that
	\[\mathbb{H}(A)=\{g\in\mathbb{G}(A)\mid g(I)=0 \},  A\in\mathsf{SAlg}_{\Bbbk}. \] 
	In particular, $\Bbbk[\mathbb{H}]\simeq \Bbbk[\mathbb{G}]/I$. For example, $\Bbbk[\mathbb{G}]\Bbbk[\mathbb{G}]_{\bar 1}$ is a Hopf superideal, which determines the largest \emph{purely even} supersubgroup $\mathbb{G}_{ev}$. Note that for any superalgebra $A$ we have $\mathbb{G}_{ev}(A)=\mathbb{G}_{ev}(A_0)=\mathbb{G}(A_0)$, hence $\mathbb{G}_{ev}$ can be regarded as an affine group $G$, represented by the Hopf algebra $\Bbbk[G]=\Bbbk[\mathbb{G}]/\Bbbk[\mathbb{G}]\Bbbk[\mathbb{G}]_{\bar 1}$. 
	
	A supersubgroup $\mathbb{H}$ of $\mathbb{G}$ is called \emph{normal}, and it is denoted by $\mathbb{H}\unlhd\mathbb{G}$, if $\mathbb{H}(A)\unlhd\mathbb{G}(A)$ for arbitrary $A\in\mathsf{SAlg}_{\Bbbk}$. 
	
	An affine supergroup $\mathbb{G}$ is called \emph{algebraic}, provided $\Bbbk[\mathbb{G}]$ is finitely generated. From now on all supergroups are supposed to be affine and algebraic, unless stated otherwise. The category of (left) rational $\mathbb{G}$-supermodules is naturally identified with the category of (right) $\Bbbk[\mathbb{G}]$-supercomodules.
	
	Let $\mathfrak{G}$ denotes the Lie superalgebra of $\mathbb{G}$. Then $\mathfrak{G}$ is a $\mathbb{G}$-supermodule with respect to the adjoint action. 
	
	Let $G$ be an algebraic group and $V$ be a finite dimensional $G$-module.   
	The pair $(G, V)$ is said to be a \emph{Harish-Chandra pair}, if the following conditions hold:
	\begin{enumerate}
		\item There is a symmetric bilinear map $V\times V\to\mathrm{Lie}(G)$, denoted by $[ \ , \ ]$;
		\item This map is $G$-equivariant with respect to the diagonal action of $G$ on $V\times V$ and the adjoint action of $G$ on $\mathrm{Lie}(G)$;
		\item The induced action of $\mathrm{Lie}(G)$ on $V$, denoted by the same symbol $[ \ , \ ]$,  satisfies $[[v, v], v]=0$ for all $v\in V$. 
	\end{enumerate}
	The morphism of Harish-Candra pairs $(G, V)\to (H, W)$ is a couple $f : G\to H$ and $u : V\to W$, where $f$ is a morphism of algebraic groups and $u$ is a morphism of vector spaces, such that 
	\begin{enumerate}
		\item $u (gv)= f(g)u(v), \text{ for all } g\in G, \text{ and } v\in V$;
		\item $[u(v), u(v')]=\mathrm{d}_e(f)([v, v']), \text{ for all } v, v'\in V$. 
	\end{enumerate}  
	With each algebraic supergroup $\mathbb{G}$ we associate its Harish-Chandra pair $(G, \mathfrak{G}_{\bar 1})$, where $G=\mathbb{G}_{ev}$ acts on $\mathfrak{G}$ via the adjoint action
	and the bilinear map $\mathfrak{G}_{\bar 1}\times \mathfrak{G}_{\bar 1}\to\mathrm{Lie}(G)=\mathfrak{G}_{\bar 0}$ is the restriction of the Lie bracket on $\mathfrak{G}_{\bar 1}$. 
	The functor $\mathbb{G}\mapsto (G, \mathfrak{G}_{\bar 1})$ is an equivalence of the category of algebraic supergroups to the category of Harish-Chandra pairs (see \cite[Theorem 5.4 and Theorem 6.1]{mas-shib} or \cite[Theorem 12.10]{maszub}). 
	
	Furthermore, if a (closed) supersubgroup $\mathbb{H}$ of  $\mathbb{G}$ is represented by its Harish-Chandra (sub)pair $(H, \mathfrak{H}_{\bar 1})$, then  
	$\mathbb{H}$ is normal in $\mathbb{G}$ if and only if
	\begin{enumerate}
		\item $H\unlhd G$;
		\item $\mathfrak{H}_{\bar 1}$ is a $G$-submodule of $\mathfrak{G}_{\bar 1}$;
		\item $H\leq\ker(G\to\mathrm{GL}(\mathfrak{G}_{\bar 1}/\mathfrak{H}_{\bar 1}))$; 
		\item $[\mathfrak{G}_{\bar 1}, \mathfrak{H}_{\bar 1}]\subseteq\mathrm{Lie}(H)$.
	\end{enumerate}
	We call the conditions $(1)-(4)$ a \emph{normality criterion}. 
	The Harish-Chandra pair of the algebraic supergroup $\mathbb{G}/\mathbb{H}$ is naturally isomorphic to $(G/H, \mathfrak{G}_{\bar 1}/\mathfrak{H}_{\bar 1})$.
	
	Recall that a supergroup $\mathbb{H}$ is said to be \emph{unipotent}, if for any $\mathbb{H}$-supermodule $V$ there is $V^{\mathbb{H}}\neq 0$. 
	By \cite[Proposition 2.1]{maszub1}, $\mathbb{H}$ is unipotent if and only if $\mathbb{H}_{ev}=H$ is.
	
	Let $\mathbb{G}$ be a smooth connected supergroup. Using \cite[Lemma 1.1]{zub-bar}, one can easily show that $\mathbb{G}$ contains the largest normal smooth connected unipotent supersubgroup, that is called an \emph{unipotent radical} of $\mathbb{G}$ and denoted by $\mathrm{R}_u(\mathbb{G})$. We say that $\mathbb{G}$ is \emph{reductive}, if $\mathrm{R}_u(\mathbb{G})=e$.
	
	Similarly, one can define a \emph{solvable radical} of $\mathbb{G}$ as the largest normal smooth connected solvable supersubgroup (see \cite[Lemma 1.2]{zub-bar}). We denote it by 
	$\mathrm{R}(\mathbb{G})$. We obviously have $\mathrm{R}_u(\mathbb{G})\leq \mathrm{R}_u(\mathrm{R}(\mathbb{G}))$.
	
	A supergroup $\mathbb{G}$ is called \emph{purely odd}, if $G=e$. Such supergroup is isomorphic to the direct product of several copies of the one-dimensional purely odd (unipotent) supergroup $\mathbb{G}_a^-$. As a group fuctor $\mathbb{G}_a$ is defined by $\mathbb{G}_a(R)=(R_{\bar 1}, +)$.
	
	By the normality criterion, a purely odd supersubgroup $\mathbb{H}\leq\mathbb{G}$ is normal if and only if $\mathfrak{H}=\mathfrak{H}_{\bar 1}$ is a $G$-submodule of $\mathfrak{G}_{\bar 1}$, such that $[\mathfrak{G}_{\bar 1}, \mathfrak{H}]=0$. 
	
	Recall that $\mathbb{G}$ is said to be \emph{quasi-reductive}, if its largest even (super)subgroup $\mathbb{G}_{ev}=G$ is a reductive group. Note that if $\mathbb{G}$ is quasi-reductive, then $\mathrm{R}_u(\mathbb{G})$ is purely odd. Moreover, the Harish-Chandra pair of $\mathrm{R}_u(\mathbb{G})$ has a form $(e, \mathfrak{R})$, where $\mathfrak{R}$ is the largest $G$-submodule in $\mathfrak{G}_{\bar 1}$, such that $[\mathfrak{G}_{\bar 1}, \mathfrak{R}]=0$. 
	
	If the natural embedding $\mathbb{G}_{ev}\to\mathbb{G}$ is split, then $\mathbb{G}$ is said to be a \emph{split supergroup} as well. The latter takes place if and only if $[\mathfrak{G}_{\bar 1}, \mathfrak{G}_{\bar 1}]=0$ if and only if $\mathbb{G}\simeq\mathbb{G}_{ev}\ltimes \mathbb{G}_{odd}$, where $\mathbb{G}_{odd}\simeq (\mathbb{G}_a^-)^{\dim\mathfrak{G}_{\bar 1}}\leq\mathrm{R}_u(\mathbb{G})$.
	
	It is worth of noting that the property of being quasi-reductive does not imply the property of being reductive and vice versa. In fact, if $G$ is a reductive group and $V$ is a nonzero $G$-module, then the zero map $V\times V\to\mathrm{Lie}(G)$ defines a split supergroup $\mathbb{G}\simeq G\ltimes\mathrm{R}_u(\mathbb{G})$, where $\mathrm{R}_u(\mathbb{G})\simeq (\mathbb{G}_a^{-})^{\dim V}$. Conversely, a reductive, but not quasi-reductive, supergroup has been constructed in \cite[Section 8]{zubgrish}.
	
	If $V$ and $W$ are subspaces of $\mathfrak{G}_{\bar 1}$, then $[V, W]$ denotes the subspace of $\mathfrak{G}_{\bar 0}$ generated by all elements $[v, w], v\in V, w\in W$.
	\begin{lm}\label{may be useful}
		Let $\mathfrak{G}$ be a Lie superalgebra, such that $\mathfrak{G}_{\bar 0}$ is a simple Lie algebra. If $V_1$ and $V_2$ are $\mathfrak{G}_{\bar 0}$-submodules of $\mathfrak{G}_{\bar 1}$ and $V_1\cap V_2=0$, then $[V_1, V_1]\neq 0$ implies $V_2$ is a trivial $\mathfrak{G}_{\bar 0}$-module and $[V_2, V_2]=0$. 	
	\end{lm}
	\begin{proof}
		For arbitrary $x, y\in V_1, z\in V_2$ we have
		\[[[x, z], y]+[[z, y], x]=-[[x, y], z]\in V_1\cap V_2=0 .\]
		It remains to note that $[V_1, V_1]$ is a nonzero ideal in $\mathfrak{G}_{\bar 0}$, hence $[V_1, V_1]=\mathfrak{G}_{\bar 0}$ and therefore, $[x, v]=0$ for any $x\in\mathfrak{G}_{\bar 0}, v\in V_2$. Finally, if $V_2$ is a trivial $\mathfrak{G}_{\bar 0}$-module, then
		$[V_2, V_2]$ is a central ideal in $\mathfrak{G}_{\bar 0}$, that implies $[V_2, V_2]=0$.	
	\end{proof}	
	\begin{lm}\label{may be useful too}
		Let $\mathfrak{G}$ be a Lie superalgebra, such that $\mathfrak{G}_{\bar 0}$ is a simple Lie algebra. If $V_1$ and $V_2$ are $\mathfrak{G}_{\bar 0}$-submodules of $\mathfrak{G}_{\bar 1}$, such that $[V_1, V_2]=0$ and $[V_1, V_1]\neq 0$, then $V_2$ is a trivial $\mathfrak{G}_{\bar 0}$-module. 		
	\end{lm}
	\begin{proof}
		As above, $[V_1, V_1]=\mathfrak{G}_{\bar 0}$. Thus
		\[[\mathfrak{G}_{\bar 0}, V_2]=[[V_1, V_1], V_2]=[V_1, [V_1, V_2]]=0.\]		
	\end{proof}	
	
	\section{Centralizers of tori. A preliminary look at their description.}
	
	From now on we assume that $\mathrm{char}\Bbbk\neq 2$, unless stated otherwise.	Let $\mathbb{G}$ be a quasi-reductive supergroup. Let $T$ be a maximal torus in the reductive group $\mathbb{G}_{ev}=G$. Then \cite[Theorem 6.6]{mas-shib} and \cite[Proposition 17.61]{milne} imply that
	the Harish-Chandra pair of $\mathrm{Cent}_{\mathbb{G}}(T)$ is $(\mathrm{Cent}_G(T), \mathfrak{G}_{\bar 1}^T)=(T, \mathfrak{G}_{\bar 1}^T)$. The supergroup $\mathbb{T}=\mathrm{Cent}_{\mathbb{G}}(T)$ is called
	the \emph{Cartan supersubgroup} or \emph{super-torus} of $\mathbb{G}$ (cf. \cite[Lemma 3.5]{tshib}). 
	
	We have $\mathfrak{G}=\oplus_{\alpha\in\Delta}\mathfrak{G}^{\alpha} \oplus\mathfrak{G}^0$ a \emph{root decomposition} of $\mathfrak{G}$ with respect to the adjoint action of $T$,
	where $\Delta$ consists of nonzero roots. It is clear that 
	\[\mathrm{Lie}(\mathbb{T})=\mathfrak{t}\oplus\mathfrak{G}_{\bar 1}^T =\mathfrak{G}^0 .\]
	Set $\Delta_{\bar i}=\{\alpha\in\Delta| \mathfrak{G}^{\alpha}_{\bar i}=\mathfrak{G}^{\alpha}\cap\mathfrak{G}_{\bar i}\neq 0 \}, i=0, 1$. Note that the set 
	$\Delta_{\bar 0}\cap\Delta_{\bar 1}$ is not necessary empty.
	
	For a root $\alpha\in\Delta$, set $T_{\alpha}=\ker(\alpha)_t$ (the largest subtorus of $\ker(\alpha)$). By \cite[Theorem 6.6(2)]{mas-shib} the Harish-Chandra pair of $\mathbb{G}_{\alpha}=\mathrm{Cent}_{\mathbb{G}}(T_{\alpha})$ has a form $(G_{\alpha}, \mathfrak{G}_{\bar 1}^{T_{\alpha}})$, where $G_{\alpha}=\mathrm{Cent}_G(T_{\alpha})$. 
	Recall that the description of centralizers $G_{\alpha}$ plays crucial role in the description of reductive groups (see \cite[Chapters 20-21]{milne}). So, we expect that the description of centralizers $\mathbb{G}_{\alpha}$ may be also useful for the future structural theory of quasi-reductive supergroups (especially, in positive characteristic).  
	\begin{lm}\label{structure of centralizer}
		The supergroup $\mathbb{G}_{\alpha}$ is quasi-reductive. Moreover, we have :
		\begin{enumerate}
			\item If $\alpha\in\Delta_{\bar 0}$, then 
			\[\mathrm{Lie}(G_{\alpha})=\mathfrak{t}\oplus\mathfrak{G}^{\alpha}_{\bar 0}\oplus\mathfrak{G}^{-\alpha}_{\bar 0}, \ \dim\mathfrak{G}_{\bar 0}^{\pm\alpha}=1, \ \mathfrak{G}_{\bar 1}^{T_{\alpha}}=(\oplus_{\beta\in\Delta_{\bar 1}\cap\mathbb{Q}\alpha}\mathfrak{G}^{\beta}_{\bar 1})\oplus \mathfrak{G}_{\bar 1}^T ; \]
			\item If $\alpha\in\Delta_{\bar 1}\setminus \Delta_{\bar 0}$, then
			\[\mathrm{Lie}(G_{\alpha})=\mathfrak{t}\oplus\mathfrak{G}^{\gamma}_{\bar 0}\oplus\mathfrak{G}^{-\gamma}_{\bar 0}, \ \dim\mathfrak{G}_{\bar 0}^{\pm\gamma}=1, \ \mathfrak{G}_{\bar 1}^{T_{\alpha}}=(\oplus_{\beta\in\Delta_{\bar 1}\cap\mathbb{Q}\alpha}\mathfrak{G}^{\beta}_{\bar 1})\oplus \mathfrak{G}_{\bar 1}^T, \]
			where $\gamma$ is the unique positive root from $\Delta_{\bar 0}\cap\mathbb{Q}\alpha$, otherwise
			\[G_{\alpha}=T, \ \mathfrak{G}_{\bar 1}^{T_{\alpha}}=(\oplus_{\beta\in\Delta_{\bar 1}\cap\mathbb{Q}\alpha}\mathfrak{G}^{\beta}_{\bar 1})\oplus \mathfrak{G}_{\bar 1}^T.\]
		\end{enumerate} 	
	\end{lm}
	\begin{proof}
		The second identities are obvious, since $\mathfrak{G}^{\beta}_{\bar 1}\subseteq \mathfrak{G}_{\bar 1}^{T_{\alpha}}$ if and only if $T_{\alpha}\leq\ker(\beta)$ if and only if $T_{\alpha}=T_{\beta}$. It remains to note that if $\alpha$ and $\beta$ are nonzero characters, then $T_{\alpha}=T_{\beta}$ if and only if there are nonzero integers $k$ and $l$, such that $k\alpha=l\beta$. 
		
		As it has been observed in \cite[Theorem 21.11]{milne}, $G_{\alpha}$ is a split reductive group with maximal torus $T$, which has the semisimple rank $0$ or $1$. This rank equals $1$ if and only if there is $\gamma\in\Delta_0$, such that $T_{\alpha}\leq\ker(\gamma)$. Moreover, if it takes place, then the root $\gamma$ is unique.  
		If such root does not exists, then the semisimple rank of $G_{\alpha}$ is equal to zero, hence $G_{\alpha}=T$.	
	\end{proof}
	Assume that $G_{\alpha}$ has semisimple rank one. Then by \cite[Theorem 20.33]{milne}, $G_{\alpha}$ is isomorphic to one of the following groups : 
	\[ \mathrm{SL}_2\times T', \ \mathrm{PGL}_2\times T', \ \mathrm{GL}_2\times T', \]
	where $T'$ is a subtorus of $T$. Note that in the first two cases, we have $T'=T_{\alpha}$. Indeed, in these cases the center of $G_{\alpha}$ is either $\mu_2\times T'$ or just $T'$, and it contains $T_{\alpha}$ also. Since $T'$ has codimension one in $T$, we obtain $T'=T_{\alpha}$. In the last case the center of $G_{\alpha}$ coincides with $\mathrm{Z}(\mathrm{GL}_2)\times T'\simeq G_m\times T'$, which is a subtorus of $T$ of codimension one. Thus $T_{\alpha}=G_m\times T'$. Note also that in all these cases
	$\ker(\alpha)$ is equal to either $T_{\alpha}$ or $\mu_2\times T_{\alpha}$. Since $\mathrm{char}\Bbbk\neq 2$, we have $T_{\alpha}=\ker(\alpha)^0$. 
	
	Let $\overline{G_{\alpha}}$ denote the first factor of $G_{\alpha}$, so that $G_{\alpha}=\overline{G_{\alpha}}\times T'$ and $\overline{G_{\alpha}}$ is isomorphic to one of the following groups : $\mathrm{SL}_2$, $\mathrm{PGL}_2\simeq\mathrm{PSL}_2$ or $\mathrm{GL}_2$. Observe also, that if $\overline{G}_{\alpha}\simeq \mathrm{GL}_2$, then its center acts trivially on $\mathfrak{G}_{\bar 1}^{T_{\alpha}}$.	
	
	We have $\mathrm{Lie}(G_{\alpha})=\mathrm{Lie}(\overline{G_{\alpha}})\oplus\mathfrak{t}'$ , where $\mathfrak{t}'=\mathrm{Lie}(T')$ and both direct terms are $G_{\alpha}$-submodules with respect to the adjoint action. The restriction of the Lie bracket on $\mathfrak{G}_{\bar 1}^{T_{\alpha}}$ can be recorded as
	\[[x, y]=[x, y]_1 +[x, y]_2, \ \mbox{where} \ x, y\in \mathfrak{G}_{\bar 1}^{T_{\alpha}}, \ [x, y]_1\in\mathrm{Lie}(\overline{G_{\alpha}}), \ [x, y]_2\in\mathfrak{t}' .\]
	Since $\mathfrak{t}'$ acts trivially on  $\mathfrak{G}_{\bar 1}^{T_{\alpha}}$, the identity $[[x, x], x]=0$ is equivalent to $[[x, x]_1, x]=0$, that is 
	$(\overline{G_{\alpha}}, \mathfrak{G}_{\bar 1}^{T_{\alpha}})$ has the structure of Harish-Chandra pair with respect to the symmetric bilinear map $(x, y)\mapsto [x, y]_1$.
	Note that the corresponding supergroup is isomorphic to $\mathbb{G}_{\alpha}/T'$. This , in addition to the pure theoretical interest, is another motivation for describing supergroups with the even part isomorphic
	to one of the groups : $\mathrm{SL}_2$, $\mathrm{PGL}_2\simeq\mathrm{PSL}_2$ or $\mathrm{GL}_2$. 
	
	In the third and fourth sections we completely investigate the structure of such supergroups and apply it to the above discussed centralizers of certain subtori. 
	Unfortunately, this does not give a complete description of centralizers $\mathbb{G}_{\alpha}$. One of the obstacles is that, in contrast to the purely even case, the exact sequence 
	\[e\to T'\to\mathbb{G}_{\alpha}\to\mathbb{G}_{\alpha}/T'\to e \]
	is not always split.     
	
	The examples below illustrate the content of this section.
	\begin{example}\label{Centralizer in GL}
		Let $\mathbb{G}$ be the general linear supergroup $\mathrm{GL}(m|n)$. For any (super-commutative) superalgebra $R$, the group $\mathbb{G}(R)$ consists of all matrices
		\[g=\left(\begin{array}{cc}
			A & B \\
			C & D	
		\end{array}\right), \ A\in\mathrm{GL}_m(R_{\bar 0}), \ D\in\mathrm{GL}_n(R_{\bar 0}), \]
		\[B\in\mathrm{Mat}_{m\times n}(R_{\bar 1}), \ C\in\mathrm{Mat}_{n\times m}(R_{\bar 1}). \]
		If it does not lead to confusion, we omit $R$ from our notations. For example, we are recording
		\[\mathrm{GL}(m|n)=\left(\begin{array}{cc}
			\mathrm{GL}_m & \mathrm{Mat}_{m\times n} \\
			\mathrm{Mat}_{n\times m} & \mathrm{GL}_n	
		\end{array}\right).\]
		Further, $G=\mathbb{G}_{ev}\simeq \mathrm{GL}_m\times \mathrm{GL}_n$ and we choose the maximal torus $T\leq G$ consisting of all diagonal matrices. The root system has a form 
		\[\Delta=\{\epsilon_i-\epsilon_j \mid 1\leq i\neq j\leq m+n\} \]
		with
		\[\Delta_{\bar 0}=\{\epsilon_i-\epsilon_j \mid 1\leq i\neq j\leq m \ \mbox{or} \ m+1\leq i\neq j\leq m+n \} \]
		and
		\[\Delta_{\bar 1}=\{\epsilon_i-\epsilon_j \mid 1\leq i\leq m< j\leq m+n \ \mbox{or} \ 1\leq j\leq m< i\leq m+n\}. \]
		Note that $\Delta_{\bar 0}\cap\Delta_{\bar 1}=\emptyset$.
		
		Let $\alpha=\epsilon_i-\epsilon_j\in\Delta_{\bar 0}$. Since $\mathbb{Q}\alpha\cap\Delta_{\bar 1}=\emptyset$ and $\mathfrak{G}_{\bar 1}^T=0$, Lemma \ref{structure of centralizer}(1) implies $\mathbb{G}_{\alpha}=G_{\alpha}\simeq\mathrm{GL}_2\times T'$. The first factor, that is isomorphic to  $\mathrm{GL}_2$, consists of matrices $(g_{kl})_{1\leq k, l\leq m+n}$, such that $g_{kl}\neq 0$ if and only if either $k, l\in\{ i, j\}$ or $k=l\neq i, j$, and in the latter case $g_{kk}=1$. Besides, $T'$ is a subtorus of $T$ consisting of matrices $g\in T$, such that $g_{ii}=g_{jj}=1$.
		
		Let $\alpha=\epsilon_i-\epsilon_j\in\Delta_{\bar 1}$. Since $\mathbb{Q}\alpha\cap\Delta_{\bar 0}=\emptyset$ and $\mathbb{Q}\alpha\cap\Delta_{\bar 1}=\{\pm\alpha\}$, Lemma \ref{structure of centralizer}(2) implies $G_{\alpha}=T$ and $\mathfrak{G}_{\bar 1}^{T_{\alpha}}=\mathfrak{G}_{\bar 1}^{\alpha}\oplus\mathfrak{G}_{\bar 1}^{-\alpha}$. Thus 
		$\mathbb{G}_{\alpha}\simeq \mathrm{GL}(1|1)\times T'$, where the factors $\mathrm{GL}(1|1)$ and $T'$ are described similarly to the previous case.   
	\end{example}
	\begin{example}\label{centralizer in SL}
		Let $\mathbb{G}$ be the special linear supergroup $\mathrm{SL}(m|n)$, where $m, n\geq 1$. Recall that \[\mathrm{SL}(m|n)=\{g\in \mathrm{GL}(m|n)\mid\mathrm{Ber}(g)=1\},\] where
		\[\mathrm{Ber}(g)=\det(A-BD^{-1}C)\det(D)^{-1}, \ g=\left(\begin{array}{cc}
			A & B \\
			C & D	
		\end{array}\right), \ A\in\mathrm{GL}_m, \ D\in\mathrm{GL}_m, \]
		\[B\in\mathrm{Mat}_{m\times n}, \ C\in\mathrm{Mat}_{n\times m}. \]  
		In particular, $G=\mathbb{G}_{ev}\simeq \ker\chi$, where \[\chi :  \mathrm{GL}_m\times \mathrm{GL}_n\to G_m, \chi(A\times D)=\det(A)\det(D)^{-1}.\]
		The maximal torus $T$ of $G$ consists of all diagonal matrices, whose diagonal entries satisfy $\prod_{1\leq k\leq m}t_{kk}=\prod_{m+1\leq k\leq m+n}t_{kk}$. 
		
		The character group $X(T)$ is isomorphic to $(\sum_{1\leq k\leq m+n}\mathbb{Z}\epsilon_k)/\mathbb{Z}\omega$, where $\omega=\sum_{1\leq k\leq m}\epsilon_k-\sum_{m+1\leq k\leq m+n}\epsilon_k$. The root system of $\mathbb{G}$ coincides with the image of the natural map $\Delta\to X(T)$, where $\Delta=\{\epsilon_i-\epsilon_j \mid 1\leq i\neq j\leq m+n\}$ is the root system from the previous example. 
		
		It is clear that $\mathfrak{G}_{\bar 1}^T\neq 0$ if and only if 
		$\Delta_{\bar 1}\cap\mathbb{Z}\omega\neq\emptyset$ if and only if $m=n=1$. If the latter takes place, then $\Delta_{\bar 0}=\emptyset, \Delta_{\bar 1}=\{\pm\omega \}$ and $T=T_{\omega}$ 
		implies $\mathbb{G}_{\omega}=\mathrm{Cent}_{\mathbb{G}}(T)=\mathbb{G}$.
		Until the end of this fragment we assume that $(m, n)\neq (1, 1)$.
		
		Let $\alpha=\epsilon_i-\epsilon_j\in\Delta_{\bar 0}$, where $1\leq i\neq j\leq m$ (the case $m+1\leq i\neq j\leq m+n$ is similar). One easily sees that there is no odd root $\beta$, such that $\beta\in\mathbb{Q}\alpha+\mathbb{Q}\omega$. Again, Lemma \ref{structure of centralizer}(1)
		implies $\mathbb{G}_{\alpha}=G_{\alpha}\simeq\mathrm{GL}_2\times T'$. The first factor, that is isomorphic to $\mathrm{GL}_2$, consists of matrices $(g_{kl})_{1\leq k, l\leq m+n}$, such that $g_{kl}\neq 0$ if and only if either $k, l\in\{ i, j\}$ or $k=l\neq i, j$, and in the latter case $g_{kk}=1$, except $g_{m+n, m+n}=g_{ii}g_{jj}-g_{ij}g_{ji}$. As above, $T'$ is a subtorus of $T$ consisting of matrices $g\in T$, such that $g_{ii}=g_{jj}=1$.
		
		Finally, let $\alpha\in\Delta_{\bar 1}$. As above, we have $G_{\alpha}=T$. However, $\Delta_{\bar 1}\cap (\mathbb{Q}\alpha +\mathbb{Q}\omega)=\{\pm\alpha\}$ if and only if
		$(m, n)\neq (2, 2)$, otherwise $\Delta_{\bar 1}\cap (\mathbb{Q}\alpha +\mathbb{Q}\omega)=\{\pm\alpha, \pm\beta\}$. More precisely, if $\alpha=\epsilon_i-\epsilon_j, 1\leq i\leq 2< j\leq 4$, then $\beta=\epsilon_{j'}-\epsilon_{i'}$, where $\{i'\}=\{1, 2\}\setminus\{i\}, \{j'\}=\{3, 4\}\setminus\{j\}$.
		Equivalently, $\mathfrak{G}_{\bar 1}^{T_{\alpha}}=\mathfrak{G}_{\bar 1}^{\alpha}\oplus \mathfrak{G}_{\bar 1}^{-\alpha}$ and both components $\mathfrak{G}_{\bar 1}^{\alpha}, \mathfrak{G}_{\bar 1}^{-\alpha}$ are one-dimensional if and only if $(m, n)\neq (2, 2)$, otherwise $\dim\mathfrak{G}_{\bar 1}^{\alpha}=\dim\mathfrak{G}_{\bar 1}^{-\alpha}=2$.
		
		Let $(m, n)\neq (2, 2)$. Then $\mathbb{G}_{\alpha}$ consists of matrices $g=(g_{kl})_{1\leq k, l\leq m+n}$, such that $g_{kl}\neq 0$ if and only if either $k=l$ or $k\neq l\in\{i, j\}$. We also have  
		\[\mathrm{Ber}(g)=(\prod_{1\leq k\neq i\leq m}g_{kk}\prod_{m+1\leq k\neq j\leq m+n}g_{kk}^{-1})((g_{ii}g_{jj}-g_{ij}g_{ji})g_{jj}^{-2})=1 .\]
		Since $m\neq 1$ or $n\neq 1$, we conclude $\mathbb{G}_{\alpha}\simeq\mathrm{GL}(1|1)\times T'$, where $T'$ is a subtorus of $T$ consisting of matrices $g\in T$, such that $g_{ii}=g_{jj}=1$. The first factor, that is isomorphic to $\mathrm{GL}(1|1)$, can be chosen as follows. Assume that $m>1$ (the case $n>1$ is similar). Let $k\in\{1, \ldots, m\}\setminus\{i\}$. Then the required supersubgroup consists of matrices $g=(g_{uv})_{1\leq u, v\leq m+n}$, such that $g_{uv}\neq 0$ if and only if either $u, v\in\{ i, j\}$ or $u=v\neq i, j$, and in the latter case $g_{uu}=1$, except $g_{kk}=(g_{ii}g_{jj}-g_{ij}g_{ji})g_{jj}^{-2}$.
		
		Finally, let $m=n=2$. Let $\alpha=\epsilon_1-\epsilon_3$ (any other choice can be treated similarly). Then $\mathbb{G}_{\alpha}$ consists of matrices 
		\[g=\left(\begin{array}{cccc}
			g_{11} & 0 & g_{13} & 0 \\
			0 & g_{22} & 0 & g_{24} \\
			g_{31} & 0 & g_{33} & 0 \\
			0 & g_{42} & 0 & g_{44} 	
		\end{array} \right)=g' g'', \ \mbox{where} \]
		\[g'=\left(\begin{array}{cccc}
			g_{11} & 0 & g_{13} & 0 \\
			0 & 1 & 0 & 0 \\
			g_{31} & 0 & g_{33} & 0 \\
			0 & 0 & 0 & 1 	
		\end{array} \right), \ g'' =\left(\begin{array}{cccc}
			1 & 0 & 0 & 0 \\
			0 & g_{22} & 0 & g_{24} \\
			0 & 0 & 1 & 0 \\
			0 & g_{42} & 0 & g_{44} 	
		\end{array}\right). \]
		Note that 
		\[\mathrm{Ber}(g')=\mathrm{Ber}(\left(\begin{array}{cc}
			g_{11} &  g_{13} \\
			g_{31} & g_{33}  	
		\end{array} \right)), \ \mathrm{Ber}(g'')=\mathrm{Ber}(\left(\begin{array}{cc}
			g_{22} & g_{24} \\
			g_{42} & g_{44} 	
		\end{array} \right)) \]
		and $\mathrm{Ber}(g)=\mathrm{Ber}(g')\mathrm{Ber}(g'')=1$. Thus
		\[g=g'\left(\begin{array}{cccc}
			1 & 0 & 0 & 0 \\
			0 & 1 & 0 & 0 \\
			0 & 0 & 1 & 0 \\
			0 & 0 & 0 & \mathrm{Ber}(g')
		\end{array} \right) \left(\begin{array}{cccc}
			1 & 0 & 0 & 0 \\
			0 & 1 & 0 & 0 \\
			0 & 0 & 1 & 0 \\
			0 & 0 & 0 & \mathrm{Ber}(g'')
		\end{array} \right)g'' ,\]
		and the matrices 
		\[g'\left(\begin{array}{cccc}
			1 & 0 & 0 & 0 \\
			0 & 1 & 0 & 0 \\
			0 & 0 & 1 & 0 \\
			0 & 0 & 0 & \mathrm{Ber}(g')
		\end{array} \right)=\left(\begin{array}{cccc}
			g_{11} & 0 & g_{13} & 0 \\
			0 & 1 & 0 & 0 \\
			g_{31} & 0 & g_{33} & 0 \\
			0 & 0 & 0 & \mathrm{Ber}(g') 	
		\end{array} \right), \]\[ \left(\begin{array}{cccc}
			1 & 0 & 0 & 0 \\
			0 & 1 & 0 & 0 \\
			0 & 0 & 1 & 0 \\
			0 & 0 & 0 & \mathrm{Ber}(g'')
		\end{array} \right)g''=\left(\begin{array}{cccc}
			1 & 0 & 0 & 0 \\
			0 & g_{22} & 0 & g_{24} \\
			0 & 0 & 1 & 0 \\
			0 & \mathrm{Ber}(g'')g_{42} & 0 & \mathrm{Ber}(g'')g_{44} 	
		\end{array}\right)\]
		belong to $\mathbb{G}=\mathrm{SL}(2|2)$. Conversely, the matrices from $\mathbb{G}$,  having the forms 
		\[h'=\left(\begin{array}{cccc}
			h_{11} & 0 & h_{13} & 0 \\
			0 & 1 & 0 & 0 \\
			h_{31} & 0 & h_{33} & 0 \\
			0 & 0 & 0 & z_{44} 
		\end{array} \right)\] 
		and 
		\[h''=\left(\begin{array}{cccc}
			1 & 0 & 0 & 0 \\
			0 & h_{22} & 0 & h_{24} \\
			0 & 0 & 1 & 0 \\
			0 & h_{42} & 0 & h_{44} 	
		\end{array}\right)\]
		respectively, form closed supersubgroups $\mathbb{H}'$ and $\mathbb{H}''$ in $\mathbb{H}$. Furthermore, $\mathbb{H}'\cap\mathbb{H}''=e$, $\mathbb{G}_{\alpha}=\mathbb{H}'\mathbb{H}''$
		and $\mathbb{H}'$ normalizes $\mathbb{H}''$. In fact,  for any $h'\in\mathbb{H}', h''\in\mathbb{H}''$ we have
		\[h' h'' (h')^{-1}=\left(\begin{array}{cccc}
			1 & 0 & 0 & 0 \\
			0 & h_{22} & 0 & h_{24}z_{44}^{-1} \\
			0 & 0 & 1 & 0 \\
			0 & z_{44}h_{42} & 0 & h_{44} 	
		\end{array}\right).\] 
		Additionally, $\mathbb{H}'\simeq\mathrm{GL}(1|1)$ and $\mathbb{H}''\simeq \mathrm{SL}(1|1)$, so that $\mathbb{G}_{\alpha}\simeq\mathrm{GL}(1|1)\ltimes\mathrm{SL}(1|1)$.
	\end{example}
	\begin{example}\label{centralizers in Q}
		Let $\mathbb{G}$ be the queer supergroup ${\bf Q}(n), n\geq 1$. Recall that ${\bf Q}(n)\leq\mathrm{GL}(n|n)$ consists of matrices
		\[\left(\begin{array}{cc}
			S & S' \\
			-S' & S	
		\end{array}\right), \ S\in\mathrm{GL}_n, \ S'\in\mathrm{Mat}_{n\times n}. \]
		We naturally identify $G=\mathbb{G}_{ev}$ with $\mathrm{GL}_n$ and the maximal torus $T\leq G$ with the subgroup of $\mathrm{GL}_n$ consisting of all diagnal matrices. 
		
		It is obvious that the root system of $\mathbb{G}$ has a form 
		\[\Delta=\{\epsilon_i-\epsilon_j\mid 1\leq i\neq j\leq n, \} \]
		and $\Delta=\Delta_{\bar 0}=\Delta_{\bar 1}$. Set $\alpha=\epsilon_i-\epsilon_j, 1\leq i< j\leq n$. Then $\mathbb{G}_{\alpha}$ consists of matrices
		\[\left(\begin{array}{cc}
			S & S' \\
			-S' & S	
		\end{array}\right), S\in G_{\alpha}, \ S'\in\mathrm{Mat}_{n\times n}^{T_{\alpha}}.\]
		More precisely, $S=(s_{kl})_{1\leq k, l\leq n}, S'=(s'_{kl})_{1\leq k, l\leq n}$ and the following conditions hold :
		$s_{kl}=0$ and $s'_{kl}=0$, whenever $k\neq l$ and $\{k, l\}\neq\{i, j\}$.
		
		One easily sees that $\mathbb{G}_{\alpha}\simeq {\bf Q}(2)\times \mathbb{T}'$. The first factor, that is isomorphic to ${\bf Q}(2)$, consists of all matrices satisfying the additional conditions $s_{kk}=1, s'_{kk}=0, k\neq i, j$. Respectively, $\mathbb{T}'$ is a supersubtorus of $\mathbb{T}$, that consists of all matrices from $\mathbb{G}_{\alpha}$ satisfying the additional conditions
		$s_{ii}=s_{jj}=1, s_{ij}=s_{ji}=0$ and $s'_{ii}=s'_{jj}=s'_{ij}=s'_{ji}=0$. Note also that $\mathbb{T}\simeq {\bf Q}(1)^n$ as well as $\mathbb{T}'\simeq {\bf Q}(1)^{n-2}$.

		The Harish-Chandra pair of $\mathbb{G}_{\alpha}$ is isomorphic to \[(\mathrm{GL}_2\times T', \mathfrak{gl}_2\oplus \Bbbk^{\oplus (n-2)}),\]
		where $T'$ acts trivially on the second component of this pair. Moreover, $\mathrm{GL}_2$ acts by conjugations on $\mathfrak{gl}_2$ and trivially one the second direct term.
	\end{example}

	\section{Supergroups with the even part isomorphic to $\mathrm{SL}_2$}
	
	Consider a supergroup $\mathbb{G}$, such that $\mathbb{G}_{ev}=G\simeq\mathrm{SL}_2$. Fix the maximal torus $T$ of $\mathrm{SL}_2$ consisting of diagonal matrices. Then $X(T)$ can be identified with $\mathbb{Z}$. Simple $\mathrm{SL}_2$-modules are indexed by nonnegative integers, so that for $n\geq 0$ the simple $\mathrm{SL}_2$-module $L(n)$ is the socle of $\mathrm{Sym}_n(V)\simeq H^0(n)$, where $V\simeq L(1)$ is the standard two-dimensional $\mathrm{SL}_2$-module with a weight basis $v_1, v_{-1}$. 
	We also note that the Weyl module $V(n)$ is isomorphic to $H^0(n)^*$ (cf. \cite[Section 2.3]{zub-bar}). In particular, $L(n)\simeq L(n)^*$ is isomorphic to the top of $V(n)\simeq\mathrm{Sym}_n(V)^*$ and additionally, if $0\leq n< p$, then $L(n)=H^0(n)\simeq V(n)$.
	
	The space $\mathrm{Sym}_n(V)$ has a (weight) basis $s_i=v_1^i v_{-1}^{n-i}, 0\leq i\leq n$. It is clear that the weight of $s_i$ is $2i-n$. The elements of dual basis of $\mathrm{Sym}_n(V)^*$ are denoted by $s_i^*$ (of weight $n-2i$ respectively).  
	
	Recall also that $L(2)=\mathrm{Sym}_2(V)$ is isomorphic to $\mathfrak{sl}_2$, which is regarded as a $\mathrm{SL}_2$-module with respect to the adjoint action.
	
	A matrix $E_{ij}$ is determined by
	$E_{ij}v_k=\delta_{jk}v_i, i, j, k\in\{1, -1\}$. The root subgroups of $\mathrm{SL}_2$ are 
	\[X_2(A)=\{X_2(t)=I_2+tE_{1, -1}\mid t\in A\} \ \mbox{and} \ X_{-2}(A)=\{ X_{-2}(t)=I_2+tE_{-1, 1}\mid t\in A\},\]
	where $A$ is a $\Bbbk$-algebra. Set $H=E_{11}-E_{-1, -1}$. The distribution algebra $\mathrm{Dist}(G)$ acts on $\mathrm{Sym}_n(V)$ and $\mathrm{Sym}_n(V)^*$ as
	\[\left(\begin{array}{c}
		H \\
		t
	\end{array}\right)s_i=\left(\begin{array}{c}
		2i-n \\
		t
	\end{array}\right)s_i, \   E_{1, -1}^{(k)}s_i=\left(\begin{array}{c}
		n-i \\
		k
	\end{array}\right)s_{i+k}, \ E_{-1, 1}^{(k)}s_i=\left(\begin{array}{c}
		i \\
		k
	\end{array}\right)s_{i-k}, \]
	\[\left(\begin{array}{c}
		H \\
		t
	\end{array}\right)s^*_i=\left(\begin{array}{c}
		n-2i \\
		t
	\end{array}\right)s^*_i, \   E_{1, -1}^{(k)}s^*_i=(-1)^k \left(\begin{array}{c}
		n-i+k \\
		k
	\end{array}\right)s^*_{i-k}, \]\[ E_{-1, 1}^{(k)}s^*_i=(-1)^k \left(\begin{array}{c}
		i+k \\
		k
	\end{array}\right)s^*_{i+k},  t, k\geq 0, \]
	respectively. 	
	\begin{lm}\label{non-symmetric case}
		Let $m< n$ be positive integers. Then
		\[\mathrm{Hom}_{\mathrm{SL}_2}(\mathrm{Sym}_m(V)^*\otimes \mathrm{Sym}_n(V)^*, \mathfrak{sl}_2)\simeq \mathrm{Hom}_{\mathrm{SL}_2}(\mathfrak{sl}_2, \mathrm{Sym}_m(V)\otimes \mathrm{Sym}_n(V))\]
		and this space is nonzero if and only if $n-m=2$. Besides, in the latter case this space is one-dimensional. 	
	\end{lm}  
	\begin{proof}
		Recall that $\mathfrak{sl}_2\simeq V(2)\simeq H^0(2)\simeq\mathfrak{sl}_2^*$, hence the first statement follows. Any morphism from 	$\mathrm{Hom}_{\mathrm{SL}_2}(\mathfrak{sl}_2, \mathrm{Sym}_m(V)\otimes \mathrm{Sym}_n(V))$ is uniquely determined by an $X_2$-stable element of weight $2$. The space $\mathrm{Sym}_m(V)\otimes \mathrm{Sym}_n(V)$ is spanned by
		the elements $s_i\otimes t_j$, where $s_i=v_1^i v_{-1}^{m-i}, t_j=v_1^j v_{-1}^{n-j}, 0\leq i\leq m, 0\leq j\leq n$. The weight of $s_i\otimes t_j$ is equal to $2(i+j)-m-n$, hence
		$\mathrm{Sym}_m(V)\otimes \mathrm{Sym}_n(V)$ has a component of weight $2$ if and only if $m+n$ is even. Set $m+n=2k$ and note that $m< k< n$. An element of weight $2$ has a form
		\[x=\sum_{0\leq i\leq m} c_i s_i\otimes t_{k+1-i},\]
		and it is $X_2$-stable if and only if $E_{1, -1}^{(r)}x=0$ for any $r\geq 1$. In particular, $E_{1, -1}x=0$ infers
		\[ (\star) \ (m-i+1)c_{i-1}+(n-k-1+i)c_i=0, c_{-1}=0, 0\leq i\leq m,\]
		provided $n\geq k+2$, otherwise the $0$-th equation should be eliminated. 
		
		Assume that $\mathrm{char}\Bbbk=0$ and $n\geq k+2$. In this case the system of linear equations $(\star)$ has only zero solution. On the other hand, if $n=k+1$ and respectively, $m=k-1$, then
		any solution of $(\star)$ has a form
		\[c_i=(-1)^i \left(\begin{array}{c}
			k-1 \\
			i
		\end{array}\right)c_0, 1\leq i\leq k-1.\]
		By Donkin-Mathieu theorem \cite[Proposition II.4.21]{jan}, the $\mathrm{SL}_2$-module $\mathrm{Sym}_m(V)\otimes \mathrm{Sym}_n(V)$ has a good filtration. Moreover, 
		\[\dim \mathrm{Hom}_{\mathrm{SL}_2}(\mathfrak{sl}_2, \mathrm{Sym}_m(V)\otimes \mathrm{Sym}_n(V))=\dim \mathrm{Hom}_{\mathrm{SL}_2}(V(2), \mathrm{Sym}_m(V)\otimes \mathrm{Sym}_n(V))\] equals the multiplicity of $H^0(2)\simeq\mathfrak{sl}_2$ in any such filtration.
		Since the formal character of $\mathrm{Sym}_m(V)\otimes \mathrm{Sym}_n(V)$, as well as the formal characters of induced modules $H^0(l)$, do not depend on $\mathrm{char}\Bbbk$, this dimension also does not. Lemma is proved.
	\end{proof}	
	Let $s_i^*$ and $t_j^*$ denote the elements of dual bases of $\mathrm{Sym}_m(V)^*$ and $\mathrm{Sym}_n(V)^*$ respectively. Then any morphism from
	$\mathrm{Hom}_{\mathrm{SL}_2}(\mathrm{Sym}_m(V)^*\otimes \mathrm{Sym}_n(V)^*, \mathfrak{sl}_2)$ has a form
	\[s^*_i\otimes t^*_{k-i}\mapsto a_i H, \ s^*_i\otimes t^*_{k-1-i}\mapsto b_i E_{1, -1}, \ s_i^*\otimes t^*_{k+1-i}\mapsto c_i E_{-1, 1}, 0\leq i\leq k-1.\]
	Since this morphism commutes with the induced action of $\mathfrak{sl}_2$, we obtain
	\[(k-i)b_{i-1}+(i+2)b_i=2a_i, \ (i+1)c_{i+1}+(k+1-i)c_i=-2a_i,\]
	\[(i+1)a_{i+1}+(k-i)a_i=b_i, \ (k-i)a_{i-1}+(i+1)a_i=-c_i, \]
	\[0\leq i\leq k-1.\]
	Note also that if $i$ occurs to be negative or at least $k$, then $a_i=b_i=c_i=0$. 
	
	It is easy to check that this system has the nontrivial, whence unique, solution
	\[(\star\star) \ a_i=(-1)^{i+1} \left(\begin{array}{c}
		k-1 \\
		i
	\end{array}\right)a, \ b_i=(-1)^{i+1} \left(\begin{array}{c}
		k-1 \\
		i
	\end{array}\right)a, \ c_i=(-1)^i \left(\begin{array}{c}
		k-1 \\
		i
	\end{array}\right)a,\]\[ 0\leq i\leq k-1, \ a\in\Bbbk.\]  
	The following lemma refines \cite[Proposition 2.15]{zub-bar}.
	\begin{lm}\label{if m=n}
		We have
		\[\mathrm{Hom}_{\mathrm{SL}_2}((\mathrm{Sym}_n(V)^*)^{\otimes 2}, \mathfrak{sl}_2)\simeq \mathrm{Hom}_{\mathrm{SL}_2}(\mathfrak{sl}_2, \mathrm{Sym}_n(V)^{\otimes 2})\]	
		and both spaces are one-dimensional. Moreover, if $n$ is odd, then
		\[\mathrm{Hom}_{\mathrm{SL}_2}((\mathrm{Sym}_n(V)^*)^{\otimes 2}, \mathfrak{sl}_2)=\mathrm{Hom}_{\mathrm{SL}_2}(\mathrm{Sym}_2(\mathrm{Sym}_n(V)^*), \mathfrak{sl}_2)\]
		and $\mathrm{Hom}_{\mathrm{SL}_2}(\Lambda^2(\mathrm{Sym}_n(V)^*), \mathfrak{sl}_2)=0$, otherwise
		\[\mathrm{Hom}_{\mathrm{SL}_2}((\mathrm{Sym}_n(V)^*)^{\otimes 2}, \mathfrak{sl}_2)=\mathrm{Hom}_{\mathrm{SL}_2}(\Lambda^2(\mathrm{Sym}_n(V)^*), \mathfrak{sl}_2)\]
		and $\mathrm{Hom}_{\mathrm{SL}_2}(\mathrm{Sym}_2(\mathrm{Sym}_n(V)^*), \mathfrak{sl}_2)=0$.
	\end{lm}
	\begin{proof}
		Arguing as in \cite[Proposition 2.15]{zub-bar}, one easily sees that the natural isomorphism $(\mathrm{Sym}_n(V)^*)^{\otimes 2}\simeq (\mathrm{Sym}_n(V)^{\otimes 2})^*$
		sends $\mathrm{Sym}_2(\mathrm{Sym}_n(V)^*)$ and $\Lambda^2(\mathrm{Sym}_n(V)^*)$ onto $\mathrm{Sym}_2(\mathrm{Sym}_n(V))^*$ and $\Lambda^2(\mathrm{Sym}_n(V))^*$ respectively.
		In particular, we have the naturally induced isomorphisms 
		\[\mathrm{Hom}_{\mathrm{SL}_2}(\mathrm{Sym}_2(\mathrm{Sym}_n(V)^*), \mathfrak{sl}_2)\simeq \mathrm{Hom}_{\mathrm{SL}_2}(\mathfrak{sl}_2, \mathrm{Sym}_2(\mathrm{Sym}_n(V))) \]	
		and 
		\[\mathrm{Hom}_{\mathrm{SL}_2}(\Lambda^2(\mathrm{Sym}_n(V)^*), \mathfrak{sl}_2)\simeq \mathrm{Hom}_{\mathrm{SL}_2}(\mathfrak{sl}_2, \Lambda^2(\mathrm{Sym}_n(V))). \]
		As above, the morphisms from $\mathrm{Hom}_{\mathrm{SL}_2}(\mathfrak{sl}_2, \mathrm{Sym}_n(V)^{\otimes 2})$ are in one-to-one correspondence with the elements
		$z=\sum_{1\leq i\leq n} c_i s_i\otimes s_{n+1-i}$ whose coefficients satisfy 
		\[(\circ) \ c_{i-1}(n-i+1)+c_i(i-1)=0, 1\leq i\leq n, \ c_0=0.\]
		If $\mathrm{char}\Bbbk=0$, then this system has the unique solution
		\[c_i=(-1)^{i-1}\left(\begin{array}{c}
			n-1 \\
			i-1
		\end{array}\right)c_1, \ 1\leq i\leq n.\]
		The element $z$ belongs to $\mathrm{Sym}_2(\mathrm{Sym}_n(V))$ (respectively, $z$ belongs to $\Lambda^2(\mathrm{Sym}_n(V))$) if and only if $\tau(z)=z$ (respectively, $\tau(z)=-z$), where $\tau$ is a \emph{twist map} $f\otimes g\mapsto g\otimes f, f, g\in \mathrm{Sym}_n(V)$. Using the same good filtration arguments, we complete the proof.  
	\end{proof}
	Recall that any morphism from $\mathrm{Hom}_{\mathrm{SL}_2}((\mathrm{Sym}_n(V)^*)^{\otimes 2}, \mathfrak{sl}_2)$
	has a form
	\[s^*_i\otimes s^*_{n-i}\mapsto a_i H, \ s^*_{j}\otimes s^*_{n-1-j}\mapsto b_j E_{1, -1}, \ s^*_k\otimes s^*_{n+1-k}\mapsto c_k E_{-1, 1}, \]
	where 
	\[(\circ\circ) \ a_i=\frac{(-1)^i}{2} (\left(\begin{array}{c}
		n-1 \\
		i
	\end{array}\right)-\left(\begin{array}{c}
		n-1 \\
		i-1
	\end{array}\right))a, \ 0\leq i\leq n,\]
	\[b_j=(-1)^j \left(\begin{array}{c}
		n-1 \\
		j
	\end{array}\right)a, \ 0\leq j\leq n-1, \]
	\[c_k=(-1)^k \left(\begin{array}{c}
		n-1 \\
		k-1
	\end{array}\right)a, \ 1\leq k\leq n, \ a\in\Bbbk.\]
	These formulas were derived in \cite{zub-bar} regardless the parity of $n$. 
	\begin{rem}\label{onto}
		From now on we use the following elementay observation. Let $\phi : L(m)\otimes L(n)\to\mathfrak{sl}_2$ be a nonzero morphism of $\mathrm{SL}_2$-modules.  Then we have a nonzero composition of $\mathrm{SL}_2$-module morphisms \[V(m)\otimes V(n)\simeq \mathrm{Sym}_m(V)^*\otimes \mathrm{Sym}_n(V)^*\to L(m)\otimes L(n)\to \mathfrak{sl}_2, \] 
		denoted by $\phi'$. Lemma \ref{non-symmetric case} and Lemma \ref{if m=n} imply that either $m=n$ or $|m-n|=2$. Moreover, $\phi'$ is uniquely (up to a nonzero multiple) determined by the formulas $(\circ\circ)$ and $(\star\star)$ respectively. In particular, $\phi$ is also uniquely (up to a nonzero multiple) determined by the same formulas. In fact,  for any
		$x\in L(m), y\in L(n)$ let $x'\in V(m), y'\in L(n)$ be their preimages. Then $\phi(x\otimes y)=\phi'(x'\otimes y')$. 
		
		As a rule, this remark will apply in the following cases. Let $V_1\simeq L(m)$ and $V_2\simeq L(n)$ be the (not necessary different) $\mathrm{SL}_2$-submodules of $\mathfrak{G}_{\bar 1}$, such that $[V_1, V_2]\neq 0$. Then the Lie bracket induces a nonzero morphism $V_1\otimes V_2\to\mathfrak{sl}_2$ of $\mathrm{SL}_2$-modules. If it does not lead to confusion, we write
		$[x, y]=[x', y']$ for $x\in V_1, y\in V_2$, where $[x', y']$ is the result of applying of the formulas $(\star\star)$ and $(\circ\circ)$ to $x'\otimes y'$.
		
		Similarly, assume that $K$ is a $\mathrm{SL}_2$-submodule of $\mathfrak{G}_{\bar 1}$, such that $[V_1, K]=0$ and $V_2=\mathfrak{G}_{\bar 1}/K\simeq L(n)$. Then we have the induced (nonzero) morphism $V_1\otimes V_2\to\mathfrak{sl}_2$ of $\mathrm{SL}_2$-modules, determined as
		\[x\otimes \overline{y}\mapsto [x, y], \ \mbox{where} \ \overline{y}=y+K,  x\in V_1, y\in \mathfrak{G}_{\bar 1}.\]
		As above, we write $[x, \overline{y}]=[x', y']$, where $x'$ and $y'$ are preimages of $x$ and $\overline{y}$ in $V(m)$ and $V(n)$ respectively.
	\end{rem}
	\begin{lm}\label{when pairing is trivial}
		The space $\mathrm{Hom}_{\mathrm{SL}_2}(L(m)\otimes L(n), \mathfrak{sl}_2)$ is trivial if :
		\begin{enumerate}
			\item $p|m$ and $p|n$;
			\item $p|n$ and $n=m+2$.
		\end{enumerate}		
	\end{lm}
	\begin{proof}
		Let $p|m$ and $p|n$. In particular, all weights of $L(n)\simeq L(\frac{n}{p})^{[1]}$ (as well as of $L(m)$) are multiples of $p$. Let $\phi\in \mathrm{Hom}_{\mathrm{SL}_2}(L(m)\otimes L(n), \mathfrak{sl}_2)\setminus 0$. Then the image of $\phi$ coincides with $\mathfrak{sl}_2$ and all its weights are multiples of $p$ also, which is a contradiction.   
		
		Let $p|n$ and $n=m+2$. Let $\phi\in \mathrm{Hom}_{\mathrm{SL}_2}(L(m)\otimes L(n), \mathfrak{sl}_2)\setminus 0$. Using $(\star\star)$, we have $\phi'(s^*_0\otimes t^*_{m+1})=-aH\neq 0$, but the weight of $t^*_{m+1}$ equals $-m$, hence it is coprime to $p$ and the image of $t^*_{m+1}$ in $L(n)$ is zero. Thus the image of $s^*_0\otimes t_{m+1}^*$ in $L(m)\otimes L(n)$ is zero, which is a contradiction.
	\end{proof}

	\begin{rem}\label{blocks}
		By \cite[II.7.2(1)]{jan} the block $B(n)$ of a dominant weight $n\geq 0$ is contained in $((n+2p\mathbb{Z})\sqcup ((-n-2)+2p\mathbb{Z}))\cap \mathbb{N}$. 
	\end{rem}
	Let $N$ be a finite dimensional $\mathrm{SL}_2$-module. Let $l=l(N)$ denote its \emph{Loewy length}. In other words, there is a $\mathrm{SL}_2$-module filtration of $N$
	\[0\subseteq N_{(1)}\subseteq N_{(2)}\subseteq \ldots\subseteq N_{(l)}=N,\]
	called a \emph{Loewy filtrations}, such that $N_{(0)}=0$ and $N_{(i)}/N_{(i-1)}$ is the socle of $N/N_{(i-1)}, 1\leq i\leq l$.
	
	In what follows we suppose that $\mathfrak{G}_{\bar 1}\neq 0$. Let $l=l(\mathfrak{G}_{\bar 1})$. To simplify our notations, we denote each term $(\mathfrak{G}_{\bar 1})_{(i)}$ just by $M_{(i)}$. For arbitrary $\mathrm{SL}_2$-submodule $S$ of $\mathfrak{G}_{\bar 1}$, let $K(S)$ denote the largest submodule of $S$, such that  $[K(S), S]=0$. We  call
	$K(S)$ a \emph{kernel} of $S$.
	\begin{lm}\label{p|m and n=m+2}
	Let $\mathfrak{G}_{\bar 1}=V_1\oplus V_2$, where $V_1\simeq L(m), V_2\simeq L(n)$, $[V_1, V_1]=[V_2, V_2]=0$ and $[V_1, V_2]\neq 0$. If $n=m+2$ and $p|m$, then $m=0$.	
	\end{lm}
\begin{proof}
	By Steinberg's tensor product theorem, we have \[L(n)\simeq L(2)\otimes L(\frac{(n-2)}{p})^{[1]}\simeq L(2)\otimes L(m).\]
	Furthermore, if $m=m_1 p+\ldots m_s p^s, 0\leq m_i< p, 1\leq i\leq s$, then the weight basis of $L(m)$ consists of $\overline{s_k^*}=s_k^*+\mathrm{rad} V(m)$, $k\in W_m=\{\sum_{1\leq i\leq s}  k_i p^i\mid 0\leq k_i\leq m_i\}$. We also fix the basis $t^*_j, 0\leq j\leq 2,$ of $L(2)$. The formulas $(\star\star)$ imply that $[\overline{s^*_i}, t^*_j\otimes\overline{s^*_k}]\neq 0$
	if and only if $i+k=m$, and then
	\[[\overline{s^*_i}, t^*_0\otimes\overline{s^*_{m-i}}]=(-1)^{i+1}\left(\begin{array}{c}
	m \\
	i
	\end{array}\right)a E_{1, -1}, \ [\overline{s^*_i}, t^*_1\otimes\overline{s^*_{m-i}}]=(-1)^{i+1}\left(\begin{array}{c}
	m \\
	i
\end{array}\right)a H, \]
\[ [\overline{s^*_i}, t^*_2\otimes\overline{s^*_{m-i}}]=(-1)^{i}\left(\begin{array}{c}
	m \\
	i
\end{array}\right)a E_{-1, 1}, \ i\in W_m.\] 
Since $p|m$, $V_1$ is a trivial $\mathfrak{sl}_2$-module and $V_2=\oplus_{i\in W_m} L(2)\otimes \overline{s_i^*}$ is isomorphic to $L(2)^{\oplus |W_m|}$ as an $\mathfrak{sl}_2$-module. Set	
\[v=\sum_{i\in W_m}\alpha_i \overline{s_i^*}+\sum_{i\in W_m, 0\leq j\leq 2} \beta_{ij} t_j^*\otimes\overline{s_i^*}.  \]
Then $[v, v]=\sum_{i\in W_m}(-1)^i \left(\begin{array}{c}
	m \\
	i
\end{array}\right)\alpha_i a z_i$, where
\[z_i=-\beta_{m-i, 0}E_{1, -1}-\beta_{m-i, 1}H+\beta_{m-i, 2} E_{-1, 1}.\]
Note also that
\[\left(\begin{array}{c}
	m \\
	i
\end{array}\right)\equiv \prod_{1\leq l\leq s}\left(\begin{array}{c}
m_l \\
i_l
\end{array}\right)\not\equiv 0\pmod p \]
for any $i=\sum_{1\leq l\leq s}i_l p^l\in W_m$.

By the above remark, the identity $[[v, v], v]=0$ is equivalent to \[[[v, v],  \sum_{0\leq j\leq 2}\beta_{ij}t_j^*\otimes\overline{s_i^*}]=0\] for each $i\in W_m$.
If $m>0$, then $|W_m|\geq 2$ and therefore, for each index $i$ one can vary the coefficients of the remaining terms in $v$ so that the element $[v, v]$ will run through the entire algebra $\mathfrak{sl}_2$. Since $L(2)$ is isomorphic to $\mathfrak{sl}_2$, regarded as the adjoint $\mathfrak{sl}_2$-module, then each $\sum_{0\leq j\leq 2}\beta_{ij}t_j^*\otimes\overline{s_i^*}\neq 0$ corresponds to a nonzero central element of $\mathfrak{sl}_2$, which is nonsense.

Finally, if $m=0$ and we identify $L(n)=L(2)$ with the adjoint $\mathrm{SL}_2$-module $\mathfrak{sl}_2$, then the nonzero Lie bracket on $L(0)\oplus L(2)$ is uniquely (up to a nonzero scalar $a$) determined as
\[ [b, b']=[x, y]=0, \ [b, x]=abx, \ b, b'\in\Bbbk\simeq L(0), \ x, y\in\mathfrak{sl}_2.\]	
\end{proof}
	\begin{pr}\label{M_{(1)}}
		We have $M_{(1)}=C_{(1)}\oplus K(M_{(1)})$. Furthermore, only one of the following alternatives holds :
		\begin{enumerate} 
			\item $C_{(1)}=L(1)^s, \ 0\leq s\leq 1$;
			\item $C_{(1)}= L(0)^t\oplus L(2)^t$, where $[L(0)^t, L(0)^t]=[L(2)^t, L(2)^t]=0$ and $0\leq t\leq 1$.  
		\end{enumerate}
		Besides, $K(M_{(1)})$ is a trivial $\mathfrak{sl}_2$-module, provided $C_{(1)}\neq 0$. 
	\end{pr}
	\begin{proof}
		The last statement easily follows by Lemma \ref{may be useful}. 
		For any simple direct term $V_1$ of $C_{(1)}$ there is a simple direct term $V_2$ (possible equal to $V_1$), such that $[V_1, V_2]\neq 0$. Let $V_1\simeq L(m)$ and $V_2\simeq L(n)$.	
		
		In the notations of  Lemma \ref{non-symmetric case}, we fix the bases $s_i^*$ and $t^*_j, 1\leq i\leq m, 1\leq j\leq n$, of $V(m)$ and $V(n)$ respectively. Their images in $L(m)$ and $L(n)$
		are denoted by $\overline{s_i^*}$ and $\overline{t_j^*}$ respectively.
		
		First, consider the case $n=m$ and $V_1\neq V_2$. If $[V_1, V_1]\neq 0$, then by \cite[Theorem 2.25]{zub-bar} we have $n=1$. On the other hand, Lemma \ref{g-invariants} and Lemma \ref{may be useful} imply $p|m$, which is a contradiction. Therefore, $[V_1, V_1]=[V_2, V_2]=0$. 
		
		By Lemma \ref{when pairing is trivial}, the integer $n$ is coprime to $p$. Set $v=\overline{s_0^*}+\overline{t^*_n}\neq 0$. We have $[[v, v], v]=an(\overline{s^*_0}-\overline{t^*_n})\neq 0$ and this contradiction
		implies $V_1=V_2$. Combining the condition $[V_1, V_1]\neq 0$ and \cite[Theorem 2.25]{zub-bar}, we obtain $n=1$. 
		
		Now, let $m< n$. Then $n-m=2$. If $[V_2, V_2]\neq 0$, then the above arguments show that $n=1$, whence $m=-1$. This contradiction infers $[V_2, V_2]=0$.
		For $V_1$ there are two alternatives, either $[V_1, V_1]\neq 0$, that in turn implies $m=1, n=3$, or $[V_1, V_1]=0$. 
		
		Assume that $[V_1, V_1]\neq 0$, so that $m=1, n=3$ . Lemma \ref{when pairing is trivial} infers $p> 3$, and hence $L(3)\simeq V(3)$. The aforementioned formulas $(\circ\circ)$ imply
		\[[s_0^*, s_1^*]=\frac{a'}{2} H, \ [s_0^*, s_0^*]=a' E_{1, -1}, \ [s_1^*, s_1^*]=-a' E_{-1, 1}, a'\in \Bbbk\setminus 0.\]
		Similarly, the formulas $(\star\star)$ infer
		\[[s_0^*, t^*_2]=-aH, \ [s^*_1, t^*_1]=aH, \]  	 
		\[[s^*_0, t^*_1]=-aE_{1, -1}, \ [s^*_1, t^*_0]=aE_{1, -1},\]
		\[[s^*_0, t^*_3]=aE_{-1, 1}, \ [s^*_1, t^*_2]=-aE_{-1, 1}, \ a\in\Bbbk\setminus 0.\]	
		Set $v=s_0^*+t^*_2$. Then
		\[ [[v, v], v]=-2a' t_1^* -2as_0^*+2at^*_2\neq 0,\]
		which is a contradiction.  Therefore, we obtain $[V_1, V_1]=0$. 
		
		Now, set $v=\overline{s^*_0}+\overline{t^*_n}\neq 0$. We have $[[v, v], v]=-2a\overline{s_1^*}\neq 0$,
		except the case $p|m$, in which case Lemma \ref{p|m and n=m+2} implies $m=0$.
		
		Summing all up, we obtain that $C_{(1)}=L(1)^s\oplus L(0)^t\oplus L(2)^u, 0\leq s\leq 1, t, u\geq 0$. Besides, $[L(1)^s, L(0)^t\oplus L(2)^u]=0$ and $t$ is nonzero if and only if $u$ is. 
		
		Assume that $s, t, u\geq 1$. Set $x=x_1+x_2, x_1\in L(1), x_2\in L(2)$. Then the identity $[[x, x], x]=0$ is equivalent to $[[x_1, x_1], x_2]=0$. Since $L(2)\simeq\mathfrak{sl}_2$, the latter would imply that  $E_{1, -1}$ is a central element of $\mathfrak{sl}_2$. This contradiction infers that either 
		$C_{(1)}=L(1)^s$ or $C_{(1)}=L(0)^t\oplus L(2)^u$. 
		
		Now, assume that $u> 1$. Set $v=x+y_1+y_2$, where $x\in L(0)$, but the elements $y_1$ and $y_2$ belong to the different terms isomorphic to $L(2)$. The identity $[[v, v], v]=0$ 
		is equivalent to $[[x, y_1], y_2]=[[x, y_2], y_1]=0$. Without loss of generality, one can suppose that $[x, y_1]\neq 0$, whence again, this element is central in $\mathfrak{sl}_2$. This contradiction implies $u\leq 1$. 
		
		Finally, if $t>1$, then let $x_1$ and $x_2$ be the basic vectors of two different terms isomorphic to $L(0)$. Since both of them are not orthogonal to $L(2)$, by Lemma \ref{non-symmetric case} there is a nonzero scalar $\alpha\in\Bbbk$, such that
		$[x_1-\alpha x_2, L(2)]=0$. In other words, $\Bbbk(x_1-\alpha x_2)\simeq L(0)$ is contained in $K(M_{(1)})\cap C_{(1)}=0$. This contradicion concludes the proof.   
	\end{proof}
	\begin{cor}\label{description in char=0}
		Assume that $l=1$, i.e. $\mathfrak{G}_{\bar 1}$ is semi-simple as a $\mathrm{SL}_2$-module. For example, it takes place if $\mathrm{char}\Bbbk=0$, since then $\mathrm{SL}_2$ is linearly reductive. As we have already observed, the Harish-Chandra subpair $(e, K(\mathfrak{G}_{\bar 1}))$ corresponds to $\mathrm{R}_u(\mathbb{G})$. If $\mathbb{G}$ is not split, then \cite[Theorem 2.25]{zub-bar} implies
		that $\mathbb{G}/\mathrm{R}_u(\mathbb{G})\simeq\mathrm{SpO}(2|1)$, provided the first alternative of Proposition \ref{M_{(1)}} holds, otherwise the Harish-Chandra subpair $(\mathrm{SL}_2, K(\mathfrak{G}_{\bar 1})\oplus L(2))$ corresponds to a normal split
		supersubgroup $\mathbb{L}\simeq\mathrm{SL}_2\ltimes (\mathbb{G}_a^-)^{\dim K(\mathfrak{G}_{\bar 1})+3}$, such that $\mathbb{L}\geq\mathrm{R}_u(\mathbb{G})$ and $\mathbb{G}/\mathbb{L}\simeq \mathbb{G}_a^-$. 
	\end{cor}
	Until the end of this section we assume that $\mathrm{char}\Bbbk=p>0$ and $l\geq 2$. Let $K_{(l)}$ denote the largest $\mathrm{SL}_2$-submodule of $\mathfrak{G}_{\bar1}$, such that
	$[K_{(l)}, M_{(l-1)}]=0$. 
	\begin{lm}\label{if K_2 is large}
		If $K_{(l)}+M_{(l-1)}=\mathfrak{G}_{\bar 1}$, then $K(\mathfrak{G}_{\bar 1})\neq 0$.	
	\end{lm}
	\begin{proof}
		Since $K_{(l)}\neq 0$, we have $0\neq M_{(1)}\cap K_{(l)}\subseteq K_{(l)}\cap M_{(l-1)}$. It remains to note that $[K_{(l)}\cap M_{(l-1)}, \mathfrak{G}_{\bar 1}]=0$.	
	\end{proof}
	
	Suppose that $l=2$ and $[\mathfrak{G}_{\bar 1}, \mathfrak{G}_{\bar 1}]\neq 0$. It is clear that $K_{(2)}\cap M_{(1)}=K(M_{(1)})$. 
	
		Let $M$ be a $\mathrm{SL}_2$-module and $N$ be its submodule. We call the factor-module $M^{\oplus s}/R$, where the submodule $R$ is generated by the elements
	\[(0, \ldots, 0, \underbrace{n}_{i-\mbox{th place}}, 0, \ldots, 0)-(0, \ldots, 0, \underbrace{n}_{j-\mbox{th place}}, 0, \ldots, 0), 1\leq i\neq j\leq s, n\in N,\]
	an \emph{amalgamated direct sum} of $s$ copies of $M$ along $N$, and denote it by $M^{\oplus_N s}$. It is clear that $N^{\oplus s}/R\simeq N$ and
	$M^{\oplus_N s}/N\simeq (M/N)^{\oplus s}$.
	
	\begin{lm}\label{K_2=0}
		If $K_{(2)}=0$, then $p=3$ and $\mathfrak{G}_{\bar 1}\simeq V(3)$.	
	\end{lm}
	\begin{proof} By Proposition \ref{M_{(1)}} we have either $M_{(1)}=L(1)$ or $M_{(1)}=L(0)\oplus L(2)$. Let $N$ be a $\mathrm{SL}_2$-submodule of $\mathfrak{G}_{\bar 1}$, such that $M_{(1)}=N_{(1)}$ and $N/M_{(1)}\simeq L(n)$ is a direct simple term of $\mathfrak{G}_{\bar 1}/M_{(1)}$. By Remark \ref{blocks}, $\mathrm{Ext}^1_{\mathrm{SL}_2}(L(s), L(t))\neq 0$ infers $2|(s-t)$, and $s\neq t$ by \cite[II.2.12(1)]{jan}.
		In particular, if $M_{(1)}=L(1)$, then $n$ is odd and $n\geq 3$. By Lemma \ref{Weyl module is mapped onto}, $N$ is a highest weight $\mathrm{SL}_2$-module with $K(N)=0$, therefore, \cite[Theorem 2.25]{zub-bar} implies that $n=p=3$ and $N\simeq V(3)$. In particular, $\mathfrak{G}_{\bar 1}/M_{(1)}\simeq L(3)^s$, that obviously entails $\mathfrak{G}_{\bar 1}\simeq V(3)^{\oplus_{L(1)} s}, s\geq 1$.
		
		Similarly, if $M_{(1)}=L(0)\oplus L(2)$, then $n$ is even and the above arguing implies that it may not be strictly larger than $2$, that is $n=0, 2$. Since $L(0)$ and $L(2)$ are obviously tilting modules (cf. \cite[II.E.1]{jan}), \cite[Proposition II.4.16]{jan} infers that the embedding $M_{(1)}\to \mathfrak{G}_{\bar 1}$ is split, which is a contradicton.
		
		Now, let $N_i$ denote the submodule of $\mathfrak{G}_{\bar 1}$, such that $M_{(1)}=(N_i)_{(1)}$ and $N_i/M_{(1)}$ is isomorphic to the $i$-th direct simple term of $\mathfrak{G}_{\bar 1}/M_{(1)}$. To simplify our notations we denote
		the basic vectors $s^*_0, s^*_1, s^*_2, s^*_3$ of $\mathrm{Sym}_3(V)^*$ by $w_3, w_1, w_{-1}, w_{-3}$, respectively. Since each $N_i$ is isomorphic to $\mathrm{Sym}_3(V)^*$, we denote the corresponding basic vectors of $N_i$ by $w^{(i)}_3, w^{(i)}_1, w^{(i)}_{-1}, w^{(i)}_{-3}$. Besides, since each two modules $N_i$ and $N_j$ have the common socle $M_{(1)}$, one can assume that $w^{(i)}_{\pm 1}=w^{(j)}_{\pm 1}$
		for any $1\leq i, j\leq s$. Using the formulas $(\circ\circ)$, we have
		\[ [w_{\pm 3}^{(i)}, w^{(i)}_{\mp 3}]=-a_iH, \ [w^{(i)}_{\pm 1}, w^{(i)}_{\mp 1}]=a_iH, \]
		\[[w^{(i)}_3, w^{(i)}_{-1}]=[w^{(i)}_{-1}, w^{(i)}_3]=[w^{(i)}_1, w^{(i)}_1]=a_iE_{1, -1}, \]
		\[[w^{(i)}_{-3}, w^{(i)}_1]=[w^{(i)}_1, w^{(i)}_{-3}]=[w^{(i)}_{-1}, w^{(i)}_{-1}]=-a_i E_{-1, 1}, \ a_i\in\Bbbk\setminus 0, \ 1\leq i\leq s,\]
		and our assumption implies $a_i=a_j=a\neq 0$ for any $1\leq i, j\leq s$.
		
		Let $1\leq i\neq j\leq t$. Then $[w_3^{(i)}, w_{-3}^{(j)}]=a_{ij}H$ and the Jacobi identity implies
		\[[[w_3^{(i)}, w_{-3}^{(j)}], w_{-1}^{(i)}]=-a_{ij}w^{(i)}_{-1}=-[[w_{-3}^{(j)}, w_{-1}^{(i)}], w_3^{(i)}]-[[w_3^{(i)}, w_{-1}^{(i)}], w_{-3}^{(j)}]= a w^{(j)}_{-1},\] 
		hence $a_{ij}=-a$ and we conclude that $[w_{\pm 3}^{(i)}, w^{(j)}_{\mp 3}]=-aH$ for all $1\leq i, j\leq s$. 
		
		Similarly, we have $[w_{3}^{(i)}, w_{-1}^{(j)}]=b_{ij}E_{1, -1}$ and the Jacobi identity implies
		\[[[w_{3}^{(i)}, w_{-1}^{(j)}], w_{-1}^{(i)}]=b_{ij}w_1^{(i)}=-[[w_{-1}^{(j)}, w_{-1}^{(i)}], w_{3}^{(i)}]-[[w_{3}^{(i)}, w_{-1}^{(i)}], w_{-1}^{(j)}]= \]
		\[-aw_1^{(i)}-aw_1^{(j)}=-2aw_1^{(i)}=aw_1^{(i)},\]
		hence $b_{ij}=a$, so that $[w^{(i)}_3, w^{(j)}_{-1}]=[w^{(j)}_{-1}, w^{(i)}_3]=aE_{1, -1}$ for all $1\leq i, j\leq t$. We leave for the reader to show that
		$[w^{(i)}_{-3}, w^{(j)}_1]=[w^{(j)}_1, w^{(i)}_{-3}]=-aE_{-1, 1}$ for all $1\leq i, j\leq s$.
		
		Since the value of $[w_a^{(i)}, w_b^{(j)}]$ does not depend on the indices $i$ and $j$, we have 
		\[[w_a^{(1)}-w_a^{(2)}, w_b^{(j)}]=0, \ a\in\{\pm 3\}, \ b\in\{\pm 1, \pm 3\}, \ 1\leq j\leq s,\]
		provided $s\geq 2$. 
		In other words, in this case $K(\mathfrak{G}_{\bar 1})\neq 0$, which contradicts the fact that $M_{(1)}\subseteq K(\mathfrak{G}_{\bar 1})$ and $[M_{(1)}, M_{(1)}]\neq 0$.  
		\end{proof}
	\begin{rem}\label{H(3^s/1) does not exists, if s> 1}
The proof of Lemma \ref{K_2=0} can be interpreted so that there is no nontrivial Harish-Chandra pair structure on $(\mathrm{SL}_2, V(3)^{\oplus_{L(1)} s})$, provided $s\geq 2$. 		
	\end{rem}
	\begin{pr}\label{if top is simple and l=2}
		Let $\mathfrak{G}_{\bar 1}/M_{(1)}\simeq L(n)$ and $K(\mathfrak{G}_{\bar 1})=0$. Then $p=3$ and $\mathfrak{G}_{\bar 1}\simeq V(3)$.
	\end{pr}
	\begin{proof}
		Without loss of generality, one can suppose that $K_{(2)}\neq 0$. By the remark in Lemma \ref{if K_2 is large}, $K_{(2)}=K(M_{(1)})$. Note also that for any simple submodule  
		$R\subseteq K(M_{(1)})$ the Lie bracket induces a nontrivial morphism $R\otimes \mathfrak{G}_{\bar 1}/M_{(1)}\to \mathfrak{sl}_2$ of $\mathrm{SL}_2$-modules. Following Remark \ref{onto}, we write $[x, y]=[x, \overline{y}]$, where $x\in R, y\in \mathfrak{G}_{\bar 1}$ and $\overline{y}=y+M_{(1)}$. As above, we have  $R\simeq L(s)$, where $s\in\{n, n\pm 2\}$.
		We will split the proof into two cases.

		A) Let $C_{(1)}\neq 0$, or equivalently, $[M_{(1)}, M_{(1)}]\neq 0$. Then by Lemma \ref{may be useful}, or by Lemma \ref{may be useful too} as well, each simple term of $K(M_{(1)})$ has the highest weight multiple of $p$. Thus follows that $K(M_{(1)})\simeq L(s)$, where $s\in\{n, n\pm 2\}$. Indeed, it is clear that $K(M_{(1)})$ does not have simple terms with different highest weights. Next, if $K(M_{(1)})$ has at least two simple terms, say $L_1$ and $L_2$, then by Lemma \ref{non-symmetric case} and Lemma \ref{if m=n} there is an isomorphism $\phi : L_1\to L_2$ of $\mathrm{SL}_2$-modules and a nonzero scalar $\alpha\in\Bbbk$, such that for any $x\in L_1, y\in\mathfrak{G}_1$, we have \[[x, y]=[x, \overline{y}]=\alpha[\phi(x), \overline{y}]=\alpha [\phi(x), y].\]
		In other words, $[L, \mathfrak{G}_{\bar 1}]=0$ for the submodule $L=\{x-\alpha\phi(x)\mid x\in L_1\}$, that contradicts to $K(\mathfrak{G}_{\bar 1})=0$.
		
		Let $s=n$. Combining Remark \ref{onto} with Lemma \ref{when pairing is trivial}(1), we conclude that  $[K(M_{(1)}), \mathfrak{G}_{\bar 1}/M_{(1)}]=0$,  which is a contradiction.   
		
		Now, assume that $s=n-2$. 
		In particular, $n\geq 2$. If $n=2$, then all weights of $\mathfrak{G}_{\bar 1}$ are less or equal to $2$, and since the modules $L(0), L(1), L(2)$ are tilting, the module $\mathfrak{G}_{\bar 1}$ is semi-simple, which is a contradiction. So, we have $n\geq 5$. 
		
		By Remark \ref{blocks} the block $B(n)$ is contained in  $(2+2p\mathbb{N})\sqcup ((p-4)+2p\mathbb{N})$.
		Since $n-2$ does not belong to $B(n)$, there is 
		$\mathfrak{G}_{\bar 1}=L(n-2)\oplus S$, where $C_{(1)}\subseteq S$ and $S/C_{(1)}\simeq L(n)$. If $C_{(1)}\simeq L(1)$, then $C_{(1)}=\mathrm{rad}(S)$ and combining Lemma \ref{Weyl module is mapped onto} with \cite[Theorem 2.25]{zub-bar}, we conclude that $n=p=3$ (and $S\simeq V(3)$), which contradicts to $p|(n-2)$. 
		If $C_{(1)}\simeq L(0)\oplus L(2)$, then $0\not\in B(n)$ implies $\mathfrak{G}_{\bar 1}=L(0)\oplus L(n-2)\oplus S'$, where $\mathrm{rad}(S')\simeq L(2)$ and $S'/\mathrm{rad}(S')\simeq L(n)$. Further, if $K(S')=0$ or $K(S')=\mathrm{rad}(S')$, then applying \cite[Theorem 2.25]{zub-bar} to the Harish-Chandra  pairs $(\mathrm{SL}_2, S')$ or $(\mathrm{SL}_2, S'/\mathrm{rad}(S'))$, we conclude $n=p=3$ or $n=1$, both contradict $n\geq 5$. Finally, let $K(S')=S'$. Then $K(L(n-2)\oplus S')=\mathrm{rad}(S')$ and apllying Lemma \ref{p|m and n=m+2} to the Harish-Chandra  pair $(\mathrm{SL}_2, L(n-2)\oplus S'/\mathrm{rad}(S'))$, we  obtain $n=2$, again a contradiction.
		
		It remains to consider the case $s=n+2$ and $p|s$. Applying Lemma \ref{when pairing is trivial}(2) to the (induced) nonzero morphism $K(M_{(1)})\otimes\mathfrak{G}_{\bar 1}/M_{(1)}\to\mathfrak{sl}_2$, we obtain a contradiction.

		B) Let $C_{(1)}=0$. Repeating the arguments of part A), we obtain $\mathfrak{G}_{\bar 1}=L(n)^{\oplus k_1}\oplus R$, where $R\cap M_{(1)}=L(n-2)^{\oplus k_2}\oplus L(n+2)^{\oplus k_3},
		R/(R\cap M_{(1)})\simeq L(n), 0\leq k_1, k_2, k_3\leq 1$ and $k_2+k_3>0$. By Remark \ref{blocks}, $n-2\in B(n)$ only if $p|n$. Similarly, $n+2\in B(n)$ only if $p|(n+2)$. Since $R$ is not semi-simple, one and only one of these alternatives takes place.
		
		Let $p|n$. Then $\mathfrak{G}_{\bar 1}=L(n)^{\oplus k_1}\oplus L(n+2)^{\oplus k_3}\oplus R'$, where $\mathrm{rad}(R')\simeq L(n-2)$ and $R'/\mathrm{rad}(R')\simeq L(n)$. 
		Note that $[\mathrm{rad}(R'), R']\neq 0$, otherwise $\mathrm{rad}(R')\subseteq K(\mathfrak{G}_{\bar 1})$. Thus $K(R')=0$ and combining Lemma \ref{Weyl module is mapped onto} with 
		\cite[Theorem 2.25]{zub-bar}, we also conclude that $n=p=3$ and $R'\simeq V(3)$. But $[\mathrm{rad}(R'), \mathrm{rad}(R')]=0$, that contradicts how the Lie bracket on $V(3)$ is defined
		in Lemma \ref{K_2=0}.
		
		Let $p|(n+2)$. Then $\mathfrak{G}_{\bar 1}=L(n)^{\oplus k_1}\oplus L(n-2)^{\oplus k_2}\oplus R'$, where $\mathrm{rad}(R')\simeq L(n+2), R'/\mathrm{rad}(R')\simeq L(n)$. 
		Since $[\mathrm{rad}(R'), \mathrm{rad}(R')]=0$ and $[R', R']\neq 0$, we have the induced nonzero morphism  \[\mathrm{rad}(R')\otimes R'/\mathrm{rad}(R')\simeq L(n+2)\otimes L(n)\to\mathfrak{sl}_2, \]
		hence again Lemma \ref{when pairing is trivial}(2) drives to a contradiction. 
	\end{proof}
	\begin{theorem}\label{Loewy length two}
		If $l=2$, then $\mathfrak{G}_1/K(\mathfrak{G}_1)$ is isomorphic to one of the following $\mathrm{SL}_2$-modules :  $0, L(1), L(0)\oplus L(2)$ or $V(3)$ (in the latter case $p=3$ also). Moreover, the Harish-Chandra pair structure on $(\mathrm{SL}_2, \mathfrak{G}_1/K(\mathfrak{G}_1))$ is described in Proposition \ref{M_{(1)}} and Lemma \ref{K_2=0}, respectively.	
	\end{theorem}
	\begin{proof}
		We use the induction on $\dim \mathfrak{G}_1$. Replacing $\mathfrak{G}_1$ by $\mathfrak{G}_1/K(\mathfrak{G}_1)$, one can suppose that $K(\mathfrak{G}_1)=0$ and the Loewy length remains two, otherwise Proposition \ref{M_{(1)}} concludes the proof. 
		
		By Lemma \ref{if K_2 is large}, $K_{(2)}+M_{(1)}$ is a proper submodule of $\mathfrak{G}_1$. 
		Next, if $K_{(2)}\not\subseteq M_{(1)}$, then let $S$ denote a submodule of $\mathfrak{G}_1$, such that $M_{(1)}=S_{(1)}$ and $S/M_{(1)}$ is a (nonzero!) direct complement to $(K_{(2)}+M_{(1)})/M_{(1)}$. Note that 
		$K(S)\subseteq K_{(2)}\cap S \subseteq M_{(1)}$ and thus the identity $\mathfrak{G}_1=K_{(2)}+S$ implies $K(S)\subseteq K(\mathfrak{G}_1)=0$. Since $\dim S<\dim \mathfrak{G}_1$, $S$ is isomorphic to $L(1), L(0)\oplus L(2)$ or $V(3)$. Furthermore, $S$ has the Loewy length two, therefore $p=3$ and $S\simeq V(3)$. Moreover, the structure of Harish-Chandra (sub)pair on $(\mathrm{SL}_2, S)$ is described in Lemma \ref{K_2=0}. In particular,  $M_{(1)}\simeq L(1)$ is the socle of $V(3)$, so that we have $[M_{(1)}, M_{(1)}]\neq 0$.  On the other hand, $0\neq K_{(2)}\cap M_{(1)}=K(M_{(1)})$ infers $K(M_{(1)})=M_{(1)}$, which is a contradiction.
		
		Now, let $K_{(2)}\subseteq M_{(1)}$, that is $K_{(2)}=K(M_{(1)})\neq 0$. If the number of direct simple terms of $\mathfrak{G}_1/M_{(1)}$ is at least two, then there are two proper
		submodules $S$ and $S'$, such that $M_{(1)}=S\cap S'$ and $S+S'=\mathfrak{G}_1$. By the above, $K(S)+K(S')\subseteq K(M_{(1)})$ and $K(S)\cap K(S')=0$.
		
		If $S/K(S)\simeq L(1)$, then $K(S)=M_{(1)}=K(M_{(1)})$. In particular, $[M_{(1)}, M_{(1)}]=0$. On the other hand, the inclusion $K(S')\subseteq M_{(1)}=K(S)$ infers $K(S')=0$. Arguing as above, we conclude
		that $S'\simeq V(3)$, which in turn implies $M_{(1)}\simeq L(1)$ and $[M_{(1)}, M_{(1)}]\neq 0$!
		
		Similarly, if $S/K(S)\simeq V(3)$, then $M_{(1)}=K(S)\oplus L$, where $L\simeq L(1)$ and $[L, L]\neq 0$ infers $K_{(2)}=K(S)$. Thus $K(S')\subseteq K(M_{(1)})=K(S)$, whence $K(S')=0$ again.
		Moreover, arguing as above, we conclude that $S'\simeq V(3)$. In particular, $M_{(1)}\simeq L(1)$ and $K_{(2)}=K(S)=0$, which contradicts our assumption. 
		
		Now, assume that $S/K(S)\simeq L(0)\oplus L(2)\simeq S'/K(S')$. It is clear that both $K(S)$ and $K(S')$ are proper submodules in $M_{(1)}$. 
		Considering the filtration
		\[0\subseteq K(S)\oplus K(S')\subseteq M_{(1)}\subseteq S,\]
		we immediately conclude that 
		$M_{(1)}=K(S)\oplus K(S')$.  In particular, each composition factor of $\mathfrak{G}_{\bar 1}$ are isomorphic either to $L(0)$ or to $L(2)$.  Combining all with the aforementioned fact that $L(0)$ and $L(2)$ are tilting modules, we conclude that $\mathfrak{G}_1$ is semi-simple, which is a contradiction.
		
		Finally, the case $\mathfrak{G}_1/M_{(1)}\simeq L(n)$ has already been considered in Proposition \ref{if top is simple and l=2}. 
	\end{proof}
	\begin{lm}\label{if W_2 is V(3)}
		Let $p=3$ and let $W$ be a $\mathrm{SL}_2$-module, such that $l(W)=3, W_{(2)}\simeq V(3)$ and $W/W_{(2)}\simeq L(1)$. There is no Harish-Chandra pair structure on $(\mathrm{SL}_2, W)$, such that $(\mathrm{SL}_2, W_{(2)})$ is a subpair isomorphic to that described in Lemma \ref{K_2=0}.	
	\end{lm}
	\begin{proof}
		We fix the basis $\{w_k\}_{k\in \{\pm 1, \pm 3\}}$ of $W_{(2)}$ as in Lemma \ref{K_2=0}. Similarly, let $u_1=s^*_0, u_{-1}=s^*_1$ be the standard basis of $W/W_{(2)}\simeq V(1)\simeq L(1)$. 
		Suppose that $p=3$ and $(\mathrm{SL}_2, W)$ is a Harish-Chandra pair, such that $(\mathrm{SL}_2, W_{(2)})$ is a subpair isomorphic to that described in Lemma \ref{K_2=0}.
		Since $W_{(1)}\simeq L(1)$ and $[W_{(1)}, W_{(1)}]\neq 0$, it follows that $K(W)=0$.
		
		It is clear that
		\[E_{1, -1}u_{-1}=-u_1+\alpha w_1, \ E_{1, -1}u_1=\beta w_3, \]
		\[E_{-1, 1}u_{-1}=\gamma w_{-3}, \ E_{-1, 1}u_1=-u_{-1}+\delta w_{-1}.\]
		Replacing $u_1$ by $u_1-\delta w_1$, one can assume that $\delta=0$.
		
		Using identity $[E_{1, -1}, E_{-1, 1}]=H$, we obtain
		\[-\alpha-\gamma=-\alpha+\beta=0.\]
		In particular, the parameter $\alpha$ is nonzero, otherwise $V=\Bbbk u_1+\Bbbk u_{-1}$ is a $\mathfrak{sl}_2$-submodule of $W$. Moreover, for any $k\geq 2$ we have
		$E_{1, -1}^{(k)}V=E_{-1, 1}^{(k)}V=0$, hence $\mathrm{Dist}(G)V=U^{[3]}(\mathfrak{sl}_2)V=V\simeq L(1)$ and $V\cap W_{(1)}=0$, which is a contradiction.   
		
		Set
		\[ [u_1, u_{-1}]=fH, \ [u_1, u_1]=gE_{1, -1}, \ [u_{-1}, u_{-1}]=h E_{-1, 1}.\]
		Consider an element $v=mu_1+nu_{-1}, m, n\in\Bbbk$. Then
		\[0=[[v, v], v]=[m^2 g E_{1, -1}-mnf H + n^2 h E_{-1, 1}, v]= \]
		\[\alpha m^3 g w_3-m^2n(f+g)u_1+\alpha g m^2n w_1-mn^2(f-h)u_{-1}-\alpha n^3 h w_{-3},\]
		which in turn implies
		\[f+g=f-h=\alpha g=\alpha h=0,\]
		and therefore, $f=g=h=0$.
		
		Similarly, we have
		\[[u_1, w_3]=0, \ [u_1, w_1]=bE_{1, -1}, \ [u_1, w_{-1}]=cH, \ [u_1, w_{-3}]=d E_{-1, 1},\]
		\[[u_{-1}, w_3]=d' E_{1, -1}, \ [u_{-1}, w_1]=c' H, \ [u_{-1}, w_{-1}]=b' E_{-1, 1}, \ [u_{-1}, w_{-3}]=0.\] 
		Then
		\[[[u_1, w_1], u_{-1}]=b(-u_1+\alpha w_1)=[u_1, [w_1, u_{-1}]]=c'u_1\]
		implies $b=c'=0$. Further, $0=[E_{-1, 1}, [u_1, w_1]]=-[u_{-1}, w_1]+[u_1, w_{-1}]$ entails $c=0$. Analogously, we have 
		$[E_{1, -1}, [u_1, w_{-3}]]=dH=-[u_1, w_{-1}]=0$, hence $d=0$.
		Summing all up, one sees that $u_1\in K(W)$, which contradicts the fact that $K(W)=0$. 
		
	\end{proof} 
	Below, Lemma \ref{K_l=0}, Proposition \ref{top is simple again} and Theorem \ref{general case} generalize Lemma \ref{K_2=0}, Proposition \ref{if top is simple and l=2} and Theorem \ref{Loewy length two}  respectively. We use the joint induction on the set of couples $(l, \dim\mathfrak{G}_{\bar 1})$ with respect to the lexicographical order. In particular, the latter three statements correspond to $l\leq 2$.
	
	\begin{lm}\label{K_l=0}
		If $l\geq 2$ and $K_{(l)}=0$, then $p=3$ and $\mathfrak{G}_{\bar 1}\simeq V(3)$. 	
	\end{lm}
	\begin{proof}
		Let $l\geq 3$.
		Then $K(M_{(l-1)})=K_{(l)}\cap M_{(l-1)}=0$, hence by induction $M_{(l-1)}$ satisfies the conclusion of Theorem \ref{general case}. In other words, $p=3$ and $M_{(l-1)}\simeq V(3)$. In particular, $l=3$. 
		
		Let $L(n)$ be a simple term of $\mathfrak{G}_{\bar 1}/M_{(2)}$. The positive integer $n$ is odd and $n\neq 3$. If $n> 3$, then let $v$ be an element of weight $n$, which generates $L(n)$ (modulo $M_{(2)}$). It is clear that $v$ is primitive, hence the $\mathrm{SL}_2$-submodule $R$, generated by $v$, is a highest weight one.
		Since $R\cap M_{(2)}\neq 0$, $R$ contains the simple socle $M_{(1)}$ of $M_{(2)}$, so that $[R, R]\neq 0$. It particular, $K(R)\subseteq M_{(2)}$  and $R/K(R)$ has the top isomorphic to $L(n)$, which contradicts Theorem \ref{general case} (valid for $R/K(R)$ by induction). 
		It remains to note that the case $n=1$ is impossible by Lemma \ref{if W_2 is V(3)}.
	\end{proof}
	\begin{pr}\label{top is simple again} 
		If $l\geq 2$, $K(\mathfrak{G}_{\bar 1})=0$ and $\mathfrak{G}_{\bar 1}/M_{(l-1)}$ is irreducible, then $p=3$ and $\mathfrak{G}_{\bar 1}\simeq V(3)$.	
	\end{pr}
	\begin{proof}
		Let $l\geq 3$. 	By Lemma \ref{if K_2 is large}, $K_{(l)}=K(M_{(l-1)})$ and by Lemma \ref{K_l=0}, one can suppose that $K_{(l)}\neq 0$. Let $\mathfrak{G}_{\bar 1}/M_{(l-1)}\simeq L(n)$.
		
		Assume that $K_{(l)}\neq M_{(l-1)}$.  Then $[M_{(l-1)}, M_{(l-1)}]\neq 0$ and by Lemma \ref{may be useful too} the algebra $\mathfrak{sl}_2$ acts trivially on $K_{(l)}$. 
		In particular, the highest weight of arbitrary composition factor of $K_{(l)}$ is multiple of $p$. 
		
		Note that for any simple term $R$ of $K_{(l)}\cap M_{(1)}$ we have $[R, \mathfrak{G}_{\bar 1}/M_{(l-1)}]\neq 0$. Using Lemma \ref{when pairing is trivial} and reasoning as in Proposition \ref{if top is simple and l=2} part $A)$, we derive $L(n-2)\simeq K_{(l)}\cap M_{(1)}\subseteq K(M_{(1)}), \ p|(n-2)$. In particular, all weights from $B(n)$ are coprime to $p$. Thus $M_{(1)}\not\subseteq K_{(l)}$, otherwise $\mathfrak{G}_{\bar 1}$ is indecomposable and $n-2$ belongs
		to $B(n)$, which is a contradiction.
		
		Let $\mathfrak{G}_{\bar 1}=S\oplus R$ be a block decomposition, where $S$ corresponds to $B(n)$ and $R$ is the sum of all other blocks. Then $S\cap K_{(l)}=0$ and $S$ is embedded into $\mathfrak{G}_{\bar 1}/K_{(l)}$. Additionally, $K_{(l)}\subseteq R\subseteq M_{(l-1)}$ and therefore, $l(S)=l\geq 3$.
		On the other hand, by induction, Theorem \ref{general case} can be applied to $M_{(l-1)}/K_{(l)}$, that is the latter is isomorphic to one of the following $\mathrm{SL}_2$-modules : $L(1), L(0)\oplus L(2)$ or $V(3)$. But in the first two cases $l(S)\leq l(\mathfrak{G}_{\bar 1}/K_{(l)})\leq 2$, which is a contradiction.  
		Therefore, $M_{(l-1)}/K_{(l)}\simeq V(3)$ and $p=3$. It remains to note that $1\in B(n)$ or $3\in B(n)$ only if $p=5$ or $p=7$ respectively, hence the embedding $S\to\mathfrak{G}_{\bar 1}/K_{(l)}$ is split, so that $\mathfrak{G}_{\bar 1}/K_{(l)}\simeq S\oplus M_{(l-1)}/K_{(l)}$ and $S\simeq L(n)$, a contradiction again. 		
		
		Now, let $[M_{(l-1)}, M_{(l-1)}]=0$. Since arbitrary simple term $R$ of $M_{(1)}$ satisfies $[R, \mathfrak{G}_{\bar 1}/M_{(l-1)}]\neq 0$, we have \[M_{(1)}\simeq L(n)^{k_1}\oplus L(n-2)^{k_2}\oplus L(n+2)^{k_3}, \ 0\leq k_1, k_2, k_3\leq 1, \ k_1+k_2+k_3>0.\]
		
		Let $M_{(1)}\simeq L(n)$. Set $v=t+m, t\in\mathfrak{G}_{\bar 1}, m\in M_{(1)}$. Using $(\circ\circ)$, we have 
		\[[[v, v], v]\equiv 2[[\overline{t}, m], t]\equiv \frac{(-1)^n na}{2}t\pmod{M_{(l-1)}},\]
		provided $\overline{t}=t+M_{(l-1)}$ and $m$ are represented by the elements $s_n^*$ and $s^*_0$ respectively. Thus $p|n$ and by Lemma \ref{when pairing is trivial}(1) there is
		$[M_{(1)}, \mathfrak{G}_{\bar 1}/M_{(l-1)}]=0$, which is a contradiction.
		
		Let $M_{(1)}\simeq L(n-2)$. Since $\mathfrak{G}_{\bar 1}$ is indecomposable, $n-2$ belongs to $B(n)$, that infers $p|n$ and Lemma \ref{when pairing is trivial}(2) drives to the contradictory statement $[M_{(1)}, \mathfrak{G}_{\bar 1}/M_{(l-1)}]=0$ again.  If $M_{(1)}\simeq L(n+2)$, then $p|(n+2)$ and this case is similar to the previous one.
		
		Assume that $M_{(1)}$ is not simple. Since the weights $n-2, n$ and $n+2$ can not belong to the same block simultaneously,  then either $\mathfrak{G}_{\bar 1}$ is indecomposable and $M_{(1)}$ contains only two simple terms, or there is a block decomposition $\mathfrak{G}_{\bar 1}=S\oplus R$ as above,
		where $S/(S\cap M_{(l-1)})\simeq L(n), R\subseteq M_{(l-1)}$, and therefore, $l(S)=l\geq 3$. Besides, $S_{(1)}\oplus R_{(1)}=L(n-2)\oplus L(n)\oplus L(n+2)$. In the second case, if $K(S)\neq 0$, then $K(S)\cap M_{(l-1)}\subseteq K(\mathfrak{G}_{\bar 1})=0$ implies $K(S)\simeq L(n)$ and $S\simeq K(S)\oplus M_{(l-1)}$, which in turn implies $l(S)\leq l-1$, a contradiction. Finally, if $K(S)=0$, then applying Theorem \ref{general case} to $S$, we obtain again a contradictory statement $l(S)\leq 2$.
		So, all we need is to consider the first case only. 
		
		So, let $M_{(1)}\simeq L(n-2)\oplus L(n)$. Since $\mathfrak{G}_{\bar 1}$ is indecomposable, then $n-2\in B(n)$ infers $p|n$ and again, 
		we obtain the contradictory statement $[L(n-2), \mathfrak{G}_{\bar 1}/M_{(l-1)}]=0$. Similarly, $M_{(1)}\simeq L(n)\oplus L(n+2)$ implies $p|(n+2)$ and we refer to the above arguments as well. Finally, the case $M_{(1)}\simeq L(n-2)\oplus L(n+2)$ is impossible, since $(n-2)$ and $(n+2)$ 
		can not belong to $B(n)$ simultaneously. 		
	\end{proof}
	\begin{theorem}\label{general case}
		If $l\geq 2$ and $K(\mathfrak{G}_{\bar 1})=0$, then $p=3$ and $\mathfrak{G}_{\bar 1}\simeq V(3)$. 	
	\end{theorem}
	\begin{proof}
		Let $l\geq 3$. By Lemma \ref{if K_2 is large}, $K_{(l)}+M_{(l-1)}$ is a proper submodule in $\mathfrak{G}_{\bar 1}$. By Lemma \ref{K_l=0}, one can assume that $K_{(l)}\neq 0$.
		
		Assume that $K_{(l)}\not\subseteq M_{(l-1)}$. Let $S$ be a submodule of $\mathfrak{G}_{\bar 1}$, such that $M_{(l-1)}=S_{(l-1)}$ and $S/M_{(l-1)}$ is a direct complement to
		$(K_{(l)}+M_{(l-1)})/M_{(l-1)}$ in $\mathfrak{G}_{\bar 1}/M_{(l-1)}$. We have $K(S)\subseteq K_{(l)}\cap S\subseteq M_{(l-1)}$, that being combined with $\mathfrak{G}_{\bar 1}=K_{(l)}+S$ implies $K(S)\subseteq K(\mathfrak{G}_{\bar 1})=0$. 
		Since $\dim S<\dim\mathfrak{G}_{\bar 1}$, by induction we conclude that $S\simeq V(3)$. Thus $l(S)=l=2$, which contradicts our assumption. 
		
		Now, let $K_{(l)}=K(M_{(l-1)})$. In particular, $[K_{(l)}, K_{(l)}]=0$, or equivalently, $K(K_{(l)})=K_{(l)}$. By Proposition \ref{top is simple again} one can suppose that $\mathfrak{G}_{\bar 1}/M_{(l-1)}$ is not irreducible. Therefore,  as in the proof of Theorem \ref{Loewy length two}, we can choose two proper submodules
		$S$ and $S'$, such that $S+S'=\mathfrak{G}_{\bar 1}$ and $S\cap S'=M_{(l-1)}$. Furthermore, $K(S)+K(S')\subseteq K(M_{(l-1)})=K_{(l)}$ and $K(S)\cap K(S')=0$.	
		
		If $l(S/K(S))\geq 2$, then by induction there is $p=3$ and $S/K(S)\simeq V(3)$. On the other hand, we have  $[K_{(l)}/K(S), K_{(l)}/K(S)]=0$, that implies $K_{(l)}=K(S)$ and therefore, $K(S')=0$. By induction, $S'\simeq V(3)$, that in turn implies $l=2$, which contradicts our assumption again. 
		
		Let $l(S/K(S))=l(S'/K(S'))=1$. If $S/K(S)\simeq L(1)$, then as above, we have  $K(S)=M_{(l-1)}$ and it infers $K(S')=0$, hence a contradictory statement $l=2$. Now, one can assume that $S/K(S)\simeq S'/K(S')\simeq L(0)\oplus L(2)$
		and both $K(S)$ and $K(S')$ are proper submodules of $M_{(l-1)}$. Since $K(S)\cap K(S')=0$, then as in Theorem \ref{Loewy length two} we conclude that $M_{(l-1)}=K(S)+K(S')$, so that each composition factor of $\mathfrak{G}_{\bar 1}$ is isomorphic to $L(0)$ or $L(2)$, and hence $\mathfrak{G}_{\bar 1}$ is semi-simple. This contradiction completes the proof. 
	\end{proof}
	The results of this section can be summarized as follows. 
	\begin{theorem}\label{or ... }
		Let $\mathbb{G}$ be a supergroup with $G\simeq\mathrm{SL}_2$. Then the Harish-Chandra pair of $\mathbb{H}=\mathbb{G}/\mathrm{R}_u(\mathbb{G})$ is $(\mathrm{SL}_2, \mathfrak{H}_{\bar 1})$, where $\mathfrak{H}_{\bar 1}$ is isomorphic to one of the following $\mathrm{SL}_2$-modules : $0, \ L(1), \ L(0)\oplus L(2), \ V(3)$. Moreover, in the first case $\mathbb{G}$ is split and in the last case $p=3$ and the supergroup $\mathbb{H}$ is almost-simple in the sense of \cite{zub-bar}.	The rest cases were discussed in Corollary \ref{description in char=0}.	
	\end{theorem}
	The supergroups from Theorem \ref{or ... }, whose Harish-Chandra pairs are isomorphic to $(\mathrm{SL}_2, L(0)\oplus L(2))$ or $(\mathrm{SL}_2, V(3))$, are denoted by $\mathbb{H}(0\oplus 2)$ and $\mathbb{H}(3/1)$ respectively. It is worth noting that $\mathbb{H}(3/1)\simeq\mathrm{Aut}(\mathfrak{spo}(2|1))^0$ (see \cite[Theorem 2.25]{zub-bar}).
	
	\section{Supergroups with the even part isomorphic to $\mathrm{PSL}_2$}
	
	The catergory of $\mathrm{PSL}_2$-modules is naturally identified with the full subcategory of $\mathrm{SL}_2$-modules, on which the central element $(-I_2)$ acts trivially. Equivalently, $(-I_2)$ acts trivially on $\mathrm{SL}_2$-module $M$ if and only if any composition factor of $M$ has a form $L(2n), n\in\mathbb{N}$. Combining with Theorem \ref{or ... }, we obtain the following.
	\begin{theorem}\label{general for PSL_2}
		Let $\mathbb{G}$ be a supergroup with $\mathbb{G}_{ev}=G\simeq\mathrm{PSL}_2$. Then the Harish-Chandra pair of $\mathbb{H}=\mathbb{G}/\mathrm{R}_u(\mathbb{G})$ is $(\mathrm{PSL}_2, \mathfrak{H}_{\bar 1})$, where $\mathfrak{H}_{\bar 1}$ is isomorphic to one of the following $\mathrm{PSL}_2$-modules : $0, \ L(0)\oplus L(2)$. 
	\end{theorem}	
	Note that in the second case the supergroup $\mathbb{H}$ is isomorphic to $\mathbb{H}(0\oplus 2)/\mu_2$.
	
	\section{Supergroups with the even part isomorphic to $\mathrm{GL}_2$}
	
	\subsection{Special cases}	
	
	Recall that the irreducible $\mathrm{GL}_2$-modules are uniquely defined by their highest weights, which are (ordered) partitions $\lambda=(\lambda_1, \lambda_2), \lambda_1\geq\lambda_2$. The set of such (dominant) weights is a poset with respect to the partial order : 
	$\mu\leq\lambda$ if $\mu_1\leq\lambda_1$ and $\lambda_1+\lambda_2=\mu_1+\mu_2$.
	
	For any (not necessary dominant) weight $\lambda$ let $|\lambda|$ denote $\lambda_1+\lambda_2$.
	
	We have $L(\lambda)\simeq L((n, 0))\otimes \det^{\lambda_2}$, where $n=\lambda_1-\lambda_2$ and $L((n, 0))|_{\mathrm{SL}_2}\simeq L(n)$ . In particular, a $\mathrm{GL}_2$-module $M$ is irreducible if and only if $M|_{\mathrm{SL}_2}$ is. For example, $\mathfrak{sl}_2$, regarded as a $\mathrm{GL}_2$-module with respect to the adjoint action, is isomorphic to $L((1, -1))\simeq L((2, 0))\otimes \det^{-1}$. Recall also that $L(\lambda)^*\simeq L(\lambda^*)$, where $\lambda^*=-w_0\lambda=(-\lambda_2, -\lambda_1)$. In particular, $L(\lambda^*)\simeq L((n, 0))\otimes\det^{-\lambda_1}$ and $L(\lambda)$ is self-dual if and only if $|\lambda|=0$. 
	
	Similarly, we have \[H^0(\lambda)\simeq H^0((n, 0))\otimes\mathrm{det}^{\lambda_2}=\mathrm{Sym}_n(V)\otimes\mathrm{det}^{\lambda_2}\] and 
	\[V(\lambda)=H^0(\lambda^*)^*\simeq (H^0((n, 0))\otimes\mathrm{det}^{-\lambda_1})^*\simeq \mathrm{Sym}_n(V)^*\otimes \mathrm{det}^{\lambda_1}.\]
	
	Let $Z\simeq G_m$ denote the center of $\mathrm{GL}_2$ consisting of all scalar matrices. Then any $\mathrm{GL}_2$-module $M$ has a weight decompostion $M=\oplus_{k\in\mathbb{Z}} M_k$ with respect to the action of $Z$. Moreover, each component $M_k$ is obviously a $\mathrm{GL}_2$-submodule. We also have $L(\lambda)=L(\lambda)_{|\lambda|}$.
	
	The Lie subalgebra $\mathrm{Lie}(Z)=\mathfrak{z}=\Bbbk I_2$ is a trivial $\mathrm{GL}_2$-submodule of $\mathfrak{gl}_2$. Besides, $\mathfrak{gl}_2=\mathfrak{z}\oplus \mathfrak{sl}_2$. 
	
	Let $\mathbb{G}$ be an algebraic  supergroup, such that $\mathbb{G}_{ev}=G\simeq\mathrm{GL}_2$. As above, for any $\mathrm{GL}_2$-submodule $S$ of $\mathfrak{G}_{\bar 1}$ let $K(S)$ denote the largest subspace of $S$, such that
	$[K(S), S]=0$. It is clear that $K(S)$ is a $\mathrm{GL}_2$-submodule as well. 
	
	For the sake of simplicity, assume that $K(\mathfrak{G}_{\bar 1})=0$, that is $\mathbb{G}$ is reductive.
	
	To simplify our notations, let $M_k$ denote $(\mathfrak{G}_{\bar 1})_k, k\in\mathbb{Z}$. Then $[M_k, M_l]\neq 0$ implies $k+l=0$. In particular, for any integer $k$ the component $M_k$ 
	is nonzero if and only if $M_{-k}$ is if and only if $[M_k, M_{-k}]\neq 0$. Furthermore, for any $\mathrm{GL}_2$-submodule $R$ of $M_k$ we have $[R, M_{-k}]\neq 0$.
	\begin{rem}\label{submodules of gl_2}
		Any nonzero $\mathrm{GL}_2$-submodule of $\mathfrak{gl}_2$ is equal to one of the following submodules : $\mathfrak{z}, \mathfrak{sl}_2$ or $\mathfrak{gl}_2$.
		In fact, if it is contained neither in $\mathfrak{z}$ nor in $\mathfrak{sl}_2$, then it contains an element $x=z+s, z\in \mathfrak{z}\setminus 0, s\in\mathfrak{sl}_2\setminus 0$.
		There is $g\in\mathrm{GL}_2(\Bbbk)$, such that $gsg^{-1}\neq s$ and thus our submodule contains $x-gxg^{-1}=s-gsg^{-1}\neq 0$, hence it entirely contains $\mathfrak{sl}_2$.
	\end{rem}
	Let $I=\{k\in\mathbb{Z}\mid [M_k, M_{-k}]\neq 0  \}$. Then $I=I_1\sqcup I_2\sqcup I_3$, where
	\[I_1=\{k\in I \mid [M_k, M_{-k}]=\mathfrak{z}\}, \ I_2=\{k\in I \mid [M_k, M_{-k}]=\mathfrak{sl}_2\}, \]
	\[I_3=\{k\in I \mid [M_k, M_{-k}]=\mathfrak{gl}_2\}.\]
	In what follows the condition $0|k$ means $k=0$.
	\begin{pr}\label{when [g, g] is  z}
		Assume that $[\mathfrak{G}_{\bar 1}, \mathfrak{G}_{\bar 1}]=\mathfrak{z}$ (or equivalently, $I_2=I_3=\emptyset$). Then for each $k$, such that $M_k\neq 0$, there is $p|k$ and $M_k^*\simeq M_{-k}$. In particular, if $p=0$, then $I_1=\{0\}$. Furthermore, 
		the Harish-Chandra pair structure on $\mathfrak{G}_{\bar 1}$ is determined as
		\[ [x, y]=<x, y> I_2, x\in M_k, y\in M_{-k},\]
		where $< \ , \ >$ is a natural bilinear map $M_k\times M_{-k}\to\Bbbk$, determined by the above mentioned isomorphism $M_k^*\simeq M_{-k}$ of $\mathrm{GL}_2$-modules.	
	\end{pr}
	\begin{proof}
		Just mimic the proof of \cite[Proposition 2.26]{zub-bar}. 	
	\end{proof}
\begin{cor}\label{p| I_1} 
There always is $I_1\subseteq p\mathbb{Z}$. 	
\end{cor}
\begin{proof}
By the above remark, $K(\oplus_{k\in I_1} M_k)=0$. It remains to apply Proposition \ref{when [g, g] is  z} to the Harish-Chandra subpair $(\mathrm{GL}_2, \oplus_{k\in I_1} M_k)$ (or, to the corresponding supersubgroup of $\mathbb{G}$).
\end{proof}
	\begin{rem}\label{if in the above char=0}
		If a supergroup $\mathbb{G}$ satisfies the condition of Proposition \ref{when [g, g] is  z}, then the subpair $(Z, \mathfrak{G}_{\bar 1})$ represents $\mathrm{R}(\mathbb{G})$. Indeed, if $\mathbb{H}$ is the normal supersubgroup of $\mathbb{G}$, represented by this subpair, then $\mathbb{H}$ is obviously smooth and connected. It remains to note that $\mathbb{H}$ is solvable by \cite[Corollary 6.4]{maszub1} and
		$\mathbb{G}/\mathbb{H}\simeq\mathrm{PGL}_2$ is semi-simple. 
	\end{rem}
	\begin{pr}\label{if [g, g] is sl_2}
		Assume that $[\mathfrak{G}_{\bar 1}, \mathfrak{G}_{\bar 1}]=\mathfrak{sl}_2$ (or equivalently, $I_1=I_3=\emptyset$). Then \[\mathfrak{G}_{\bar 1}\simeq L((-t, -t))\oplus L((t+1, t-1)), t\in\mathbb{Z}.\]
		Besides, we have
		\[[L((-t, -t)), L((-t, -t))]=[L((t+1, t-1)), L((t+1, t-1))]=0 \]
		and 
		\[[x\otimes \mathrm{det}^{-t}, y\otimes\mathrm{det}^{t}]=axy, \ x\in L((0, 0))\simeq\Bbbk, \ y\in L((1, -1)\simeq\mathfrak{sl}_2, \]
		for some scalar $a\in\Bbbk\setminus 0$.   		
	\end{pr}
	\begin{proof}
		It is obvious that $(\mathrm{SL}_2, \mathfrak{G}_{\bar 1})$ is a subpair of  $(\mathrm{GL}_2, \mathfrak{G}_{\bar 1})$, which determines a normal supersubgroup $\mathbb{H}$, such that
		$\mathbb{G}/\mathbb{H}\simeq\mathbb{G}_m$ is the one-dimensional torus. If $\mathfrak{G}_{\bar 1}|_{\mathrm{SL}_2}$ is indecomposable, then Theorem \ref{or ... } implies that
		$\mathfrak{G}_{\bar 1}|_{\mathrm{SL}_2}$ is isomorphic to $L(1)$ or $p=3$ and $\mathfrak{G}_{\bar 1}|_{\mathrm{SL}_2}\simeq V(3)$. 
		
		In both cases $\mathfrak{G}_{\bar 1}=M_0$ and the socle of $\mathfrak{G}_{\bar 1}$ is isomorphic to $L(\lambda)$, where $\lambda=(\lambda_1, \lambda_2)$ satisfies $\lambda_1-\lambda_2=1, \lambda_1+\lambda_2=0$, that is obviously inconsistent.  
		
		It remains to consider the case $\mathfrak{G}_{\bar 1}|_{\mathrm{SL}_2}\simeq L(0)\oplus L(2)$. If $\mathfrak{G}_{\bar 1}=M_0$, then $\mathfrak{G}_{\bar 1}\simeq L((0, 0))\oplus L((1, -1))$. Otherwise, $\mathfrak{G}_{\bar 1}\simeq L(\mu)\oplus L(\lambda)$, where $|\mu|=-k, |\lambda|=k$ for some integer $k\neq 0$. Besides, $\lambda_1-\lambda_2=2$ and $\mu_1-\mu_2=0$. 
		Thus $k=2t$ and $\lambda=(t+1, t-1), \mu=(-t, -t)$.	Combining isomorphisms $L((-t, -t))\simeq L((0, 0))\otimes\det^{-t}$ and $L((t+1, t_1))\simeq L((1, -1))\otimes\det^t$ with the final remark in Lemma \ref{p|m and n=m+2}, we conclude the proof.
	\end{proof}
	Let $\mathbb{K}(t)$ denote the supergroup from Proposition \ref{if [g, g] is sl_2}. A little bit later we will describe the isomorphism classes of supergroups $\mathbb{K}(t)$ for various $t$.
	
	The following lemma is obvious.
	\begin{lm}\label{a qubic trick}
Let $\mathfrak{L}$ be a Lie superalgebra. Let $V_1$ and $V_2$ be homogeneous supersubspaces in $\mathfrak{L}$ (not necessarily of the same parity), such that 
$[V_1, V_2]=0$. Then $[[V_1, V_1], V_2]=[[V_2, V_2], V_1]=0$. 		
	\end{lm}
	\begin{pr}\label{[g, g] is gl_2}
		Let $[\mathfrak{G}_{\bar 1}, \mathfrak{G}_{\bar 1}]=\mathfrak{gl}_2$. Then only one of the following alternaives holds :
		\begin{enumerate}
			\item $I_3=\emptyset$, $I_2=\{\pm l \}$ and $I\subseteq p\mathbb{Z}$, where $p>0$. Moreover, $\oplus_{k\in I_1} M_k$ is a trivial $\mathfrak{gl}_2$-module and the restriction of Lie bracket on it is described in Proposition \ref{when [g, g] is  z}, as well as the restriction of Lie bracket on $M_l+M_{-l}$ is described in Proposition \ref{if [g, g] is sl_2}.
			\item $I_3\neq\emptyset$, then $I_3=\{\pm k\}$, $I_2=\emptyset$ and if $I_1\neq\emptyset$, then $I\subseteq p\mathbb{Z}$, where $p>0$. Besides, $\oplus_{l\in I_1} M_l$ is a trivial $\mathfrak{gl}_2$-module. 
		\end{enumerate}
	\end{pr}	
	\begin{proof}	
		First, assume that $I_3=\emptyset$. Then $I_1\neq\emptyset$ and $I_2\neq\emptyset$. Applying Lemma \ref{a qubic trick} to $V_1=\oplus_{k\in I_1} M_k$ and $V_2=\oplus_{k\in I_2} M_k$, we obtain that $I_2\subseteq p\mathbb{Z}$ and $V_1$ is trivial as a $\mathfrak{sl}_2$-module, hence as a $\mathfrak{gl}_2$-module as well. 
		Again, $K(V_2)=0$ and using Proposition \ref{if [g, g] is sl_2} we conclude that $I_2=\{\pm l\}$ and $M_l+ M_{-l}\simeq L((-t, -t)\oplus L((t+1, t-1))$.

		Now, assume that $I_3\neq\emptyset$. Let $V_3$ denote $\oplus_{k\in I_3} M_k$. As above, applying Lemma \ref{a qubic trick} to $V_2$ and $V_3$, we conclude that $V_2$ is a trivial
		$\mathfrak{gl}_2$-module, and combining Lemma \ref{g-invariants} with Proposition \ref{if [g, g] is sl_2} we obtain $V_2=0$, or, equivalently, $I_2=\emptyset$. Similarly, applying Lemma \ref{a qubic trick} to $V_1$ and $V_3$, we derive $I_3\subseteq p\mathbb{Z}$, provided $I_1\neq \emptyset$, and $V_1$ is a trivial $\mathfrak{gl}_2$-module. Finally, if $|I_3|> 2$, then $V_3$ is a direct sum of two proper $G$-submodules, say $U$ and $W$, such that $[U, W]=0$. Thus each of them is a trivial $\mathfrak{gl}_2$-module, which in turn implies $[U, U]=[V, V]=\mathfrak{z}$, a contradiction.
	\end{proof}
	This proposition shows that all we need is to consider the case $\mathfrak{G}_{\bar 1}=M_k+M_{-k}$, where $[M_k, M_{-k}]=\mathfrak{gl}_2$. In this subsection we consider the case $k=0$ only. 
	
	First of all, note that $Z$ is the central (super)subgroup in $\mathbb{G}$, so that $\mathbb{G}/Z$ has the Harish-Chandra pair $(\mathrm{GL}_2/Z, \mathfrak{G}_{\bar 1})\simeq (\mathrm{PGL}_2, \mathfrak{G}_{\bar 1})$. Let $U\subseteq\mathfrak{G}_{\bar 1}$ denote the kernel of the latter pair, that is $U=\{x\in \mathfrak{G}_{\bar 1}\mid [x, \mathfrak{G}_{\bar 1}]\subseteq\mathfrak{z}\}$. Then Theorem \ref{general for PSL_2} implies $\mathfrak{G}_{\bar 1}/U\simeq L((0, 0))\oplus L((1, -1))$. 
	The Lie bracket induces a bilinear pairing $U\times\mathfrak{G}_{\bar 1}\to\mathfrak{z}\simeq\Bbbk$ (of $\mathrm{GL}_2$-modules) and let $R$ denote its right kernel. Then the induced pairing
	$U\times \mathfrak{G}_{\bar 1}/R\to\Bbbk$ is non-degenerate, hence $U\simeq (\mathfrak{G}_{\bar 1}/R)^*$ and $\dim R=\dim \mathfrak{G}_{\bar 1}-\dim U=4$.  Similarly, we have the pairing \[(R\cap U)\times \mathfrak{G}_{\bar 1}/(U+R)\to\Bbbk,\]
	whose the left kernel iz zero. Since $\dim(R\cap U)=\dim(\mathfrak{G}_{\bar 1}/(U+R))$, this pairing is non-degenerate also, hence $R\cap U\simeq (\mathfrak{G}_{\bar 1}/(U+R))^*$. Further, combining $R/(R\cap U)\simeq (U+R)/U$ with the fact that both irreducible terms of $\mathfrak{G}_{\bar 1}/U$ are self-dual, we obtain $R\simeq L((0, 0))\oplus L((1, -1))$.  	
	\begin{pr}\label{GL_2 and k=0}
		If $\mathfrak{G}_{\bar 1}=M_0$ and $R\cap U=0$, then $\mathfrak{G}_{\bar 1}=R\oplus U$, where $R\simeq\mathfrak{gl}_2$ regarded as a $\mathrm{GL}_2$-module with respect to the adjoint action. Moreover, the Lie bracket on
		$R$ is given by 
		\[[\alpha I_2+x, \beta I_2+y]=(\alpha\beta a+c\mathrm{tr}(xy)) I_2+ b(\beta x+\alpha y), \ x, y\in\mathfrak{sl}_2, \ \alpha, \beta\in\Bbbk,\]
		for some scalars $a, c\in\Bbbk, b\in\Bbbk\setminus 0$ and $a$ and $c$ are not zero simultaneously. Finally,  $U$ is a trivial $\mathfrak{gl}_2$-module and self-dual as a $\mathrm{GL}_2$-module, so that the Lie bracket on $U$ is defined as in Proposition \ref{when [g, g] is  z}.	 
	\end{pr}
	\begin{proof}
		The first two statement and, by applying Lemma \ref{a qubic trick}, the last one are obvious. By the final remark in Lemma \ref{p|m and n=m+2}, the Lie bracket on $R$ is defined as
		\[ [I_2, I_2]=a I_2, \ [I_2, x]= b x+f(x) I_2, \ [x, y]= g(x, y) I_2, \ x, y\in \mathfrak{sl}_2\simeq L((1, -1)), \]	
		where $b\in\Bbbk\setminus 0$, $f : L((1, -1))\to \Bbbk$ is a morphism of $\mathrm{GL}_2$-modules, thus obviously zero,  and $g : L((1, -1))\times L((1, -1))\to\Bbbk$ is a symmetric bilinear $\mathrm{GL}_2$-equivariant map. 
		
		Let $W$ denote $L((1, -1))$. Since $W\simeq W^*$, Schur lemma implies
		\[\mathrm{Hom}_{\mathrm{GL}_2}(W^{\otimes 2}, \Bbbk)\simeq \mathrm{Hom}_{\mathrm{GL}_2}(W, W)\simeq\Bbbk. \]
		On the other hand, $W^{\otimes 2}\simeq\mathrm{Sym}_2(W)\oplus\Lambda^2(W)$ and $\Lambda^2(W)\simeq W^*\simeq W$ infer
		\[\mathrm{Hom}_{\mathrm{GL}_2}(W^{\otimes 2}, \Bbbk)=\mathrm{Hom}_{\mathrm{GL}_2}(\mathrm{Sym}_2(W), \Bbbk).\]
		Therefore, up to a nonzero scalar there is only one such map $g$, and we obtain
		\[ g(x, y)=c\mathrm{tr}(xy), \ c\in\Bbbk, \ x, y\in\mathfrak{sl}_2,\]
		where $\mathrm{tr}$ is the trace function. 
	\end{proof}
\begin{rem}\label{G as a product}
As in Remark \ref{if in the above char=0}, one easily sees that the Harish-Chandra subpair $(Z, U)$ represents $\mathrm{R}(\mathbb{G})$. Let $\mathbb{H}$ denote the supersubgroup of $\mathbb{G}$ corresponding to the Harish-Chandra subpair $(\mathrm{GL}_2, R)$. Then $\mathbb{G}=\mathbb{H}\mathrm{R}(\mathbb{G}), \mathbb{H}\cap\mathrm{R}(\mathbb{G})=Z$ and $\mathbb{H}$ is reductive.
\end{rem}
	\begin{rem}\label{when G is Q(2)}
		Without loss of generality, one can assume that $b=1$. Then the structure of the supersubgroup $\mathbb{H}$ depends on two parameters $a, c$. In accordance with this remark, we denote it by
		${\bf Q}(2; a, c)$. The meaning of this notation is clear from the following. 
		
		Recall that for any matrices $x, y\in\mathfrak{sl}_2$, there holds $xy+yx=\mathrm{tr}(xy) I_2$. In particular, if $a=2c=1$, then the above formula can be transformed as	
		\[[u, v]=\frac{1}{2}(uv+vu), \ u, v\in\mathfrak{gl}_2, \]
		that is ${\bf Q}(2; 1, \frac{1}{2})\simeq {\bf Q}(2)$. In other words, a supergroup ${\bf Q}(2; a, c)$ is a "deformation" of ${\bf Q}(2)$.
	\end{rem}
	Now, consider the case $R\cap U\neq 0$, but $R\not\subseteq U$. Let $R'\subseteq R$ be an irreducible complement to the simple $\mathrm{GL}_2$-submodule $R\cap U$. In particular, $U+R=U\oplus R'$.
	Let $L$ be a $\mathrm{GL}_2$-submodule of $\mathfrak{G}_{\bar 1}$, such that $L/U$ is a complement to $(U+R)/U$. Then $\mathfrak{G}_{\bar 1}=R'\oplus L$ and by Theorem \ref{general for PSL_2}, $[R', R'], [L, L]\subseteq \mathfrak{z}$, and $[R', L]=\mathfrak{gl}_2$. Besides, $R'\not\simeq L/U$.		
	\begin{pr}\label{if R intersects V but not contained in}
The $\mathrm{GL}_2$-modules $R'$ and $L$ are self-dual, so that the Lie bracket on them is defined as in Proposition \ref{when [g, g] is  z}. We also have :
\begin{enumerate}
	\item $U$ is the trivial $\mathfrak{gl}_2$-module;
		\item $R'\simeq L((1, -1))$ and $L/U\simeq L((0, 0))\simeq\Bbbk$. For any $r\in R', l\in L$, there is $[r, l]=[r, \bar l]=\bar l r$, where $\bar l=l+U$ is identified with a scalar in $\Bbbk$.
\end{enumerate}
	\end{pr}
\begin{proof}
	Theorem \ref{general for PSL_2} infers $\mathfrak{sl}_2\subseteq [R', L]$. Then
	\[[[R', L], U]\subseteq [R', [L, U]]+[L, [R', U]]\subseteq [\mathfrak{z}, R']=0,\]
	and $(1)$ follows. 
	
Note that $[R', R']=[L, L]=\mathfrak{z}$. Indeed, $[L, L]=0$ implies $L\subseteq R$, which is obviously nonsense. If $[R', R']=0$, then $R'\times \mathfrak{G}_{\bar 1}/(U+R)\to\Bbbk$ is a non-degenerate pairing, hence $R'\simeq L/U$, that is impossible by the above remark. Moreover, the induced pairing $L\times L\to\Bbbk$ is non-degenerate, otherwise its kernel coincides with $R\cap U$, which contradicts the fact established above that the induced pairing $(R\cap U)\times L/U\to\Bbbk$ is non-degenerate.

If $R'\simeq L((0, 0))$, then $L/U\simeq L((1, -1))\simeq R\cap U$. On the other hand, by Lemma \ref{g-invariants} all composition factors of $U$ have a form $L(p\lambda)$, which  contradicts $p\nmid (1, -1)$. 
The proof of $(2)$ is similar to Proposition \ref{GL_2 and k=0} and we leave it for the reader. 
\end{proof}
\begin{pr}\label{R is not contained in V}
The case $R\subseteq U$ is impossible. 
\end{pr}
\begin{proof}
Note that the weights $(0, 0)$ and $(1, -1)$ belong to the different blocks (use \cite[II.7.2(1)]{jan}). Let $\mathfrak{G}_{\bar 1}=L_1\oplus L_2\oplus S$, where $L_1$ and $L_2$ correspond to the blocks of $(0, 0)$ and $(1, -1)$ respectively, and $S$ is a sum of all other block terms. Then \[U=(U\cap L_1)\oplus (U\cap L_2)\oplus S\] and \[L_1/(U\cap L_1)\simeq L((0, 0)), L_2/(U\cap L_2)\simeq L((1, -1)).\] 
By Theorem \ref{general for PSL_2}, $\mathfrak{sl}_2\subseteq [L_1, L_2]$, and combining with
\[[[L_1, L_2], U]\subseteq [L_1, [L_2, U]]+[L_2, [L_1, U]]\subseteq [\mathfrak{z}, L_1+L_2]=0,\]
we conclude that $U$ is a trivial $\mathfrak{gl}_2$-module. As above, Lemma \ref{g-invariants} infers $U\cap L_2=0$, which contradicts the fact that $U$ contains $L((1, -1))$.
\end{proof}
	\begin{pr}\label{when Q(a, c) is isomorphic to Q(a', c')?}
		${\bf Q}(2; a, c)\simeq {\bf Q}(2; a', c')$ if and only if there is $\alpha\in\Bbbk\setminus 0$, such that $a=\alpha a', c=\alpha^{-1} c'$.	
	\end{pr}
	\begin{proof}
		Let $(f, u)$ be an isomorphism of Harish-Chadra pairs, which induces an isomorphism ${\bf Q}(2; a, c)\simeq {\bf Q}(2; a', c')$. Without loss of generality, one can assume that
		$f\in\mathrm{Aut}(\mathrm{GL}_2)$. Recall that $\mathrm{Aut}(\mathrm{GL}_2)$ is generated by inner automorphisms (induced by elements $g\in\mathrm{GL}_2(\Bbbk)$) and the \emph{graph automorphism} $\tau : x\mapsto (x^t)^{-1}$(cf. \cite[Corollary 23.47]{milne}). Note that $\tau^2=\mathrm{id}_{\mathrm{GL}_2}$.
		
		If $f$ is an inner automorphism or a product of an inner automorphism and the graph automorphism $\tau$, then $\mathrm{d}_e(f)(x)=gxg^{-1}$ or $\mathrm{d}_e(f)(x)=-gx^tg^{-1}, x\in \mathfrak{gl}_2$, respectively. In both cases $\mathrm{tr}(xy)=\mathrm{tr}(\mathrm{d}_e(f)(x)\mathrm{d}_e(f)(y))$ for arbitrary $x, y\in\mathfrak{gl}_2$.  
		
		Since $\Bbbk I_2$ and $\mathfrak{sl}_2$ are non-isomorphic simple $\mathrm{GL}_2$-modules with respect to the both adjoint and twisted (by the automorphism $f$) adjoint actions, the linear map $u$ sends  $\Bbbk I_2$ and $\mathfrak{sl}_2$ to themselves. Set $u(I_2)=\alpha I_2, \alpha\in\Bbbk\setminus 0$.    
		
		Let $f=\tau^k\mathrm{Ad}(g)\in \mathrm{Aut}(\mathrm{GL}_2), k\in\{0, 1\}, g\in\mathrm{GL}_2(\Bbbk)$. For any $x, y\in\mathfrak{sl}_2, s, t\in\Bbbk$, we have  
		\[ [\alpha s I_2+u(x), \alpha t I_2+u(y) ]=(\alpha^2 st a'+c'\mathrm{tr}(u(x)u(y)))I_2 +\alpha(tu(x)+su(y))=\]
		\[\mathrm{d}_e(f)((sta+c\mathrm{tr}(xy))I_2+(tx+sy) )=(-1)^k (sta+c\mathrm{tr}(xy))I_2+(t\mathrm{d}_e(f)(x)+s\mathrm{d}_e(f)(y))\]
		and varying $s, t$, one easily sees that
		\[\alpha^2a'=(-1)^k a, \ c'\mathrm{tr}(u(x)u(y))=(-1)^k c\mathrm{tr}(xy),\] \[ u(x)=\alpha^{-1}\mathrm{d}_e(f)(x), u(y)=\alpha^{-1}\mathrm{d}_e(f)(y).\]
		By the above remark, $\mathrm{tr}(u(x)u(y))=\alpha^{-2}\mathrm{tr}(xy)$ and therefore, $\alpha^{-2}c'=(-1)^k c$.   
		
		It remains to note that $\mathrm{d}_e(f)(hxh^{-1})=f(h)\mathrm{d}_e(f)(x)f(h)^{-1}$ fo any $h\in\mathrm{GL}_2, x\in\mathfrak{sl}_2$.
	\end{proof}
	\begin{rem}\label{determinant-based}
		The group $\mathrm{Aut}(\mathrm{GL}_2)$ contains so-called {\bf determinant-based} automorphism $\sigma : h\mapsto \det(h)^{-1} h$, which satisfies the relation $\sigma\tau=\mathrm{Ad}(g)$, where
		\[g=\left(\begin{array}{cc}
			0 & 1 \\
			-1 & 0
		\end{array}\right).\]	
		The automorphisms $\sigma$ and $\tau$ commute with each other. 
		Furthermore, for any $n\in\mathbb{Z}$ one can define a group scheme endomorphism $\sigma_n : h\mapsto \det(h)^n h$. However, $\sigma_n$ is an automorphism if and only if $n=-1$ or $n=0$. 
	\end{rem}
	\begin{rem}\label{in general}
		Similarly to Corollary \ref{description in char=0}, one sees that $\mathbb{G}\simeq{\bf Q}(2; a, c)$ contains a normal supersubgroup $\mathbb{H}$, that corresponds to the Harish-Chandra subpair $(\mathrm{GL}_2, \mathfrak{sl}_2)$, such that $\mathbb{G}/\mathbb{H}\simeq\mathbb{G}_a^{-}$. Moreover, the center $Z$ of $G$ is contained in $\mathbb{H}$ and $\mathbb{H}/Z\simeq
		\mathrm{PGL}_2\ltimes (\mathbb{G}_a^-)^3$. One can also show that $\mathbb{H}=\mathcal{D}(\mathbb{G})$ is the derived supersubgroup of $\mathbb{G}$.
	\end{rem}
	\begin{pr}\label{isomorphism classes of K(t)}
		$\mathbb{K}(t)\simeq\mathbb{K}(t')$ if and only if $t=\pm t'$.	
	\end{pr}
	\begin{proof}
		Following the proof of Proposition \ref{when Q(a, c) is isomorphic to Q(a', c')?} and using the notations therein, we have
		\[u(x\otimes\mathrm{det}^t)=u(x)\otimes\mathrm{det}^{t'}, \ u(b\otimes\mathrm{det}^{-t})=\alpha b\otimes\mathrm{det}^{-t'}, x\in\mathfrak{sl}_2, b\in\Bbbk, \]
		for some $\alpha\in\Bbbk\setminus 0$. 	
		For any $h\in\mathrm{GL}_2$ there holds
		\[\alpha b\mathrm{det}^{-t}(h)\otimes \mathrm{det}^{-t'}=\mathrm{det}^{-t'}(f(h))\alpha b\otimes \mathrm{det}^{-t'},\]
		where $f=\tau^k\mathrm{Ad}(g)\in \mathrm{Aut}(\mathrm{GL}_2), k\in\{0, 1\}, g\in\mathrm{GL}_2(\Bbbk)$. Then $\det(f(h))=\det(h)^{\delta(k)}$, where $\delta(1)=-1$ and $\delta(0)=1$. Thus $t=\delta(k)t'$ and the equalities
		\[ [u(x)\otimes\mathrm{det}^{t'}, \alpha b\otimes\mathrm{det}^{-t'}]=\alpha a' bu(x)=\mathrm{d}_e(f)([x\otimes\mathrm{det}^t, b\otimes\mathrm{det}^{-t}])=ab' \mathrm{d}_e(f)(x)\]
		imply $u(x)=a(a'\alpha)^{-1}\mathrm{d}_e(f)(x)$.
	\end{proof}
	\begin{rem}\label{K is not Q}
		Note that despite of coincidence of superdimensions, no ${\bf Q}(2; a, c)$ is isomorphic to any $\mathbb{K}(t)$. Besides, $\mathbb{K}(t)$ contains a normal supersubgroup $\mathbb{H}$, isomorphic to
		$\mathbb{H}(0\oplus 2)$, such that $\mathbb{K}(t)/\mathbb{H}$ is isomorphic to the one-dimensional torus.	
	\end{rem}
	\subsection{Completion of description}
	
	Now, let $\mathfrak{G}_{\bar 1}=M_k\oplus M_{-k}, k\neq 0$. 
	For any $x, y\in\mathfrak{G}_{\bar 1}$ we have $[x, y]=[x, y]_1+[x, y]_2$, where $[x, y]_1\in\mathfrak{z}, [x, y]_2\in\mathfrak{sl}_2$. We also determine a bilinear pairing $< \ , \ >$ (of $\mathrm{GL}_2$-modules), such that $[x, y]_1=<x, y>I_2$.
	
	Let $v=x+y$, where $x\in M_k, y\in M_{-k}$.
	Then the identity $[[v, v], v]=0$ is equivalent to 
	\[(1) \ -k<x, y> x=[[x, y]_2, x], \ (2) \ k<x, y>y=[[x, y]_2, y].\]
	Let $p|k$, whence $p>0$. Then $(\mathrm{GL}_2, \mathfrak{G}_{\bar 1})$ has the structure of Harish-Chandra pair with respect to $[ \ , \ ]_2$. Its kernel equals
	$N_k\oplus N_{-k}$, where $N_{\pm k}=\{n\in M_{\pm k}\mid [n, M_{\mp k}]\subseteq\mathfrak{z} \}$. 
	Note that $N_{\pm k}\neq M_{\pm k}$ and by Proposition \ref{if [g, g] is sl_2} 
	we have \[M_{l}/N_l\simeq L((-t, -t)) \ \mbox{and} \ M_{-l}/N_{-l}\simeq L((t+1, t-1)), t\in p\mathbb{Z}, l\in\{\pm k\}.\]
	If $N_k=N_{-k}=0$, then the bilinear map $[ \ , \ ]_1$ induces a non-zero, hence non-degenerate, pairing \[L((-t, -t))\times L((t+1, t-1))\to\mathfrak{z},\] that infers $L((-t, -t))^*\simeq L((t, t))\simeq L((t+1, t-1))$, which is obviously inconsistent. 
	\begin{pr}\label{N_k is nonzero and N_{-k} is zero} If $N_k\neq 0$ and $N_{-k}=0$, then one of the following statements takes place:
		\begin{enumerate}
			\item $M_k=L((t, t))\oplus L((t+1, t-1)), M_{-k}=L((-t, -t))\simeq L((t, t))^*$, where the Lie bracket $[ \ , \ ]$ on $L((t, t))\times L((-t, -t))$ coincides with $[ \ , \ ]_1$, i.e. it is determined by the standard pairing of mutually dual modules, and on $L((t+1, t-1))\times L((-t, -t))$ is defined as in Proposition \ref{if [g, g] is sl_2};
			\item $M_k=L((t, t))\oplus L((t+1, t-1)), M_{-k}=L((-t+1, -t-1))\simeq L((t+1, t-1))^*$, where the Lie bracket $[ \ , \ ]$ on $L((t+1, t-1))\times L((-t+1, -t-1))$ coincides with $[ \ , \ ]_1$, i.e. it is determined by the standard pairing of mutually dual modules, and on $L((t, t))\times L((-t+1, -t-1))$ is defined as in Proposition \ref{if [g, g] is sl_2}. 
		\end{enumerate}
	\end{pr}
	\begin{proof}
		Since $M_{-k}$ is irreducible, the pairing $N_k\times M_{-k}\to\mathfrak{z}$, induced by $[ \ , \ ]$, is non-degenerate. Thus $N_k\simeq M_{-k}^*$ is isomorphic to either
		$L((t, t))$ or $L((-t+1, -t-1))$. Respectively, $M_k/N_k$ is isomorphic to either $L((t+1, t-1))$ or $L((-t, -t))$. Using \cite[II.7.2(1)]{jan}, one can easily show that for any integers $l$ and $s$ the $\mathrm{GL}_2$-modules $L((l+1, l-1))$ and $L((s, s))$ belong to the different blocks, hence $M_k=R_k\oplus N_k$ with $R_k\simeq M_k/N_k$.  
		Since $R_k$ is irreducible, the restriction of $[ \ , \ ]_1$ on $R_k\times M_{-k}$ is either zero, or non-degenerate. In the latter case $R_k\simeq M_{-k}^*$, which is obviously impossible.  
	\end{proof}
	\begin{pr}\label{N_k and N_{-k} are not zero}
		Let $N_k\neq 0$ and $N_{-k}\neq 0$. For the sake of certainty, let us assume that $M_{k}/N_{k}\simeq L((t+1, t-1))$ and $M_{-k}/N_{-k}\simeq L((-t, -t))$. Then one of the following statements holds :
		\begin{enumerate}
			\item $\mathfrak{G}_{\bar 1}\simeq (L((t+1, t-1))\oplus N_k)\oplus M_{-k}$, where $[L((t+1, t-1)), N_{-k}]=0$ and the induced pairing
			$L((t+1, t-1))\times M_{-k}/N_{-k}\to\mathfrak{sl}_2$ is determined as in Proposition \ref{if [g, g] is sl_2}. Besides, $N_k\times M_{-k}\to\mathfrak{z}$ is non-degenerate and both
			$N_k, M_{-k}$ are trivial $\mathfrak{gl}_2$-modules.  
			\item $\mathfrak{G}_{\bar 1}\simeq (L((t+1, t-1))\oplus N_k)\oplus (L((-t, -t))\oplus N_{-k})$, where 
			\[[L((t+1, t-1)), N_{-k}]=[L((-t, -t)), N_k]=0, \]
			the Lie bracket on $L((t+1, t-1))\times L((-t, -t))$ is determined as in Proposition \ref{if [g, g] is sl_2}, and it induces a non-degenerate pairing $N_k\times N_{-k}\to\mathfrak{z}$, where both $N_k, N_{-k}$ are trivial $\mathfrak{gl}_2$-modules.
		\end{enumerate}
	\end{pr}
	\begin{proof}
		
		Let $S_{\pm k}=\{y\in M_{\pm k}\mid [N_{\mp k}, y]=0 \}$.
		Then the induced pairings 
		$N_{\mp k}\times M_{\pm k}/S_{\pm k}\to\mathfrak{z}$ are non-degenerate, hence $N_{\mp k}\simeq (M_{\pm k}/S_{\pm k})^*$. 
		Comparing dimensions, we obtain
		\[\dim N_{-k}+\dim S_k=\dim M_k=\dim N_k+3, \]\[  \dim N_{k}+\dim S_{-k}=\dim M_{-k}=\dim N_{-k}+1 .\]
		These equalities imply $\dim S_k+\dim S_{-k}=4$. In particular, either $\dim S_k\neq 0$ or $\dim S_{-k}\neq 0$.
		
		We have two Harish-Chandra pairs 
		\[(a) \ (\mathrm{GL}_2, S_k\oplus M_{-k}/N_{-k}) \ \mbox{and} \ (b) \ (\mathrm{GL}_2, M_k/N_k\oplus S_{-k}),\]
		where
		\[[s, \overline{m'}]=[s, m'], \ [\overline{m}, s']=[m, s'], \ \overline{m'}=m'+N_{-k}, \ \overline{m}=m+N_{k}, \ s\in S_{k}, s'\in S_{-k}.\]
		Note that the kernel of pair $(a)$ or $(b)$ is equal to zero if and only if $S_k\neq 0$ or $S_{-k}\neq 0$ respectively.
		
		Assume that $S_k=0$, that infers $\dim S_{-k}=4$. Thus obviously follows that the Harish-Chandra pair $(b)$ satisfies the conclusions neither  of Proposition \ref{when [g, g] is  z} nor of Proposition \ref{if [g, g] is sl_2}, that is $[M_k/N_k, S_{-k}]=\mathfrak{gl}_2$. 
		In particular, $S_{-k}\not\subseteq N_{-k}$, which in turn  implies $S_{-k}+N_{-k}=M_{-k}$. Thus, if $N_k$ is properly contained in a $\mathrm{GL}_2$-submodule $R$, such that 
		$[R, S_{-k}]\subseteq\mathfrak{z}$, then $[R, M_{-k}]\subseteq\mathfrak{z}$, which contradicts the maximality of $N_k$.  
		 
		Applying Proposition \ref{N_k is nonzero and N_{-k} is zero}(2), we conclude $S_{-k}=S'_{-k}\oplus (S_{-k}\cap N_{-k})$, where \[S'_{-k}\simeq L((-t, -t)), S_{-k}\cap N_{-k}\simeq L((-t+1, -t-1)).\]  
		Set $v=m+s'+n, m\in M_k, s'\in S'_{-k}, n\in N_{-k}$. Then the identity $[[v, v], v]=0$ implies $[[\overline{m}, s'], n]=0$, hence $N_{-k}$ is a trivial $\mathfrak{sl}_2$-module,
		which contradicts the fact that $N_{-k}$ contains non-trivial $\mathfrak{sl}_2$-submodule $S_{-k}\cap N_{-k}$. 
	
		Assume that $S_k\neq 0$ and consider the Harish-Chandra pair $(a)$. 
		
		Let $[S_k, M_{-k}/N_{-k}]=\mathfrak{gl}_2$. Then $S_k\not\subseteq N_k$ and Proposition \ref{N_k is nonzero and N_{-k} is zero}(1) infers $S_k=S'_k\oplus (S_k\cap N_k)$, where $S'_k\simeq L((t+1, t-1))$ and $S_k\cap N_k\simeq L((t, t))$. Moreover,  we have $M_k=S'_k\oplus N_k$ and since $\dim S_k=4$, there is $S_{-k}=0$. 
		
		Set $v=s'+n+m', s'\in S_k', n\in N_k, m'\in M_{-k}$. Then $[[v, v], v]=0$ implies $[[s', m'], n]=0$, hence $N_k$ is a trivial $\mathfrak{sl}_2$-module. Since $M_{-k}\simeq N_k^*$, $M_{-k}$ is 
		a trivial $\mathfrak{sl}_2$-module also, and $(1)$ follows. 
		
		If $[S_k, M_{-k}/N_{-k}]=\mathfrak{sl}_2$, then $S_{k}\cap N_{k}=0$ and Proposition \ref{if [g, g] is sl_2} implies that $S_{k}\simeq L((t+1, t-1))$ and $M_{k}=S_{k}\oplus N_{k}$. 
		Furthermore, $\dim S_{-k}=1$ and since $\dim N_k=\dim N_{-k}$, the induced pairing $N_k\times N_{-k}\to\mathfrak{z}$ is non-degenerate. In particula, $S_{-k}\cap N_{-k}=0$, so that $M_{-k}=S_{-k}\oplus N_{-k}$ and
		$S_{-k}\simeq L((-t, -t))$. As above, the restriction of $[ \ , \ ]_1$ on $S_k\times S_{-k}$ is trivial.  
		
		Set $v=s+n+s' +n'$, where $s\in S_k, s'\in S_{-k}, n\in N_k, n'\in N_{-k}$. Then the identity $[[v, v], v]=0$ implies 
		$[[s, s'], n]=[[s, s'], n']=0$, hence $(2)$ follows. 
		
		Let $[S_k, M_{-k}/N_{-k}]=\mathfrak{z}$, that is $S_k\subseteq N_k$. By Proposition \ref{when [g, g] is  z}, $S_k\simeq L((t, t))$ and therefore, $\dim S_{-k}=3$. Consider the Harish-Chandra pair $(b)$.
		
		If $[M_k/N_k, S_{-k}]=\mathfrak{sl}_2$, then $S_{-k}\cap N_{-k}=0$, that implies $S_{-k}\simeq L((-t, -t))$, which contradicts $\dim S_{-k}=3$. If $[M_k/N_k, S_{-k}]=\mathfrak{gl}_2$, then Proposition \ref{N_k is nonzero and N_{-k} is zero}(2) infers $\dim S_{-k} =4$, that is impossible as well. So, we conclude that $[M_k/N_k, S_{-k}]=\mathfrak{z}$, hence $S_{-k}\subseteq N_{-k}$ and $S_{-k}\simeq L((-t+1, -t-1))$. 
		
		Let $M_k=M'\oplus M''$ be a block decomposition, where $M'$ corresponds to the block of $(t+1, t-1)$ and $M''$ is the direct sum of all other block terms. Then $S_k\subseteq M''\subseteq N_k$ and $M'/(M'\cap N_k)\simeq L((t+1, t-1))$. The kernel of the Harish-Chandra pair $(\mathrm{GL}_2, M'\oplus M_{-k})$ is equal to $L=\{y\in M_{-k}\mid [M', y]=0 \}$.
		Since $M_k=M'+N_k$, we obtain that $L\subseteq N_{-k}$ and $S_{-k}\cap L=0$. Note that $\mathfrak{sl}_2\subseteq [M', M_{-k}]$ and if $[M', M_{-k}]=\mathfrak{sl}_2$, then Proposition
		\ref{if [g, g] is sl_2} implies $M_{-k}/L\simeq L((-t, -t))$, from which we derive a contradictory statement that $L=N_{-k}$ contains $S_{-k}$. Therefore, $[M', M_{-k}]=\mathfrak{gl}_2$.
		Arguing as above, we have \[\{x\in M' \mid [x, N_{-k}]=[x, N_{-k}/L]=0\}\subseteq M'\cap S_k=0,\]
		that is, we found ourselves in the already disscussed case $S_k=0$, which drives to a contradiction.
	\end{proof}
	\begin{rem}\label{normal in supergroups from Prop 5.14}
		Let $\mathbb{G}$ be a supergroup from Proposition \ref{N_k and N_{-k} are not zero}. It is obvious that $Z$ is a central (super)subgroup of $\mathbb{G}$. Moreover, $\mathrm{R}_u(\mathbb{G}/Z)$ is represented by the Harish-Chandra subpair $(e, N_k\oplus N_{-k})$, so that $(\mathbb{G}/Z)/\mathrm{R}_u(\mathbb{G}/Z)\simeq \mathbb{H}(0\oplus 2)/\mu_2$. The subpair $(Z, N_k\oplus N_{-k})$ represents $\mathrm{R}(\mathbb{G})$ (use again \cite[Corollary 6.4]{maszub1}). If $\mathbb{G}$ as in Proposition \ref{N_k and N_{-k} are not zero}(2), then  $\mathrm{R}(\mathbb{G})$ is reductive. However, if $\mathbb{G}$ as in Proposition \ref{N_k and N_{-k} are not zero}(1), then $\mathrm{R}(\mathbb{G})$ is not reductive. 
		
	\end{rem}
	Let $p\nmid k$. 
	Polarization of the identity $(1)$ in $x$ gives
	\[ \ -k<u, y> v-k<v, y>u=[[u, y]_2, v]+[[v, y]_2, u], u, v\in M_k,\]
	and symmetrically, polarization of the identity $(2)$ in $y$ gives
	\[ \ k<x, w>z+k<x, z>w=[[x, w]_2, z]+[[x, z]_2, w], z, w\in M_{-k}.\]
	Assume that $M_k$ has a proper submodule $N$. Let $u\in M_k\setminus N$ and $v\in N\setminus 0$. Then
	\[-k<v, y>u\equiv [[v, y]_2, u]\pmod{N}. \]
	Thus, if $[N, M_{-k}]_2=\mathfrak{sl}_2$, then $\mathfrak{sl}_2$ acts on $M/N$ by scalar operators, hence trivially, that in turn implies
	$<N, M_{-k}>=0$. If $[N, M_{-k}]_2=0$, then $<N, M_{-k}>=0$ infers impossible equality $[N, M_{-k}]=0$. 
	Summing all up, we obtain that either $M_k$ is simple or $\mathrm{rad}(M_k)\neq 0$, so that $[\mathrm{rad}(M_k), M_{-k}]=\mathfrak{sl}_2$ and $M_k/\mathrm{rad}(M_k)$ is a simple $\mathrm{GL}_2$-module.  Indeed, if $M_k$ is semi-simple but not simple, or $M_k/\mathrm{rad}(M_k)$ is not simple, then there are proper submodules $N$ and $N'$, such that $N+N'=M_k$. The latter infers $[M_k, M_{-k}]\subseteq \mathfrak{sl}_2$, which is a contradiction.
	
	Symmetrically, we obtain that either $M_{-k}$ is simple or $\mathrm{rad}(M_{-k})\neq 0$, so that $[M_k, \mathrm{rad}(M_{-k})]=\mathfrak{sl}_2$ and $M_{-k}/\mathrm{rad}(M_{-k})$ is a simple $\mathrm{GL}_2$-module.
	
	Assume that both $M_k$ and $M_{-k}$ are not simple. The kernel of the Harish-Chandra subpair $(\mathrm{SL}_2, \mathrm{rad}(M_k)\oplus M_{-k})$ equals $L=\{y\in M_{-k}\mid [\mathrm{rad}(M_k), y]=0 \}$, so that Proposition \ref{if [g, g] is sl_2} implies that $\mathrm{rad}(M_k)$ and $M_{-k}/L$ are non-isomorphic simple
	$\mathrm{GL}_2$-modules from the list $L((-t, -t)), L((t+1, t-1))$, for some integer $t\in\mathbb{Z}$. In particular, $L=\mathrm{rad}(M_{-k})$.
	
	The same arguments, applied to the subpair $(\mathrm{SL}_2, M_k\oplus\mathrm{rad}(M_{-k}))$, show that $M_k/\mathrm{rad}(M_k)$ and $\mathrm{rad}(M_{-k})$ are non-isomorphic simple
	$\mathrm{GL}_2$-modules from the list $L((-l, -l)), L((l+1, l-1))$, for some integer $l\in\mathbb{Z}$. By the above remark on blocks of these modules, one can suppose that  
	\[\mathrm{rad}(M_k)\simeq L((-t, -t)), \ M_k/\mathrm{rad}(M_k)\simeq L((-l, -l)), \]
	\[\mathrm{rad}(M_{-k})\simeq L((l+1, l-1)), \ M_{-k}/\mathrm{rad}(M_{-k})\simeq L((t+1, t-1)).\]
	Then $k=-2t=-2l$ infers $t=l$, hence $M_k$ and $M_{-k}$ are semi-simple by \cite[II.2.12(1)]{jan}, which is a contradiction.
	
	Now, suppose that $M_{-k}$ is simple but $M_{k}$ is not. Arguing as above, we obtain that $\mathrm{rad}(M_k)$ and $M_{-k}$ are non-isomorphic simple $\mathrm{GL}_2$-modules from the list $L((-t, -t)), L((t+1, t-1))$, for some integer $t\in\mathbb{Z}$. The bilinear map $[ \ , \ ]_1$ induces a non-degenerate pairing $M_k/\mathrm{rad}(M_k)\times M_{-k}\to\mathfrak{z}$, that implies 
	$M_k/\mathrm{rad}(M_k)\simeq M_{-k}^*$ and therefore, $M_k$ is again semi-simple. 
	
	Finally, if $M_k$ and $M_{-k}$ are simple, then $[ \ , \ ]_1$ induces a non-degenerate pairing $M_k\times M_{-k}\to\mathfrak{z}$, that implies $M_k\simeq M_{-k}^*$. 
	
	Let $M_k\simeq L(\lambda)$. Note that $n=\lambda_1-\lambda_2\geq 1$, otherwise $[M_k, M_{-k}]$ is at most one-dimensional. Following Remark \ref{onto}, we have the induced   morphism $V(\lambda)\otimes V(\lambda^*))\to\mathfrak{gl}_2$ of $\mathrm{GL}_2$-modules.
	
	We use the notations from Section 3, so that $V(\lambda)$ has a basis $s_i^*\otimes\det^{\lambda_1}, 0\leq i\leq n$, and $V(\lambda^*)$ has a basis $t^*_j\otimes\det^{-\lambda_2},
	0\leq j\leq n$, respectively.
	\begin{lm}\label{a sort of Lemma 3.1} 
		The space $\mathrm{Hom}_{\mathrm{GL}_2}(V(\lambda)\otimes V(\lambda^*), \mathfrak{gl}_2)$ is two dimensional. More precisely, any $\phi\in \mathrm{Hom}_{\mathrm{GL}_2}(V(\lambda)\otimes V(\lambda^*), \mathfrak{gl}_2)$ has a form $\phi=\phi_1+\phi_2$, where $\phi_1\in \mathrm{Hom}_{\mathrm{GL}_2}(V(\lambda)\otimes V(\lambda^*), \mathfrak{z})$ and 	
		$\phi_2\in \mathrm{Hom}_{\mathrm{GL}_2}(V(\lambda)\otimes V(\lambda^*), \mathfrak{sl}_2)$, such that
		\[\phi_1((s_i^*\otimes\mathrm{det}^{\lambda_1})\otimes (t_j^*\otimes\mathrm{det}^{-\lambda_2}))=\delta_{n-i, j}(-1)^i\left( \begin{array}{c}
			n \\
			i
		\end{array}\right)d I_2\]
		and
		\[\phi_2((s_i^*\otimes\mathrm{det}^{\lambda_1})\otimes (t_j^*\otimes\mathrm{det}^{-\lambda_2}))=\psi(s_i^*\otimes t^*_j)  \ \mbox{for some} \]
		\[\psi\in\mathrm{Hom}_{\mathrm{SL}_2}((\mathrm{Sym}_n(V)^*)^{\otimes 2}, \mathfrak{sl}_2), d\in\Bbbk; 0\leq i, j\leq n.\]
		
	\end{lm}
	\begin{proof}
		We have
		\[\mathrm{Hom}_{\mathrm{GL}_2}(V(\lambda)\otimes V(\lambda^*), \mathfrak{gl}_2)\simeq \mathrm{Hom}_{\mathrm{GL}_2}((\mathrm{Sym}_n(V)^*)^{\otimes 2}\otimes \mathrm{det}^n, \mathfrak{gl}_2)=\]
		\[\mathrm{Hom}_{\mathrm{PGL}_2}((\mathrm{Sym}_n(V)^*)^{\otimes 2}\otimes \mathrm{det}^n, \mathfrak{gl}_2)=\mathrm{Hom}_{\mathrm{PSL}_2}((\mathrm{Sym}_n(V)^*)^{\otimes 2}\otimes \mathrm{det}^n, \mathfrak{gl}_2)=\]
		\[\mathrm{Hom}_{\mathrm{SL}_2}((\mathrm{Sym}_n(V)^*)^{\otimes 2}\otimes \mathrm{det}^n, \mathfrak{gl}_2)\simeq \mathrm{Hom}_{\mathrm{SL}_2}((\mathrm{Sym}_n(V)^*)^{\otimes 2}, \mathfrak{gl}_2)\]	
		\[\simeq \mathrm{Hom}_{\mathrm{SL}_2}((\mathrm{Sym}_n(V)^*)^{\otimes 2}, \mathfrak{z})\oplus \mathrm{Hom}_{\mathrm{SL}_2}((\mathrm{Sym}_n(V)^*)^{\otimes 2}, \mathfrak{sl}_2),\]
		where the latter two terms are both nonzero. Further, we have 
		\[\mathrm{Hom}_{\mathrm{SL}_2}((\mathrm{Sym}_n(V)^*)^{\otimes 2}, \mathfrak{z})\simeq\mathrm{Hom}_{\mathrm{SL}_2}(\mathrm{Sym}_n(V)^*, \mathrm{Sym}_n(V))\simeq\mathrm{Hom}_{\mathrm{SL}_2}(V(n), H^0(n))\]
		is one-dimensional, regardless what is $\mathrm{char}\Bbbk$ equal to.
		
		Comparing the weights, one easily sees that for any $\phi\in\mathrm{Hom}_{\mathrm{SL}_2}((\mathrm{Sym}_n(V)^*)^{\otimes 2}, \mathfrak{z})\setminus 0$,
		there holds
		\[\phi(s_i^*\otimes t_j^*)=0, \ \mbox{if} \ i+j\neq n, \ \mbox{otherwise} \ \phi(s_i^*\otimes t_{n-i}^*)=d_i I_2, d_i\in\Bbbk, 0\leq i\leq n.\]  
		Using equalities
		\[E_{-1, 1}(s^*_{i-1}\otimes t^*_{n-i})=-is^*_i\otimes t^*_{n-i}-(n-i+1)s^*_{i-1}\otimes t^*_{n-i+1}, 0\leq i\leq n,\]
		we derive
		\[id_i+(n-i+1)d_{i-1}=0, 1\leq i\leq n, d_{-1}=0,\]
		hence 
		\[d_i=(-1)^i\left( \begin{array}{c}
			n \\
			i
		\end{array}\right)d, 0\leq i\leq n, \]
	at least if $\mathrm{char}\Bbbk=0$. Set $d=1$ and then note that the $\mathbb{Z}$-form $\sum_{0\leq i\leq n}s_i^*$ of $\mathrm{Sym}_2(V)^*$ is a $\mathrm{U}(\mathfrak{sl}_2)_{\mathbb{Z}}$-stable, where $\mathrm{U}(\mathfrak{sl}_2)_{\mathbb{Z}}$ is the Kostant's $\mathbb{Z}$-form of the enveloping algebra 
	of $\mathfrak{sl}_2$ (cf. \cite[II.1.12]{jan}). Since $\mathrm{Dist}(\mathrm{SL}_2)\simeq \mathrm{U}(\mathfrak{sl}_2)_{\mathbb{Z}}\otimes_{\mathbb{Z}}\Bbbk$, we complete the proof.
	\end{proof}
\begin{pr}\label{when G_1=M_k+M_{-k}}
	Let $G\simeq\mathrm{GL}_2$ and $\mathfrak{G}_{\bar 1}=M_k\oplus M_{-k}, p\nmid k$. Then $M_k\simeq L(\lambda), M_{-k}\simeq L(\lambda)^*$, where $\lambda=(t, t-1)$, so that $L(\lambda)\simeq V^*\otimes\det^{t}$ and $L(\lambda)^*\simeq V^*\otimes\det^{1-t}$. Moreover, the Lie bracket $[ \ , \ ]$ on $M_k\times M_{-k}$ is defined as
	\[ [s_0^*\otimes\mathrm{det}^t, t_0^*\otimes\mathrm{det}^{1-t}]=aE_{1, -1}, \ [s_0^*\otimes\mathrm{det}^t, t^*_1\otimes\mathrm{det}^{1-t}]=\frac{a}{2}(H-\frac{1}{k}I_2), \]
	\[ [s_1^*\otimes\mathrm{det}^t, t^*_0\otimes\mathrm{det}^{1-t}]=\frac{a}{2}(H+\frac{1}{k}I_2), \ [s_1^*\otimes\mathrm{det}^t, t_1^*\otimes\mathrm{det}^{1-t}]=-aE_{-1, 1},\]
	where $a\in\Bbbk\setminus 0$.		
\end{pr}
	\begin{proof}
		As above, the space $\mathrm{Hom}_{\mathrm{GL}_2}(L(\lambda)\otimes L(\lambda^*), \mathfrak{gl}_2)$ has a nonzero direct term isomorphic to $\mathrm{Hom}_{\mathrm{GL}_2}(L(\lambda)\otimes L(\lambda^*), \mathfrak{sl}_2)$, otherwise $[L(\lambda), L(\lambda^*)]\subseteq\mathfrak{z}$. Next, we have
		\[\mathrm{Hom}_{\mathrm{GL}_2}(L(\lambda)\otimes L(\lambda^*), \mathfrak{sl}_2)\simeq \mathrm{Hom}_{\mathrm{SL}_2}(L(n)^{\otimes 2} , \mathfrak{sl}_2),\]
		so that \cite[Proposition 2.15] {zub-bar} and Lemma \ref{when pairing is trivial} imply that $n$ is odd and $p\nmid n$.
		
		Assume that $n>1$. Set $v=\overline{s^*_0\otimes\det^{\lambda_1}}+\overline{t^*_{n}\otimes\det^{-\lambda_2}}$. 	Using formulas $(\circ\circ)$ and Lemma \ref{a sort of Lemma 3.1}, we obtain
		\[[[v, v], v]=[-aH, v]=-an\overline{s_0^*\otimes\det^{\lambda_1}} +an\overline{t^*_n\otimes\det^{-\lambda_2}}\neq 0,\] which is a contradiction. 
				If $n=1$, then $L(\lambda)\simeq V(\lambda), \ L(\lambda^*)\simeq V(\lambda^*)$ and
		\[ [s_0^*\otimes\mathrm{det}^{\lambda_1}, t_0^*\otimes\mathrm{det}^{-\lambda_2}]=aE_{1, -1}, \ [s_0^*\otimes\mathrm{det}^{\lambda_1}, t^*_1\otimes\mathrm{det}^{-\lambda_2}]=d I_2+\frac{a}{2}H, \]
		\[ [s_1^*\otimes\mathrm{det}^{\lambda_1}, t^*_0\otimes\mathrm{det}^{-\lambda_2}]=-d I_2+\frac{a}{2}H, \ [s_1^*\otimes\mathrm{det}^{\lambda_1}, t_1^*\otimes\mathrm{det}^{-\lambda_2}]=-aE_{-1, 1},\]
		where $\lambda_2=\lambda_1-1$.
		
		Set $v=\alpha s_0^*\otimes\mathrm{det}^{\lambda_1}+\beta s_1^*\otimes\mathrm{det}^{\lambda_1}+\gamma t_0^*\otimes\mathrm{det}^{-\lambda_2}+\delta t_1^*\otimes\mathrm{det}^{-\lambda_2}$.
		
		We have
		\[[v, v]=2\alpha\gamma aE_{1, -1}+2\alpha\delta (d I_2+\frac{a}{2}H) +2\beta\gamma (-d I_2+\frac{a}{2}H)-2\beta\delta a E_{-1, 1}= \]
		\[2\alpha\gamma aE_{1, -1}+2(\alpha\delta-\beta\gamma)d I_2+(\alpha\delta+\beta\gamma)a H-2\beta\delta a E_{-1, 1}\]
		and
		\[[[v, v], v]= \]
		\[-2\alpha\beta\gamma as^*_0\otimes\mathrm{det}^{\lambda_1}-2\alpha\gamma\delta a t_0^*\otimes\mathrm{det}^{-\lambda_2}+\]
		\[2(\alpha\delta-\beta\gamma)dk(\alpha s_0^*\otimes\mathrm{det}^{\lambda_1}+\beta s_1^*\otimes\mathrm{det}^{\lambda_1})\]\[-2(\alpha\delta-\beta\gamma)dk(\gamma t_0^*\otimes\mathrm{det}^{-\lambda_2}+\delta t_1^*\otimes\mathrm{det}^{-\lambda_2})+\]
		\[(\alpha\delta+\beta\gamma)a(\alpha s_0^*\otimes\mathrm{det}^{\lambda_1}-\beta s_1^*\otimes\mathrm{det}^{\lambda_1}+\gamma t_0^*\otimes\mathrm{det}^{-\lambda_2}-\delta t_1^*\otimes\mathrm{det}^{-\lambda_2})\]
		\[+2\alpha\beta\delta as_1^*\otimes\mathrm{det}^{\lambda_1}+2\beta\gamma\delta a t_1^*\otimes\mathrm{det}^{-\lambda_2}=\]
		\[(-2\alpha\beta\gamma a+2(\alpha\delta-\beta\gamma)\alpha dk+(\alpha\delta+\beta\gamma)\alpha a)s_0^*\otimes\mathrm{det}^{\lambda_1} +\]
		\[(2(\alpha\delta-\beta\gamma)\beta dk -(\alpha\delta+\beta\gamma)\beta a +2\alpha\beta\delta a)s^*_1\otimes\mathrm{det}^{\lambda_1}+\]
		\[(-2\alpha\gamma\delta a-2(\alpha\delta-\beta\gamma)\gamma dk+(\alpha\delta+\beta\gamma)\gamma a)t_0^*\otimes\mathrm{det}^{-\lambda_2}+\]
		\[(-2(\alpha\delta-\beta\gamma)\delta dk-(\alpha\delta+\beta\gamma)\delta a+2\beta\gamma\delta a)t_1^*\otimes\mathrm{det}^{-\lambda_2}=\]
		\[(\alpha\delta-\beta\gamma)(a+2dk)(\alpha s_0^*\otimes\mathrm{det}^{\lambda_1}+\beta s_1^*\otimes\mathrm{det}^{\lambda_1}-\gamma t_0^*\otimes\mathrm{det}^{-\lambda_2}-\delta t_1^*\otimes\mathrm{det}^{-\lambda_2}).\]
		Thus one easily sees that $[[v, v], v]=0$ if and only if $d=-\frac{a}{2k}$.  
	\end{proof}
	Summing all up, we obtain the following.
	Let $\mathbb{S}(t), \mathbb{L}(t)$ and $\mathbb{H}(t)$ denote the supergroups from Proposition \ref{N_k is nonzero and N_{-k} is zero}(1, 2) and Proposition \ref{when G_1=M_k+M_{-k}}, respectively. 
	Their Lie superalgebras are denoted by $\mathfrak{S}(t), \mathfrak{L}(t)$and $\mathfrak{H}(t)$, respectively. Compairing dimensions, one immediately sees that for any integers
	$t, t'$ there is $\mathbb{S}(t)\not\simeq\mathbb{L}(t')$, $\mathbb{S}(t)\not\simeq\mathbb{H}(t')$ and $\mathbb{L}(t)\not\simeq\mathbb{H}(t')$.
	\begin{pr}\label{isomorphism classes of S}
		$\mathbb{S}(t)\simeq\mathbb{S}(t')$ if and only if $t=\pm t'$.	
	\end{pr}
	\begin{proof}
		Let $(f, u)$ be a Harish-Chandra isomorphism, that determines a supergroup isomorphism $\mathbb{S}(t)\simeq\mathbb{S}(t')$.	Arguing as in Proposition \ref{when Q(a, c) is isomorphic to Q(a', c')?}, we conclude that $u(L((t+1, t-1))=L((t'+1, t'-1))$ and either  \[u(L((t, t)))=L((t', t')), u(L((-t, -t)))=L((-t', -t')), \] or \[u(L((t, t)))=L((-t', -t')), u(L((-t, -t)))=L((t', t')).\]
		Set $f=\tau^k\mathrm{Ad}(g)\in \mathrm{Aut}(\mathrm{GL}_2), k\in\{0, 1\}; g\in\mathrm{GL}_2(\Bbbk)$. Following Proposition \ref{isomorphism classes of K(t)}, 
		one can record
		\[u(x\otimes\mathrm{det}^t)=u(x)\otimes\mathrm{det}^{t'}, x\in\mathfrak{sl}_2,\]
		and either
		\[u(b\otimes\mathrm{det}^t)=\alpha b\otimes\mathrm{det}^{t'}, \ u(b'\otimes\mathrm{det}^{-t})=\alpha' b'\otimes\mathrm{det}^{-t'},\]
		or 
		\[u(b\otimes\mathrm{det}^t)=\alpha b\otimes\mathrm{det}^{-t'}, \ u(b'\otimes\mathrm{det}^{-t})=\alpha' b'\otimes\mathrm{det}^{t'}, \ b, b'\in\Bbbk.\]
		Then either $t=\delta(k)t'$, or $t=-\delta(k)t'$. In the first case we obtain $u(x)=a(a'\alpha)^{-1}\mathrm{d}_e(f)(x), x\in\mathfrak{sl}_2$, and $\alpha\alpha'=\delta(k)$.
		In the second case we obtain $\mathrm{d}_e(f)|_{\mathfrak{sl}_2}=0$, which is a contradiction.
	\end{proof}
	\begin{rem}\label{normal in S}
		$\mathbb{S}(t)$ contains a normal supersubgroup $\mathbb{H}$, which is represented by the Harish-Chandra subpair $(\mathrm{SL}_2, L((t+1, t-1)))$. Moreover, $\mathbb{H}\simeq \mathrm{SL}_2\ltimes (\mathbb{G}_a^-)^3$ and $\mathbb{S}(t)/\mathbb{H}$	is represented by the Harish-Chandra pair $(G_m, \Bbbk_{2t}\oplus\Bbbk_{-2t})$, where $p|t$. In particular,
		$\mathbb{S}(t)/\mathbb{H}$ is not split and solvable.
	\end{rem}
	\begin{pr}\label{isomorphism classes of L}
		$\mathbb{L}(t)\simeq\mathbb{L}(t')$ if and only if 	$t=\pm t'$.
	\end{pr}
	\begin{proof}
		Similarly to Proposition \ref{isomorphism classes of S}, we have 
		$u(b\otimes\mathrm{det}^t)=\alpha b\otimes\mathrm{det}^{t'}, b\in\Bbbk,$ that implies $t=\delta(k)t'$.	Further, if $t=-t'$, then we set
		\[u(L((t+1, t-1)))=L((-t+1, -t-1)), \ u(L((-t+1, -t-1)))=L((t+1, t-1)),\]
		where
		\[u(x\otimes\mathrm{det}^t)=u(x)\otimes\mathrm{det}^{-t}, \ u(x\otimes\mathrm{det}^{-t})=u'(x)\otimes\mathrm{det}^t, x\in\mathfrak{sl}_2.\]
		Thus $u(hxh^{-1})=f(h)u(x)f(h)^{-1}$ and $u'(hxh^{-1})=f(h)u'(x)f(h)^{-1}$ for arbitrary $h\in\mathrm{GL}_2$. Without loss of generality, one can assume that in both
		$\mathfrak{L}(t)_{\bar 1}$ and $\mathfrak{L}(t')_{\bar 1}$ it holds that  
		\[[x\otimes\mathrm{det}^t, y\otimes\mathrm{det}^{-t}]=\mathrm{tr}(xy)I_2, x, y\in\mathrm{sl}_2. \]
		It follows that
		\[a\mathrm{d}_e(f)(x)=\alpha a' u'(x) \ \mbox{and} \ <x, y>=-\mathrm{tr}(u(x)u'(y)), x, y\in\mathfrak{sl}_2 .\] 
		It remains to set $u=-a^{-1}\alpha a'\mathrm{d}_e(f)$ and use the above mentioned fact that $\mathrm{tr}(xy)=\mathrm{tr}(\mathrm{d}_e(f)(x)\mathrm{d}_e(f)(y))$.
	\end{proof}
	\begin{rem}\label{normal in L}
		Similarly to $\mathbb{S}(t)$, the supergroup $\mathbb{L}(t)$ contains a normal supersubgroup $\mathbb{H}$, represented by the Harish-Chandra subpair $(\mathrm{SL}_2, L((t, t)))$. Besides, $\mathbb{H}\simeq\mathrm{SL}_2\ltimes\mathbb{G}_a^-$ and $\mathbb{L}(t)/\mathbb{H}$ is represented by the Harish-Chandra pair $(G_m, \Bbbk_{2t}^{\oplus 3}\oplus\Bbbk_{-2t}^{\oplus 3}), p|t$. In particular, $\mathbb{L}(t)/\mathbb{H}$ is not split and solvable as well.	
	\end{rem}
	Observe that $\mathrm{SL}(2|1)\simeq\mathbb{H}(1)$. Indeed, recall that for any superalgebra $R$ the group $\mathrm{SL}(2|1)(R)$ consists of matrices
	\[ g=\left(\begin{array}{cc}
		A & v \\
		w & c	
	\end{array}\right), \ A\in\mathrm{GL}_2(R_{\bar 0}), v\in \mathrm{Mat}_{2\times 1}(R_{\bar 1}), \ w\in\mathrm{Mat}_{1\times 2}(R_{\bar 1}), c\in R_{\bar 0}^{\times},\]
	such that $\mathrm{Ber}(g)=\det(A-c^{-1}vw)c^{-1}=1$. In particular, $\mathrm{SL}(2|1)_{ev}(R)$ consists of matrices
	\[\left(\begin{array}{cc}
		A & 0 \\
		0 & \mathrm{det}(A)	
	\end{array}\right), A\in\mathrm{GL}_2(R_{\bar 0}),\]
	and the $\mathrm{SL}(2|1)_{ev}$-module $\mathfrak{sl}(2|1)_{\bar 1}$ is isomorphic to $L((1, 0))\oplus L((1, 0))^*$, where $L((1, 0))$ and $L((1, 0))^*$ coincide with the subspaces consisting of matrices
	\[\left(\begin{array}{cc}
		0 & 0 \\
		w & 0	
	\end{array}\right), w\in\mathrm{Mat}_{1\times 2}(\Bbbk),\]
	and
	\[\left(\begin{array}{cc}
		0 & v \\
		0 & 0	
	\end{array}\right), v\in\mathrm{Mat}_{2\times 1}(\Bbbk),\]
	respectively. 
	
	In general, it is clear that $\mathbb{H}(t)\simeq \mathbb{H}(1-t)$.
	\begin{pr}\label{all is SL(2|1)?}
		$\mathbb{H}(t)\simeq \mathbb{H}(t')$ if and only if either $t=t'$ or $t+t'=1$.
	\end{pr}	
	\begin{proof}
		Let $(f, u)$ be a Harish-Chandra isomorphism, that determines a supergroup isomorphism $\mathbb{H}(t)\simeq\mathbb{H}(t')$. Arguing as in Proposition \ref{when Q(a, c) is isomorphic to Q(a', c')?}, we conclude that either $u(L(\lambda))=L(\lambda'), u(L(\lambda)^*)=L(\lambda')^*$, or $u(L(\lambda))=L(\lambda')^*, u(L(\lambda)^*)=L(\lambda')$. 
		
		Consider the first case. 
		
		For arbitrary integer $l$ we identify  $V^*\otimes\det^l$ with $V\otimes\det^{l-1}$ by the rule \[s_0^*\otimes\mathrm{det}^l\mapsto v_1\otimes\mathrm{det}^{l-1}, \ s_1^*\mathrm{det}^l\mapsto -v_{-1}\mathrm{det}^{l-1}.\] 
		Then the formulas in Proposition \ref{when G_1=M_k+M_{-k}}(2) are transformed as
		\[ [v_1\otimes\mathrm{det}^{t-1}, v_1\otimes\mathrm{det}^{-t}]=E_{1, -1}, \ [v_1\otimes\mathrm{det}^{t-1}, v_{-1}\otimes\mathrm{det}^{-t}]=-\frac{1}{2}(H-\frac{1}{k}I_2), \]
		\[ [v_{-1}\otimes\mathrm{det}^{t-1}, v_1\otimes\mathrm{det}^{-t}]=-\frac{1}{2}(H+\frac{1}{k}I_2), \ [v_{-1}\otimes\mathrm{det}^{t-1}, v_{-1}\otimes\mathrm{det}^{-t}]=-E_{-1, 1},\]
		where $t\in\mathbb{Z}$ and $k=2t-1$ is coprime to $p=\mathrm{char}\Bbbk$.
		
		Let \[u(v_i\otimes\mathrm{det}^{t-1})=\sum_{j\in\{\pm 1\}} u_{ji}v_j\otimes\mathrm{det}^{t'-1} \ \mbox{and} \ u(v_i\otimes\mathrm{det}^{-t})=\sum_{j\in\{\pm 1\}} u'_{ji}v_j\otimes\mathrm{det}^{-t'},\]
		where $u_{ji}, u'_{ji}\in\Bbbk, i\in\{\pm 1\}$. Let $U$ and $U'$ denote the matrices $(u_{ji})_{j, i\in\{\pm 1\}}$ and $(u'_{ji})_{j, i\in\{\pm 1\}}$ respectively. For any $h\in\mathrm{GL}_2$, we have
		\[U(\mathrm{det}(h)^{t-1}h)=\mathrm{det}(f(h))^{t'-1}f(h)U, \ U'(\mathrm{det}(h)^{-t}h)=\mathrm{det}(f(h))^{-t'}f(h)U'.\]
		Set $f=\tau^k\mathrm{Ad}(g)\in \mathrm{Aut}(\mathrm{GL}_2), k\in\{0, 1\}; g\in\mathrm{GL}_2(\Bbbk)$. Then 
		\[f(h)=\mathrm{det}(h)^{t-1-\delta(k)(t'-1)} UhU^{-1}, \ f(h)=\mathrm{det}(h)^{-t+\delta(k)t'}U'h(U')^{-1}.\]
		By Remark \ref{determinant-based}, there should be $t-1-\delta(k)(t'-1)=-t+\delta(k)t'=0, -1$ and $U$ equals $U'$ up to a nonzero scalar multiple. If $k=1$, then $t+t'=1$, otherwise $k=0$ and $t-t'=0$.
		
		The second case can be reduced to the first one if the isomorphism $\mathbb{H}(t)\simeq\mathbb{H}(t')$ is replaced by a composition of isomorphisms
		$\mathbb{H}(t)\simeq\mathbb{H}(t')\simeq\mathbb{H}(1-t')$.
	\end{proof}
	\begin{rem}\label{but Lie superalgebras are the same!}
		Contrary to the statement of Proposition \ref{all is SL(2|1)?}, for any $t$ and $t'$, such that $p$ is coprime to both $k=2t-1$ and $k'=2t'-1$, we have $\mathfrak{H}(t)\simeq\mathfrak{H}(t')\simeq\mathfrak{sl}(2|1)$!
		In the matter of fact, let $f$ be a {\bf trace-based} automorphism of $\mathfrak{gl}_2$, such that $f(x)=x+q\mathrm{tr}(x)I_2, x\in\mathfrak{gl}_2, q\in\Bbbk$. We need to find a linear map $u : \mathfrak{H}(t)_{\bar 1}\to\mathfrak{H}(t')_{\bar 1}$, such that
		\[(\cdot) \ u([x, v])=[f(x), u(v)], \ \mbox{and} \ f([v, v'])=[u(v), u(v')], \ x\in\mathfrak{gl}_2, v, v'\in\mathfrak{H}(t)_{\bar 1}.\]
		In the above notations, we set $U=U'=I_2$. Then the equalities in $(\cdot)$ are valid if and only if $k=(2q+1)k'$.
	\end{rem}
	
	\section{Applications and concluding remarks}
	
	Let $\mathbb{G}$ be a smooth connected quasi-reductive supergroup.  
	Let $T$ be a maximal torus of $G=\mathbb{G}_{ev}$ and $\Delta$ be the corresponding root system of $\mathbb{G}$. 
	
	For any $\alpha\in\Delta$, let $\overline{\mathbb{G}_{\alpha}}$ denote $\mathbb{G}_{\alpha}/T'$. 
	\begin{pr}\label{if alpha is not multiple of odd}
		Let $\alpha\in\Delta_{\bar 0}\setminus\mathbb{Q}\Delta_{\bar 1}$. Then $[\mathfrak{G}_{\bar 1}^T, \mathfrak{G}_{\bar 1}^T]\subseteq\mathrm{Lie}(T_{\alpha})$, so that the subpair
		$(T_{\alpha}, \mathfrak{G}_{\bar 1}^T)$ corresponds to the solvable radical $\mathrm{R}(\mathbb{G}_{\alpha})$. Moreover, if $\overline{G_{\alpha}}$ is isomorphic to $\mathrm{SL}_2$ or
		$\mathrm{PSL}_2$, then $\mathbb{G}_{\alpha}\simeq \overline{G_{\alpha}}\ltimes\mathrm{R}(\mathbb{G}_{\alpha})$. In general, $\mathbb{G}_{\alpha}/\mathrm{R}(\mathbb{G}_{\alpha})\simeq\mathrm{PGL}_2$. 	
	\end{pr}	
	\begin{proof}
		By Lemma \ref{structure of centralizer}(1), the Harish-Chandra pair of centralizer $\mathbb{G}_{\alpha}$ is isomorphic to 
		$(G_{\alpha}, \mathfrak{G}_{\bar 1}^T)$. 
		Since all weights of $\mathfrak{G}_{\bar 1}^T$ are zero, it is a trivial $G_{\alpha}$-module. Thus $[\mathfrak{G}_{\bar 1}^T, \mathfrak{G}_{\bar 1}^T]\subseteq \mathrm{Lie}(G_{\alpha})^{G_{\alpha}}=\mathrm{Lie}(\mathrm{Z}(G_{\alpha}))=\mathrm{Lie}(T_{\alpha})$, hence $(T_{\alpha}, \mathfrak{G}_{\bar 1}^T)$ is a Harish-Chandra subpair, that corresponds to a normal smooth connected supersubgroup $\mathbb{H}$, which is also solvable by \cite[Corollary 6.4]{maszub1}. It remains to note that  $\mathbb{G}_{\alpha}/\mathbb{H}\simeq\mathrm{PGL}_2$ is semi-simple.	
	\end{proof}
	\begin{pr}\label{if any even is not multiple of odd}
		If $\Delta_{\bar 0}\cap\mathbb{Q}\Delta_{\bar 1}=\emptyset$, then $\mathfrak{G}_{\bar 1}^T\subseteq \mathfrak{G}_{\bar 1}^G$ and $[\mathfrak{G}_{\bar 1}, \mathfrak{G}_{\bar 1}^T]\subseteq\mathrm{Lie}(\mathrm{Z}(G))$. In particular,  the subpair
		$(\mathrm{Z}(G), \mathfrak{G}_{\bar 1}^T)$ corresponds to a normal solvable supersubgroup $\mathbb{H}$ of $\mathbb{G}$. Moreover, if $\mathrm{Z}(G)$ is a torus, then	
		$\mathbb{H}\leq\mathrm{R}(\mathbb{G})$.
	\end{pr}
	\begin{proof}
		Arguing as in Proposition \ref{if alpha is not multiple of odd}, we obtain that $\mathfrak{G}_{\bar 1}^T$ is a trivial $G_{\alpha}$-module for arbitrary $\alpha\in\Delta_{\bar 0}$. Thus, \cite[Theorem 21.11(f)]{milne} implies that $\mathfrak{G}_{\bar 1}^T$ is a trivial $G$-submodule of $\mathfrak{G}_{\bar 1}$, hence again $[\mathfrak{G}_{\bar 1}^T, \mathfrak{G}_{\bar 1}^T]\subseteq\mathrm{Lie}(\mathrm{Z}(G))$. Further, Lemma \ref{structure of centralizer}(2) implies that for any $\alpha\in\Delta_{\bar 1}$ the Harish-Chandra pair of
		centralizer $\mathbb{G}_{\alpha}$ is $(T, (\oplus_{\beta\in\mathbb{Q}\alpha\cap\Delta_{\bar 1}}\mathfrak{G}_{\bar 1}^{\beta})\oplus\mathfrak{G}_{\bar 1}^T)$. Comparing the weights, we obtain $[\mathfrak{G}_{\bar 1}^{\alpha}, \mathfrak{G}_{\bar 1}^T]=0$. 
		
		Since $\mathrm{Z}(G)\leq T$, then $\mathrm{Z}(G)$ is a central subgroup of $\mathbb{H}$, such that $\mathbb{H}/\mathrm{Z}(G)\simeq (\mathbb{G}_a^-)^{\dim\mathfrak{G}_{\bar 1}^T}$ is an abelian supergroup.  
		
		Finally, if $\mathrm{Z}(G)$ is a torus, then $\mathrm{R}(G)=\mathrm{Z}(G)$ by \cite[Corollary 17.62]{milne}.
	\end{proof}
	Recall that for any $\alpha\in\Delta$, such that $\mathbb{Q}\alpha\cap\Delta_{\bar 1}\neq\emptyset$, the induced Harish-Chandra pair structure on $(\overline{G_{\alpha}}, \mathfrak{G}_{\bar 1}^{T_{\alpha}})$ (see Section 2) is associated with the supergroup $\overline{\mathbb{G}_{\alpha}}\simeq\mathbb{G}_{\alpha}/T'$. The corresponding kernel of $\mathfrak{G}_{\bar 1}^{T_{\alpha}}$ is also a $G_{\alpha}$-module, which is equal to $\{n\in \mathfrak{G}_{\bar 1}^{T_{\alpha}}\mid [n, \mathfrak{G}_{\bar 1}^{T_{\alpha}}]\subseteq\mathrm{Lie}(T')\}$.
	We denote it by $K'(\mathfrak{G}_{\bar 1}^{T_{\alpha}})$. As we have already noted, the subpair $(e, K'(\mathfrak{G}_{\bar 1}^{T_{\alpha}}))$ determines $\mathrm{R}_u(\overline{\mathbb{G}_{\alpha}})$.
	
	Let $K''(\mathfrak{G}_{\bar 1}^{T_{\alpha}})$ also denote $\{n\in \mathfrak{G}_{\bar 1}^{T_{\alpha}}\mid [n, \mathfrak{G}_{\bar 1}^{T_{\alpha}}]\subseteq\mathrm{Lie}(T_{\alpha})\}$.
	One easily sees that the subpair $(T_{\alpha}, K''(\mathfrak{G}_{\bar 1}^{T_{\alpha}}))$ corresponds to a smooth connected normal supersubgroup $\mathbb{G}_{\alpha}$.
	\begin{pr}\label{SL_2}
		Let $\alpha\in\Delta_{\bar 0}\cap\mathbb{Q}\Delta_{\bar 1}$ and $\overline{G_{\alpha}}\simeq\mathrm{SL}_2$ or $\overline{G_{\alpha}}\simeq\mathrm{PSL}_2$.	Then the subpair $(T_{\alpha}, K'(\mathfrak{G}_{\bar 1}^{T_{\alpha}}))$ corresponds to $\mathrm{R}(\mathbb{G}_{\alpha})$ and $\mathbb{G}_{\alpha}/\mathrm{R}(\mathbb{G}_{\alpha})$ is isomorphic to one of the reductive supergroups from Theorem \ref{or ... } or Theorem \ref{general for PSL_2} respectively.
	\end{pr}
	\begin{proof}
		By normality criterion the subpair $(T_{\alpha}, K'(\mathfrak{G}_{\bar 1}^{T_{\alpha}}))$ corresponds to a normal supersubgroup $\mathbb{H}$. Moreover, $\mathbb{H}$ is smooth, connected and solvable. Using Theorem \ref{or ... }, we conclude that $\mathbb{G}_{\alpha}/\mathbb{H}$ is isomorphic to one of the following supergroups : 
		$\mathrm{SL}_2$, $\mathrm{SpO}(2|1)$,  $\mathbb{H}(3/1)\simeq\mathrm{Aut}(\mathfrak{spo}(2|1))^0$ (if $\mathrm{char}\Bbbk=3$), or $\mathbb{H}(0\oplus 2)$. Since all of them are obviously semi-simple, the first statement follows. The case $\overline{G_{\alpha}}\simeq\mathrm{PSL}_2$ is similar. 
	\end{proof}
	\begin{pr}\label{GL_2}
		Let $\alpha\in\Delta_{\bar 0}\cap\mathbb{Q}\Delta_{\bar 1}$ and $\overline{G_{\alpha}}\simeq\mathrm{GL}_2$. Then $\mathrm{R}(\mathbb{G}_{\alpha})$ is
		represented by the Harish-Chandra subpair $(T_{\alpha}, K''(\mathfrak{G}_{\bar 1}^{T_{\alpha}}))$. Furthermore, we have :
		\begin{enumerate}
			\item If $\overline{\mathbb{G}_{\alpha}}/\mathrm{R}_u(\overline{\mathbb{G}_{\alpha}})$ is isomorphic to a supergroup from Proposition \ref{when [g, g] is  z}, then $K''(\mathfrak{G}_{\bar 1}^{T_{\alpha}})=\mathfrak{G}_{\bar 1}^{T_{\alpha}}$ and $\mathbb{G}_{\alpha}/\mathrm{R}(\mathbb{G}_{\alpha})\simeq\mathrm{PGL}_2$;
			\item If $\overline{\mathbb{G}_{\alpha}}/\mathrm{R}_u(\overline{\mathbb{G}_{\alpha}})$ is isomorphic to a supergroup from Proposition \ref{if [g, g] is sl_2},  then $K''(\mathfrak{G}_{\bar 1}^{T_{\alpha}})=K'(\mathfrak{G}_{\bar 1}^{T_{\alpha}}))$ and $\mathbb{G}_{\alpha}/\mathrm{R}(\mathbb{G}_{\alpha})\simeq\mathbb{H}(0\oplus 2)/\mu_2$;
			\item If $\overline{\mathbb{G}_{\alpha}}/\mathrm{R}_u(\overline{\mathbb{G}_{\alpha}})$ is isomorphic to one of the  supergroups from Proposition \ref{GL_2 and k=0} or Proposition \ref{if R intersects V but not contained in}, then
			 $\mathbb{G}_{\alpha}/\mathrm{R}(\mathbb{G}_{\alpha})\simeq\mathbb{H}(0\oplus 2)/\mu_2$.
		\end{enumerate} 	
	\end{pr}
	\begin{proof}
		Recall that in the notations of the previous section, $\mathfrak{G}_{\bar 1}^{T_{\alpha}}=M_0$.
		Let $\mathbb{H}$ denote the supersubgroup of $\mathbb{G}_{\alpha}$, which is associated with the subpair $(T_{\alpha}, K''(\mathfrak{G}_{\bar 1}^{T_{\alpha}}))$. As we have already observed, 
		$\mathbb{H}\leq\mathrm{R}(\mathbb{G}_{\alpha})$. 
		
		For example, assume $(3)$. Then one easily sees that $K''(\mathfrak{G}_{\bar 1}^{T_{\alpha}})/K'(\mathfrak{G}_{\bar 1}^{T_{\alpha}})$ is nothing else but the submodule $U$ therein. Thus $\mathbb{G}_{\alpha}/\mathbb{H}\simeq \mathbb{H}(0\oplus 2)/\mu_2$ is semi-simple, hence our statement follows. The remaining cases are similar. 
	\end{proof}
	\begin{rem}\label{odd root}
		If $\alpha\in\Delta_{\bar 1}\setminus\mathbb{Q}\Delta_{\bar 0}$, then $\mathbb{G}_{\alpha}=\mathrm{R}(\mathbb{G}_{\alpha})$.	
	\end{rem}
	It is long known that centralizers of tori in reductive algebraic groups are reductive as well (cf. \cite[Corollary 17.59]{milne}). We believe that this statement remains true in the category of algebraic supergroups.
	\begin{hyp}\label{hyp about reductivness}
		Let $\mathbb{G}$ be a quasi-reductive and reductive supergroup, and let $S$ be a torus in $\mathbb{G}_{ev}=G$. Then $\mathrm{Cent}_{\mathbb{G}}(S)$ is a reductive supergroup also. 	
	\end{hyp} 
	As we have already noted, $\mathrm{R}_u(\mathbb{G}_{\alpha})\leq \mathrm{R}_u(\mathrm{R}(\mathbb{G}_{\alpha}))$. Thus immediately follows that under conditions of Proposition \ref{if alpha is not multiple of odd} the supergroup $\mathbb{G}_{\alpha}$ is reductive if and only if $\mathrm{R}(\mathbb{G}_{\alpha})$ is if and only if $K(\mathfrak{G}_{\bar 1}^T)=0$. Similarly, under conditions of Proposition \ref{GL_2}(1) the supergroup $\mathbb{G}_{\alpha}$ is reductive if and only if $\mathrm{R}(\mathbb{G}_{\alpha})$ is. In the remaining cases the reductivity of $\mathbb{G}_{\alpha}$ does not necessarily entail the reductivity of its solvable radical (but the converse obviously takes place). 
	\begin{pr}\label{about K(g_1^T)}
		We have $[K(\mathfrak{G}_{\bar 1}^T), \mathfrak{G}_{\bar 1}]=0$, or equivalently, $K(\mathfrak{G}_{\bar 1}^T)=K(\mathfrak{G}_{\bar 1})\cap\mathfrak{G}_{\bar 1}^T=K(\mathfrak{G}_{\bar 1})^T$. In particular, if $\mathbb{G}$ is reductive, then $K(\mathfrak{G}_{\bar 1}^T)=0$ and the Cartan supergroup $\mathbb{T}$ is reductive as well.	
	\end{pr}
	\begin{proof}
		Let $\alpha\in\Delta$ be a root, such that $[K(\mathfrak{G}_{\bar 1}^T), \mathfrak{G}^{\alpha}_{\bar 1}]\neq 0$. Since 	$[K(\mathfrak{G}_{\bar 1}^T), \mathfrak{G}^{\alpha}_{\bar 1}]\subseteq\mathfrak{G}_{\bar 0}^{\alpha}$, it implies $\alpha\in\Delta_{\bar 0}\cap\Delta_{\bar 1}$. If the space $\mathfrak{G}_{\bar 1}^T$ is not entirely contained in the kernel of
		the Harish-Chandra pair $(\overline{G_{\alpha}}, \mathfrak{G}_{\bar 1}^{T_{\alpha}})$, then $\mathbb{G}_{\alpha}/\mathrm{R}_u(\mathbb{G}_{\alpha})$ is isomorphic either to $\mathbb{H}(0\oplus 2), \mathbb{H}(0\oplus 2)/\mu_2$, or to one of the supergroups from Proposition \ref{if [g, g] is sl_2}, Proposition \ref{GL_2 and k=0} or Proposition \ref{if R intersects V but not contained in}. 
		
		If one of the first three cases takes place, then 
		$W=\mathfrak{G}_{\bar 1}^{T_{\alpha}}/K'(\mathfrak{G}_{\bar 1}^{T_{\alpha}})$ is isomorphic either to $L(0)\oplus L(2)$ or to $L((0, 0))\oplus L((1, -1))$, where both second terms are identified with $\mathfrak{sl}_2$. Note that  \[W^T=(\mathfrak{G}_{\bar 1}^T + K'(\mathfrak{G}_{\bar 1}^{T_{\alpha}}))/K'(\mathfrak{G}_{\bar 1}^{T_{\alpha}})=\Bbbk\oplus \Bbbk H,\] and obviously, $K(W^T)=0$. The latter infers $K(\mathfrak{G}_{\bar 1}^T)\subseteq K'(\mathfrak{G}_{\bar 1}^{T_{\alpha}})$, hence \[[K(\mathfrak{G}_{\bar 1}^T), \mathfrak{G}^{\alpha}_{\bar 1}]\subseteq \mathrm{Lie}(T')\cap\mathfrak{G}_{\bar 0}^{\alpha}=0,\] which contradicts our assumption. 

		In the fourth case $W\simeq\mathfrak{gl}_2\oplus U$ and $W^T=(\Bbbk I_2+\Bbbk H)\oplus U^T$. If \[v=(\alpha I_2+\beta H)\oplus u\in (K(\mathfrak{G}_{\bar 1}^T)+K(\mathfrak{G}_{\bar 1}^{T_{\alpha}}))/K(\mathfrak{G}_{\bar 1}^{T_{\alpha}}), u\in U^T,\]
		then $[v, I_2]=\alpha I_2+\beta H=0$ implies $v\in U^T$. On the other hand, $[U, \mathfrak{gl}_2]=0$ and $[K(\mathfrak{G}_{\bar 1}^T), \mathfrak{G}^{\alpha}_{\bar 1}]=\mathfrak{G}_{\bar 0}^{\alpha}\neq 0$ entail $v=0$, hence the same contradictory statement $K(\mathfrak{G}_{\bar 1}^T)\subseteq K'(\mathfrak{G}_{\bar 1}^{T_{\alpha}})$.
		
		Finally, in the fifth case, $W\simeq \mathfrak{sl}_2\oplus L$ and $L$ contains a submodule $U$, such that $[\mathfrak{sl}_2, U]=0$ and $L/U\simeq\Bbbk$. 
		As above, $W^T=\Bbbk H\oplus L^T$ and let 
		\[v=\alpha H+l\in (K(\mathfrak{G}_{\bar 1}^T)+K(\mathfrak{G}_{\bar 1}^{T_{\alpha}}))/K(\mathfrak{G}_{\bar 1}^{T_{\alpha}})\setminus 0, l\in L^T .\]
		Then $[v, H]=\alpha<H, H>I_2+\bar {l}H=0$ and $<H, H>\neq 0$ imply $v=l\in U$, which again contradicts $[K(\mathfrak{G}_{\bar 1}^T), \mathfrak{G}^{\alpha}_{\bar 1}]=\mathfrak{G}_{\bar 0}^{\alpha}\neq 0$.
	\end{proof}
	The problem of describing centralizers of tori can be modified into the problem of describing centralizers of super-tori in the case where the Cartan supergroup is not purely even. Observe that for any $\alpha\in\Delta$, $T_{\alpha}$ can be defined as $\ker(T\stackrel{\mathrm{Ad}}{\to}\mathrm{GL}(\mathfrak{G}^{\alpha}))^0$. Note also that $\mathfrak{G}^{\alpha}$ is a $\mathbb{T}$-supersubmodule of $\mathfrak{G}$. In fact, the supersubgroup $\mathrm{Stab}_{\mathbb{G}}(\mathfrak{G}^{\alpha})$ contains $T$ and its Lie superalgebra contains $\mathfrak{G}_{\bar 1}^T$.
	
	Following this idea, we set \[\mathbb{T}_{\alpha}=\ker(\mathbb{T}\stackrel{\mathrm{Ad}}{\to}\mathrm{GL}(\mathfrak{G}^{\alpha}))^0, \alpha\in\Delta .\]
	\begin{lm}
		Let $\alpha\in\Delta$. Then the Harish-Chandra pair of $\mathbb{T}_{\alpha}$ is isomorphic to $(T_{\alpha}, \{v\in \mathfrak{G}_{\bar 1}^T\mid [v, \mathfrak{G}^{\alpha}]=0\})$.	
	\end{lm}
	\begin{proof}
		It is clear that $\ker(\mathbb{T}\stackrel{\mathrm{Ad}}{\to}\mathrm{GL}(\mathfrak{G}^{\alpha}))_{ev}=\ker(\alpha)$, hence $(\mathbb{T}_{\alpha})_{ev}=T_{\alpha}$. It remains to note that
		\[\mathrm{Lie}(\ker(\mathbb{T}\stackrel{\mathrm{Ad}}{\to}\mathrm{GL}(\mathfrak{G}^{\alpha})))_{\bar 1}=\ker(\mathrm{Lie}(\mathbb{T})\stackrel{\mathrm{ad}}{\to}\mathfrak{gl}(\mathfrak{G}^{\alpha}))_{\bar 1}=  \{v\in \mathfrak{G}_{\bar 1}^T\mid [v, \mathfrak{G}^{\alpha}]=0\}.\]
	\end{proof}
	Applying \cite[Theorem 6.6(2)]{mas-shib} again, we obtain that the Harish-Chandra pair of $\widehat{\mathbb{G}}_{\alpha}=\mathrm{Cent}_{\mathbb{G}}(\mathbb{T}_{\alpha})$ has a form
	\[(\mathrm{Cent}_{G_{\alpha}}(W_{\alpha}), \{v\in \mathfrak{G}_{\bar 1}^{T_{\alpha}}\mid [v, W_{\alpha}]=0\}),\] where $W_{\alpha}=\{v\in \mathfrak{G}_{\bar 1}^T\mid [v, \mathfrak{G}^{\alpha}]=0\}$. It is clear that $T\leq\mathrm{Cent}_{G_{\alpha}}(W_{\alpha})\leq G_{\alpha}$. Comparing the weights, we obtain $[\oplus_{\beta\in(\Delta_{\bar 1}\cap\mathbb{Q}\alpha)\setminus\{\pm\alpha\}}\mathfrak{G}^{\beta}_{\bar 1}, \mathfrak{G}_{\bar 1}^T]=0$, that in turn implies
	\[\{v\in \mathfrak{G}_{\bar 1}^{T_{\alpha}}\mid [v, W_{\alpha}]=0\}=(\oplus_{\beta\in\Delta_{\bar 1}\cap\mathbb{Q}\alpha}\mathfrak{G}^{\beta}_{\bar 1})\oplus \{v\in \mathfrak{G}_{\bar 1}^T\mid [v, W_{\alpha}]=0 \}.\]
	Let $\alpha\in\Delta_{\bar 0}\setminus\mathbb{Q}\Delta_{\bar 1}$. By Proposition \ref{if alpha is not multiple of odd}, $\mathfrak{G}_{\bar 1}^{T_{\alpha}}=\mathfrak{G}_{\bar 1}^T$ is
	a trivial $G_{\alpha}$-module. Thus $W_{\alpha}=\mathfrak{G}_{\bar 1}^T$ and the Harish-Chandra pair of $\widehat{\mathbb{G}}_{\alpha}$ is isomorphic to
	$(G_{\alpha}, K(\mathfrak{G}_{\bar 1}^T))$. In particular, $\widehat{\mathbb{G}}_{\alpha}\simeq G_{\alpha}\times (\mathbb{G}_a^-)^{\dim K(\mathfrak{G}_{\bar 1}^T)}$.
	This also shows that, in general, $\widehat{\mathbb{G}}_{\alpha}$ does not contain $\mathbb{T}$. For example, it takes place if $\mathbb{G}$ is reductive and $\mathfrak{G}_{\bar 1}^T\neq 0$.
	
	If $\alpha\in\Delta_{\bar 1}\setminus\mathbb{Q}\Delta_{\bar 0}$, then $\mathrm{Cent}_{G_{\alpha}}(W_{\alpha})=T$. Arguing as in Proposition \ref{if any even is not multiple of odd}, we obtain
	$W_{\alpha}=\mathfrak{G}_{\bar 1}^T$, hence the Harish-Chandra pair of $\widehat{\mathbb{G}}_{\alpha}$ is isomorphic to 
	\[(T, (\oplus_{\beta\in\Delta_{\bar 1}\cap\mathbb{Q}\alpha}\mathfrak{G}^{\beta}_{\bar 1})\oplus K(\mathfrak{G}_{\bar 1}^T)).\]
	Since $[(\oplus_{\beta\in\Delta_{\bar 1}\cap\mathbb{Q}\alpha}\mathfrak{G}^{\beta}_{\bar 1}), \mathfrak{G}_1^T]=0$, we obtain that the subpair $(e, K(\mathfrak{G}_{\bar 1}^T))$
	determines a normal purely odd supersubgroup in $\widehat{\mathbb{G}}_{\alpha}$ (which is trivial, provided $\mathbb{G}$ is reductive). 
	
	However, if $\alpha\in\Delta_{\bar 0}\cap\mathbb{Q}\Delta_{\bar 1}$, we can not say anything definite about the subspace $W_{\alpha}$ and the group $\mathrm{Cent}_{G_{\alpha}}(W_{\alpha})$. 
	Thus, even the formulation of plausible hypotheses about the structure of super-tori centralizers seems premature yet.
	
	Another interesting topic is the extension of the even Weyl group, determined only by a couple $(G, T)$, by \emph{odd reflections}, whatever the latter would mean.
	
	By \cite[Theorem 6.6(1)]{mas-shib}, the normalizers $\mathrm{N}_{\mathbb{G}}(\mathbb{T})$ and $\mathrm{N}_{\mathbb{G}_{\alpha}}(\mathbb{T})$ are represented by the Harish-Chandra pairs $(\mathrm{N}_{G}(T), \mathfrak{G}_{\bar 1}^{T})$ and 
	$(\mathrm{N}_{G_{\alpha}}(T), \mathfrak{G}_{\bar 1}^{T})$, respectively. 
	
	Similarly, the Harish-Chandra pairs of $\mathrm{Cent}_{\mathbb{G}}(\mathbb{T})$ and $\mathrm{Cent}_{\mathbb{G}_{\alpha}}(\mathbb{T})$ are both isomorphic to
	$(T, K(\mathfrak{G}_{\bar 1}^T))$. Thus the \emph{Weyl supergroups} of $\mathbb{G}$ and $\mathbb{G}_{\alpha}$, i.e. the supergroups $W(\mathbb{G}, \mathbb{T})=\mathrm{N}_{\mathbb{G}}(\mathbb{T})/\mathrm{Cent}_{\mathbb{G}}(\mathbb{T})$ and $W(\mathbb{G}_{\alpha}, \mathbb{T})=\mathrm{N}_{\mathbb{G}_{\alpha}}(\mathbb{T})/\mathrm{Cent}_{\mathbb{G}_{\alpha}}(\mathbb{T})$,  are represented by the Harish-Chandra pairs 
	$(W(G, T), \mathfrak{G}_{\bar 1}^T/K(\mathfrak{G}_{\bar 1}^T))$ and $(W(G_{\alpha}, T), \mathfrak{G}_{\bar 1}^T/K(\mathfrak{G}_{\bar 1}^T))$, respectively.
	In other words, both Weyl supergroups are not purely even if and only if $[\mathfrak{G}_{\bar 1}^T, \mathfrak{G}_{\bar 1}^T]\neq 0$.

\end{document}